\documentclass[12pt]{amsart}

\usepackage{enumerate, amsmath, amsthm, amsfonts, amssymb, mathrsfs, graphicx, multicol}
\usepackage[usenames, dvipsnames]{color}
\usepackage[margin=1in]{geometry} 
\usepackage[bookmarks, colorlinks=true, linkcolor=blue, citecolor=blue, urlcolor=blue]{hyperref}
\usepackage{stmaryrd}
\usepackage{mathtools}
\usepackage[all]{xy}

\newcommand{\noarticle}{%
  \removelastskip%
  \vskip.4\baselineskip%
  \par}

\newcounter{article}
\newcommand{\article}[1][-]{%
  \setcounter{equation}{0}%
  \removelastskip%
  \vskip.4\baselineskip%
  \refstepcounter{article}%
  \noindent%
  {\bf (\thearticle)} %
  \ifx-#1{}\else{\it #1.}\fi \phantomsection}

\def\ifempty#1{\if&#1&}

\def\Thmtpl#1#2#3#4{%
  \removelastskip%
  \vskip.6\baselineskip%
  \refstepcounter{article}%
  \noindent%
  {\bf (\thearticle)} %
  \bgroup{}#2{}#1%
  \ifempty{#4}\else{ (#4)}\fi%
  .\egroup%
  #3%
  \ifempty{#4}\ \fi%
  \phantomsection%
}
\def\endthm{\par\vskip.4\baselineskip}

\def\thmtpl#1#2#3#4{%
  \removelastskip%
  \par\vskip.4\baselineskip%
  \addtocounter{article}{-1}
  \refstepcounter{article}
  \noindent%
  \bgroup#2#1%
  \ifempty{#4}\else{ (#4)}\fi%
  .\egroup%
  #3%
  \ifempty{#4}\ \fi%
  \phantomsection%
}

\def\newtheoremenvironment#1#2#3#4#5{%
  \newenvironment{#1}[1][]{\Thmtpl{#3}{#4}{#5}{##1}}{\endthm}%
  \newenvironment{#2}[1][]{\thmtpl{#3}{#4}{#5}{##1}}{\endthm}%
  }

\newtheoremenvironment{Theorem}{theorem}{Theorem}{\bf}{\it}
\newtheoremenvironment{Conjecture}{conjecture}{Conjecture}{\bf}{\it}
\newtheoremenvironment{Proposition}{proposition}{Proposition}{\bf}{\it}
\newtheoremenvironment{Lemma}{lemma}{Lemma}{\bf}{\it}
\newtheoremenvironment{Corollary}{corollary}{Corollary}{\bf}{\it}
\newtheoremenvironment{Example}{example}{Example}{\it}{}
\newtheoremenvironment{Definition}{definition}{Definition}{\it}{}
\newtheoremenvironment{Remark}{remark}{Remark}{\it}{}
\newtheoremenvironment{Question}{question}{Question}{\it}{}

\renewenvironment{proof}[1][Proof]{
  \removelastskip%
  \par\vskip.4\baselineskip\noindent%
  \topsep=0pt%
  \partopsep=0pt%
  \labelsep=0pt%
  \begin{trivlist}%
  \item[\emph{#1.} ]%
  \pushQED{\qed} 
}{
\popQED%
  \end{trivlist}%
  \vskip.4\baselineskip%
}

\renewcommand{\subsection}[1]{%
  \refstepcounter{subsection}%
  \removelastskip%
  \vskip.4\baselineskip%
  \noindent%
  \thesubsection{}. {\bf #1}%
  \leavevmode%
  \setcounter{article}{0}%
  \renewcommand{\thearticle}{\thesubsection.\arabic{article}}%
  \nopagebreak%
  \addcontentsline{toc}{subsection}{\hskip 4.2ex\thesubsection{}. #1}%
  }

\newcommand{\bA}{\mathbf{A}}
\newcommand{\cA}{\mathcal{A}}

\newcommand{\cB}{\mathcal{B}}

\newcommand{\rB}{\mathrm{B}}

\newcommand{\bC}{\mathbf{C}}
\newcommand{\cC}{\mathcal{C}}

\newcommand{\sC}{\mathscr{C}}

\newcommand{\rD}{\mathrm{D}}

\newcommand{\bE}{\mathbf{E}}

\newcommand{\cF}{\mathcal{F}}

\newcommand{\sF}{\mathscr{F}}
\newcommand{\bG}{\mathbf{G}}
\newcommand{\cG}{\mathcal{G}}

\newcommand{\sG}{\mathscr{G}}

\newcommand{\rH}{\mathrm{H}}

\newcommand{\cK}{\mathcal{K}}

\newcommand{\rK}{\mathrm{K}}

\newcommand{\bM}{\mathbf{M}}

\newcommand{\bO}{\mathbf{O}}
\newcommand{\cO}{\mathcal{O}}

\newcommand{\cP}{\mathcal{P}}

\newcommand{\bQ}{\mathbf{Q}}
\newcommand{\cQ}{\mathcal{Q}}

\newcommand{\rR}{\mathrm{R}}

\newcommand{\bS}{\mathbf{S}}
\newcommand{\cS}{\mathcal{S}}
\newcommand{\fS}{\mathfrak{S}}

\newcommand{\bT}{\mathbf{T}}
\newcommand{\cT}{\mathcal{T}}

\newcommand{\cU}{\mathcal{U}}

\newcommand{\bV}{\mathbf{V}}
\newcommand{\cV}{\mathcal{V}}

\newcommand{\bZ}{\mathbf{Z}}

\newcommand{\rf}{\mathrm{f}}

\newcommand{\fg}{\mathfrak{g}}

\newcommand{\bs}{\mathbf{s}}
\newcommand{\fs}{\mathfrak{s}}

\newcommand{\bt}{\mathbf{t}}

\newcommand{\un}{\ul{n}}

\renewcommand{\phi}{\varphi}
\renewcommand{\emptyset}{\varnothing}
\newcommand{\eps}{\varepsilon}

\newcommand{\ol}[1]{\overline{#1}}
\newcommand{\ul}[1]{\underline{#1}}

\newcommand{\arxiv}[1]{\href{http://arxiv.org/abs/#1}{{\tt arXiv:#1}}}

\makeatletter
\def\Ddots{\mathinner{\mkern1mu\raise\p@
\vbox{\kern7\p@\hbox{.}}\mkern2mu
\raise4\p@\hbox{.}\mkern2mu\raise7\p@\hbox{.}\mkern1mu}}
\makeatother

\DeclareMathOperator{\im}{image} 

\DeclareMathOperator{\coker}{coker}
\renewcommand{\hom}{\operatorname{Hom}}

\DeclareMathOperator{\rad}{rad}

\DeclareMathOperator{\ext}{Ext}

\DeclareMathOperator{\End}{End}

\DeclareMathOperator{\Sym}{Sym}

\DeclareMathOperator{\Aut}{Aut}

\DeclareMathOperator{\Tor}{Tor}
\DeclareMathOperator{\colim}{colim}

\DeclareMathOperator{\Spec}{Spec}

\DeclareMathOperator{\res}{res}
\DeclareMathOperator{\sgn}{sgn}

\DeclareMathOperator{\len}{len}
\DeclareMathOperator{\Mod}{Mod}
\DeclareMathOperator{\CoMod}{CoMod}

\newcommand{\GL}{\mathbf{GL}}
\newcommand{\GA}{\mathbf{GA}}
\newcommand{\SL}{\mathbf{SL}}
\newcommand{\Sp}{\mathbf{Sp}}
\newcommand{\SO}{\mathbf{SO}}

\newcommand{\fsl}{\mathfrak{sl}}
\newcommand{\fgl}{\mathfrak{gl}}
\newcommand{\fso}{\mathfrak{so}}

\let\mf\mathfrak

\let\ss\scriptstyle
\let\wt\widetilde
\let\ol\overline
\let\ul\underline
\let\bs\backslash

\let\wh\widehat

\newcommand{\pref}[1]{{\bf (}\ref{#1}{\bf )}}

\newcommand{\stab}{\mathrm{st}}
\newcommand{\Cat}{\mathbf{Cat}}
\newcommand{\TCat}{\mathbf{TenCat}}
\newcommand{\id}{\mathrm{id}}

\newcommand{\op}{\mathrm{op}}
\newcommand{\gfin}{\mathrm{gf}}
\newcommand{\pol}{\mathrm{pol}}
\newcommand{\fin}{\mathrm{f}}

\renewcommand{\Vec}{\mathrm{Vec}}
\renewcommand{\fs}{\mathrm{(fs)}}

\renewcommand{\fg}{\mathrm{fg}}
\newcommand{\db}{(\mathrm{db})}
\newcommand{\ub}{(\mathrm{ub})}
\newcommand{\usb}{(\mathrm{usb})}
\newcommand{\dsb}{(\mathrm{dsb})}
\newcommand{\tors}{\mathrm{tors}}

\newcommand{\dpc}{\mathrm{(dp)}}
\newcommand{\upc}{\mathrm{(up)}}

\newcommand{\ds}{\mathrm{(ds)}}
\newcommand{\us}{\mathrm{(us)}}

\newcommand{\dwb}{\mathrm{(dwb)}}
\newcommand{\uwb}{\mathrm{(uwb)}}

\newcommand{\irr}{\mathrm{irr}}

\newcommand{\hatotimes}{\ \wh{\otimes}\ }
\newcommand{\lw}{{\textstyle \bigwedge}}
\newcommand{\bomega}{\mbox{\boldmath$\omega$}}

\DeclareMathOperator{\Ind}{Ind}
\DeclareMathOperator{\Hom}{Hom}
\DeclareMathOperator{\Rep}{Rep}

\DeclareMathOperator{\Fun}{Fun}
\DeclareMathOperator{\LEx}{LEx}

\DeclareMathOperator{\Ext}{Ext}

\title{Stability patterns in representation theory}

\author{Steven V Sam}
\address{Department of Mathematics, University of California, Berkeley, CA}
\email{\href{mailto:svs@math.berkeley.edu}{svs@math.berkeley.edu}}
\urladdr{\url{http://math.berkeley.edu/~svs/}}

\author{Andrew Snowden}
\address{Department of Mathematics, University of Michigan, Ann Arbor, MI}
\email{\href{mailto:asnowden@umich.edu}{asnowden@umich.edu}}
\urladdr{\url{http://www-personal.umich.edu/~asnowden/}}

\date{May 12, 2015}

\thanks{S.~Sam was supported by an NDSEG fellowship and a Miller research fellowship.  A.~Snowden was supported by NSF fellowship DMS-0902661.}


\begin{document}

\maketitle

\begin{abstract}
We develop a comprehensive theory of the stable representation categories of several sequences of groups, including the classical and symmetric groups, and their relation to the unstable categories.  An important component of this theory is an array of equivalences between the stable representation category and various other categories, each of which has its own flavor (representation theoretic, combinatorial, commutative algebraic, or categorical) and offers a distinct perspective on the stable category.  We use this theory to produce a host of specific results, e.g., the construction of injective resolutions of simple objects, duality between the orthogonal and symplectic theories, a canonical derived auto-equivalence of the general linear theory, etc.
\end{abstract}

\tableofcontents

\section{Introduction}

\subsection{Overview}

\article[Stable representation theory]
Let $(G_d)_{d \ge 1}$ be a sequence of groups.  Suppose we have some notion of compatibility for representations of different $G$'s.  For instance, it could be that the irreducibles of each $G_d$ are parametrized in some common manner, and representations of different $G$'s are considered compatible if they have matching irreducible constituents.  (This definition is too weak in practice, but a good first approximation.)  The {\bf stable representation theory} of $G$ is the theory of compatible sequences $(V_d)_{d \ge 1}$ of representations where we care only about what happens for $d$ large.  There are two natural questions to consider:
\begin{itemize}
\item[A.] What is the structure of the stable representation theory of $G$?  To be more precise, we introduce the {\bf stable representation category} $\Rep^{\stab}(G_{\ast})$: informally, its objects are compatible sequences of representations, with two such sequences considered equivalent if they are the same for $d$ large; formally, it can be described as a Serre quotient.  The present question can then be rephrased as:  what is the structure of $\Rep^{\stab}(G_{\ast})$?  This abstract question naturally suggests many concrete ones:  What are the simple (or projective or injective) objects?  What are the projective (or injective) resolutions of simple objects?  How does the tensor product of two simple objects decompose?  Etc.
\item[B.] How does the stable representation theory relate to the representation theory of each $G_d$?  In cases of interest, we construct a {\bf specialization functor} $\Gamma_d \colon \Rep^{\stab}(G_{\ast}) \to \Rep(G_d)$, and the question can be rephrased as:  what is the structure of $\Gamma_d$?  Again, this leads to concrete questions:  What are the exactness (or tensorial) properties of $\Gamma_d$?  What does $\Gamma_d$ (or its derived functors) do to simple (or projective or injective) objects?  Etc.
\end{itemize}

\article[Example: the polynomial theory of the general linear group]
\label{intro:pol}
A complex representation $\rho$ of $\GL(d)$ is {\bf polynomial} if the entries of $\rho(g)$ are polynomial functions of the matrix entries of $g \in \GL(d)$.  The irreducible polynomial representations of $\GL(d)$ are parametrized by partitions of length at most $d$.  A partition of length at most $d$ is also a partition of length at most $d+1$, and this gives rise (roughly) to the notion of compatibility.

The stable theory of these representations is well-understood (see \S \ref{ss:schurfunctors} for a review).  An object of the stable category can be thought of in several different ways:  as a polynomial representation of $\GL(\infty)$, as a formal sum of representations of symmetric groups, or as a Schur functor.  Furthermore, the stable category is distinguished as the universal tensor category.  Each of these descriptions is valuable, and offers its own perspective on the stable theory.  Using them, one can easily give a complete answer to Question~A: the stable category is semi-simple, and the simple objects are naturally in bijection with partitions.

The specialization functor $\Gamma_d$ from the stable category to $\Rep(\GL(d))$ is most easily seen from the perspective of Schur functors:  it is then simply evaluation on $\bC^d$.  As is well-known, $\bS_{\lambda}(\bC^d)$ is the irreducible of $\GL(d)$ associated to $\lambda$ if $\lambda$ has at most $d$ parts, and 0 otherwise.  Since the stable category is semi-simple, this completely describes $\Gamma_d$, and answers Question~B.

\article[Purpose of this paper]
In this paper, we study the stable representation theory of five families of groups:  the three families of classical groups $\GL(d)$ (here we are interested in all rational representations, not just the polynomial ones), $\bO(d)$ and $\Sp(d)$, the symmetric groups $\fS(d)$ and the (non-reductive) general affine groups $\GA(d)$.  We work exclusively with complex representations.  The representation theory of these groups is so ubiquitous that our objective hardly requires motivation; nonetheless, let us provide some:

\begin{itemize}
\item One might hope to obtain a better understanding of problems at finite level from the stable theory.  In fact, this strategy has already been carried out, for at least one problem, by Koike and Terada \cite{koike}, \cite{koiketerada}: They gave good answers to both questions for classical groups at the level of characters.  Furthermore, they were able to understand the stable decomposition of tensor products of simple modules.  Thus they were able to understand the decomposition of tensor products at finite level (by combining the stable result with their answer to Question~B), solving a basic problem in representation theory.

\item One can reasonably expect the stable categories to relate to other parts of representation theory.  Indeed, this turns out to be the case:  we will see that these categories satisfy elegant universal properties and are closely related to Deligne's interpolation categories \cite{deligne}.

\item Many examples of compatible sequences of representations of $\fS(d)$ occur in nature, e.g., in the study of configuration spaces. More examples are listed by Church, Ellenberg, and Farb in \cite{fimodules} (see also \cite{churchfarb}), where they are called FI-modules. Due to such examples, it is fair to say that compatible sequences are interesting in their own right.
\end{itemize}

\article[Results of this paper]
Our main results give a thorough answer to Question~A, for each of the five families of groups under consideration.  Our ``answer'' consists of a collection of equivalences between $\Rep^{\stab}(G_{\ast})$ and several other categories, analogous to the picture sketched in \pref{intro:pol}.  Each of these categories has its own flavor --- representation theoretic, combinatorial, commutative algebraic, or categorical --- and thus offers a distinct lens through which any given question about $\Rep^{\stab}(G_{\ast})$ can be viewed.  From our perspective, homological questions are typically best attacked in the commutative algebra category, where tools like the Koszul complex and classical invariant theory are available.  (The commutative algebra description of stable categories is, in our opinion, the most novel aspect of this paper, and makes essential use of objects called twisted commutative algebras.  It has no interesting analogue in the polynomial theory of $\GL$.)  We answer all of the concrete questions mentioned in the general discussion of Question~A, as well as many more.

We also give a thorough answer to Question~B.  However, most of the real work on this --- computing the derived functors of the specialization functor on simple objects --- was done elsewhere (\cite{ssw} for the classical groups and \cite{symc1} for the symmetric group).  In this paper, we develop the basic theory of the specialization functor, and explain how the cited results (which use a different language) can be rephrased using it.

We emphasize that there is no single result that can be pointed to as the main result:  the final product of this paper is a theory which describes stable representation theory and its relation to representation theory at finite level.

\article
Unfortunately, and despite the uniformity of our results, each of the five families is treated separately.  This accounts for the bulk of the paper.  It also accounts for its title:  each case fits nearly the same pattern, but we have no unifying theory.  In fact, there are additional families which fit these patterns as well --- such as wreath products of symmetric groups with finite groups and generalizations of the general affine group --- and probably more still to be discovered.  It would therefore be of great interest to find a general theory, if for no other reason than to know exactly how far the phenomena observed here extend.

\subsection{Descriptions of stable representation theory}
\label{intro:cats}

\noarticle
As mentioned, the main results of this paper establish equivalences between stable representation categories and various other categories.  We now describe these categories and some of the constructions which yield the equivalences.

\article[Infinite rank groups]
In the families of groups under consideration, $G_d$ is naturally a subgroup of $G_{d+1}$.  We can therefore form the limit group $G=\bigcup_{d \ge 1} G_d$.  For example, when $G_d=\GL(d)$, the group $G=\GL(\infty)$ consists of invertible infinite matrices which differ from the identity matrix at only finitely many entries.  In any compatible sequence $(V_d)$ of representations, we have a natural inclusion $V_d \subset V_{d+1}$.  The limit $V=\bigcup_{d \ge 1} V_d$ is therefore a representation of $G$.  It is clear that $V$ is a stable invariant of $(V_d)$, that is, it only depends on $V_d$ for $d$ large.  In fact, no information is lost by passing from $(V_d)$ to $V$.  It therefore suffices to study the representations of $G$ which arise from compatible sequences.

Fortunately, this class of representations is easily distinguished.  Let $\bV=\bC^{\infty}=\bigcup_{d \ge 1} \bC^d$ be the standard representation of $\GL(\infty)$ and let $\bV_*=\bigcup_{d \ge 1} \bC^d{}^*$ be its restricted dual.  In each of the five cases under consideration, $G$ is naturally a subgroup of $\GL(\infty)$, and so $\bV$ and $\bV_*$ are representations of $G$.  We say that a representation of $G$ is {\bf algebraic} (resp.\ {\bf polynomial}) if it appears as a constituent of a finite direct sum of tensor powers of $\bV$ and $\bV_*$ (resp.\ tensor powers of $\bV$).  We write $\Rep(G)$ (resp.\ $\Rep^{\pol}(G)$) for the category of algebraic (resp.\ polynomial) representations of $G$.  The algebraic representations are those which come from compatible sequences, but in certain cases we will want to restrict attention to polynomial representations.

We therefore have an equivalence $\Rep^{\stab}(G_{\ast})=\Rep(G)$.  In fact, we never precisely define $\Rep^{\stab}(G_{\ast})$, and so this equivalence is by fiat.  Although $\Rep^{\stab}(G_{\ast})$ is never formally employed, it is the motivation behind everything that we do, and a constant source of intuition.

\begin{remark}
When $\bV$ and $\bV_*$ are isomorphic as representations of $G$, as is the case for the orthogonal, symplectic, and symmetric groups, there is no distinction between algebraic and polynomial representations.  Furthermore, algebraic representations of the general affine group are of a somewhat different nature from the other representations studied in this paper, and are therefore not considered.  Hence, there are six classes of representations under consideration:  polynomial representations of $\GL$ and $\GA$ and algebraic representations of $\GL$, $\bO$, $\Sp$ and $\fS$.  The polynomial representations of $\GL$ have been well-understood for some time, and were briefly discussed in \pref{intro:pol}.
\end{remark}

\article[(Lack of) semi-simplicity]
\label{intro:ss}
Before continuing, we highlight an important feature of the categories $\Rep(G)$:  they are not semi-simple.  (See below for an example.)  This is in stark contrast to the case of $\Rep^{\pol}(\GL)$, and is why these categories  are more complicated than it.  In particular, it implies that character theory does not capture the whole picture.

\begin{example}
Consider the pairing $\bV \otimes \bV_* \to \bC$, which defines a surjection in $\Rep(\GL)$.  We claim that it is not split.  To see this, it is enough to show that $\bV \otimes \bV_*$ has no invariants.  Think of $\bV \otimes \bV_*$ as endomorphisms of $\bV$ which kill all but finitely many basis vectors; the group $\GL(\infty)$ acts on this space by conjugation.  Any endomorphism commuting with $\GL(\infty)$ is a scalar matrix.  Since scalar matrices do not belong to $\bV \otimes \bV_*$, there are no invariants.

In fact, one can see this lack of semi-simplicity in $\Rep^{\stab}(\GL(\ast))$, without passing to the infinite group.  The point is that the decomposition of $\bC^d \otimes \bC^d{}^*$ as a representation of $\GL(d)$ is not stable:  the diagram
\begin{displaymath}
\xymatrix{
\bC \ar[r] \ar@{=}[d] & \bC^d \otimes \bC^d{}^* \ar[d] \\
\bC \ar[r] & \bC^{d+1} \otimes (\bC^{d+1})^* }
\end{displaymath}
does not commute, where the horizontal maps are the inclusions of the trivial isotypic pieces and the right vertical map comes from standard inclusions.  (Note, however, that the obvious diagram with the horizontal arrows reversed does commute.  This is why $\bC$ is a quotient of $\bV \otimes \bV_*$ in $\Rep(\GL)$, but not a subobject.)
\end{example}

\article[Diagram algebras] \label{art:diagram-algebras}
There is an obvious action of the symmetric group $S_n$ on $(\bC^d)^{\otimes n}$ which commutes with the action of $\GL(d)$.  Schur--Weyl duality states that the image of $\bC[S_n]$ in $\End((\bC^d)^{\otimes n})$ is the full centralizer of $\GL(d)$.  Furthermore, it provides a decomposition
\begin{displaymath}
(\bC^d)^{\otimes n} = \bigoplus_{\vert \lambda \vert=n,\ \ell(\lambda) \le d} V_{\lambda} \otimes \bM_{\lambda},
\end{displaymath}
as a representation of $\GL(d) \times S_n$.  Here the sum is over partitions $\lambda$ of $n$ with at most $d$ rows, $V_{\lambda}$ is the irreducible of $\GL(d)$ with highest weight $\lambda$ and $\bM_{\lambda}$ is the irreducible of $S_n$ corresponding to $\lambda$.  This decomposition stabilizes for $d \ge n$, and provides a natural bijection between the irreducible representations of $S_n$ and the irreducible polynomial representations of $\GL(d)$ in which the center acts through the $n$th power character.  Letting $d=\infty$, we obtain a bijection between irreducible representations of symmetric groups and irreducible polynomial representations of $\GL(\infty)$.  In fact, we obtain an equivalence between the category $\Rep(S_{\ast})$ of sequences $(M_n)_{n \ge 0}$, where $M_n$ is a representation of $S_n$, and the category $\Rep^{\pol}(\GL)$.  This provides a combinatorial description of the stable polynomial representation theory of $\GL$.

Our first piece of real progress is to generalize the above picture to the other theories under consideration.  For expository purposes, we restrict ourselves to the orthogonal group here.  The first step in carrying out this generalization is to understand the centralizer of the $\bO(d)$ action on $(\bC^d)^{\otimes n}$.  Fortunately, this is well-known:  it is the Brauer algebra $\cB_n(d)$.  This algebra is generated by three obvious operations:
\begin{itemize}
\item The transposition $\sigma_{i,j}$ of the $i$th and $j$th tensor factors.
\item The contraction $c_{i,j} \colon (\bC^d)^{\otimes n} \to (\bC^d)^{\otimes (n-2)}$, which applies the pairing to the $i$th and $j$th tensor factors.
\item The co-contraction $c_{i,j}^* \colon (\bC^d)^{\otimes (n-2)} \to (\bC^d)^{\otimes n}$, which inserts a copy of the invariant in $(\bC^2)^{\otimes 2}$ in the $i$th and $j$th tensor factors.
\end{itemize}
Of course, contraction and co-contraction are not endomorphisms of $(\bC^d)^{\otimes n}$.  However, one can build non-trivial endomorphisms of $(\bC^d)^{\otimes n}$ using them, e.g., $c^*_{1,3} c^*_{2,4} c_{1,2} c_{3,4}$.  These are the endomorphisms which generate the Brauer algebra.

We would now like to take $d=\infty$.  However, there is a problem:  $\bV^{\otimes 2}$ has no invariant (this is similar to the example in \pref{intro:ss}), and so there is no co-contraction.  Our solution to this problem is to simply discard co-contraction.  We thus consider the algebra $\cB$ generated by transpositions and contractions, which acts on the space $K=\bigoplus_{n \ge 0} \bV^{\otimes n}$ and commutes with the action of $\bO(\infty)$. Given a $\cB$-module $M$, the space $\Hom_{\cB}(M, K)$ is naturally a representation of $\bO(\infty)$.  We show that this construction defines a contravariant equivalence between a certain category of $\cB$-modules and $\Rep(\bO)$.

Actually, we find it more convenient to state the above result in a different language.  Define the {\bf downwards Brauer category}, denote $\db$, as follows:
\begin{itemize}
\item The objects are finite sets.
\item A morphism $L \to L'$ consists of a matching $\Gamma$ on $L$ and a bijection $L \setminus V(\Gamma) \to L'$, where $V(\Gamma)$ denotes the set of vertices of $\Gamma$.
\end{itemize}
The terminology ``downwards'' refers to the fact that morphisms cannot go from smaller sets to larger ones:  if $L \to L'$ is a morphism then $\#L' \le \#L$.  A {\bf representation} of $\db$ is a functor $\db \to \Vec$.  Representations of $\db$ are closely related to $\cB$-modules.  For instance, the $\cB$-module $K$ corresponds to the functor $\cK$ which attaches to a finite set $L$ the space $\cK_L=\bV^{\otimes L}$; a morphism $L \to L'$ in $\db$ induces a morphism $\cK_L \to \cK_{L'}$ via contractions and permutations. The equivalence mentioned above can be rephrased in this new language as:  the category of finite length representations of $\db$ is contravariantly equivalent to $\Rep(\bO)$.  We obtain a covariant equivalence by using the {\bf upwards Brauer category} instead.  This is our combinatorial description of the stable representation theory of the orthogonal group.

\begin{remark} \label{rmk:symgpdirect}
With the exception of the symmetric group, the combinatorial description of the other categories is very similar to the above.  For the symmetric group, the category that replaces the downwards Brauer category is the downwards partition category (which relates to partition algebras).  It is exceptional in that it is not {\bf weakly directed}:  the maps do not all go in one direction.  This greatly complicates the analysis of its representation category.
\end{remark}

\article[Twisted commutative algebras] \label{art:description-tca}
Suppose that $M$ is a representation of the downwards Brauer category $\db$.  Evaluating $M$ on the finite set $\ul{n}=\{1, \ldots, n\}$, we obtain a representation $M_{\ul{n}}$ of $\Aut_{\db}(\ul{n})=S_n$.  We can apply Schur--Weyl duality to this representation to obtain a polynomial representation $V_n$ of $\GL(\infty)$.  There is a map $\alpha_n \colon \ul{n} \to \ul{n+2}$ in $\db$, given by a graph with a single edge (in fact, there are several such maps).  A careful examination of how $\alpha_n$ interacts with Schur--Weyl duality shows that it corresponds to a map $\beta_n \colon \Sym^2(\bV) \otimes V_n \to V_{n+2}$ of representations of $\GL(\infty)$.  Furthermore, the relations between the $\alpha_n$ for various $n$ translate exactly to the $\beta_n$ defining the structure of a $\Sym(\Sym^2(\bV))$-module on $\bigoplus_{n \ge 0} V_n$.

A {\bf twisted commutative algebra} (tca) is a commutative associative unital algebra endowed with an action of $\GL(\infty)$ by algebra homomorphisms, under which it constitutes a polynomial representation.  A {\bf module} over a tca $A$ is an $A$-module $M$, in the usual sense, endowed with a compatible action of $\GL(\infty)$, under which it too forms a polynomial representation.  The previous paragraph can be rephrased in this language as follows:  the category $\Rep(\bO)$ is equivalent to the category of finite length modules over the tca $\Sym(\Sym^2(\bV))$.  This is our commutative algebraic description of the stable representation theory of the orthogonal group.

\begin{remark}
The module category over a non-trivial tca is necessarily not semi-simple.  Therefore the above description of representations of $\bO(\infty)$ is truly specific to the infinite case:  there is no similar description of the representation category of $\bO(d)$.  The module category over the trivial tca $\bC$ is, by definition, the semi-simple category $\Rep^{\pol}(\GL)$.
\end{remark}

\article[Generalizations of Schur functors] \label{art:generalize-schur}
The category of all functors $\Vec \to \Vec$ is an abelian category.  An object of this category is called {\bf polynomial} if it appears as a subquotient of a direct sum of objects of the form $V \mapsto V^{\otimes n}$.  Let $\cS$ be the full subcategory of polynomial functors; this is the {\bf Schur algebra}.  The theory of Schur functors provides an equivalence between $\cS$ and $\Rep^{\pol}(\GL)$.

We present an analogous theory for the other cases under consideration.  For now, we consider only the orthogonal case.  Let $T_0$ be the category of pairs $(V, \omega)$, where $V$ is a finite dimensional vector space and $\omega$ is a symmetric bilinear form on $V$.  (We have not placed any non-degeneracy conditions on $\omega$, but one can do this without changing what follows.)  We associate to every object of $\Rep(\bO)$ a natural functor $T_0 \to \Vec$, which we call an {\bf orthogonal Schur functor}.  We show that the resulting functor $\Rep(\bO) \to \Fun(T_0, \Vec)$ is fully faithful.  Thus the stable representation theory of the orthogonal group can be interpreted in terms of orthogonal Schur functors.

\article[Universal descriptions]
The category $\Rep^{\pol}(\GL)$ is distinguished by an elegant universal property:  it is the universal abelian tensor category.  Precisely, to give a tensor functor from $\Rep^{\pol}(\GL)$ to some abelian tensor category $\cA$ is the same as to give an object of $\cA$ (see \pref{art:schur-univ} for details).  We show that the other categories under consideration satisfy similar universal properties.  For instance, to give a left-exact tensor functor from $\Rep(\bO)$ to some tensor category $\cA$ is the same as giving a pair $(A, \omega)$, where $A$ is an object of $\cA$ and $\omega$ is a symmetric bilinear form on $A$.

\subsection{Additional results, applications, and remarks}

\noarticle
We apply the descriptions of the stable representation categories discussed in the previous section to prove an array of other results.  We now describe some of them, and make some additional comments.

\article[Structural results]
We obtain many structural results in each of the five cases:  classification of simple, injective, and projective objects, blocks of the category, minimal resolutions of simple objects and computation of the $\Ext$ groups between simple objects.  Some of these results are contained in \cite{koszulcategory}.  In many cases, however, our proofs are more natural, due to the additional tools at our disposal.  For instance, \cite{koszulcategory} calculates the $\Ext$ groups between simples for $\GL$, but not in other cases.  We give an easy calculation in all cases using the tca viewpoint.

\article[The specialization functor] \label{art:specializationfunctor}
As stated, one of the main aims of our theory is to relate stable representation theory to representation theory at finite level.  The link between the two is provided by the {\bf specialization functor}.  We discuss the orthogonal case here, the others being similar.  The standard representation $\bC^d$ of $\bO(d)$ admits, by definition, an invariant symmetric form.  Thus, by the universal property of $\Rep(\bO)$, there is a corresponding left-exact tensor functor
\begin{displaymath}
\Gamma_d \colon \Rep(\bO) \to \Rep(\bO(d)).
\end{displaymath}
This is the specialization functor.  We give a much more concrete description of this functor in terms of the representation theory:  $\Gamma_d(V)$ is the space of invariants of $V$ under a certain subgroup $H_d$ of $\bO(\infty)$.  However, by far the most interesting result on $\Gamma_d$ comes from \cite{ssw}:  if $V$ is a simple object of $\Rep(\bO)$ then $\rR^i \Gamma_d(V)$ either vanishes for all $i$ or is non-zero for at most one $i$, and is then an irreducible representation of $\bO(d)$.  Furthermore, there is a combinatorial rule,  reminiscent of the Borel--Weil--Bott theorem, which gives the index of non-vanishing (if it exists) and the resulting irreducible of $\bO(d)$.  This had previously been proved at the level of Euler characteristic by Koike--Terada \cite[\S 2.4]{koiketerada}.  See \pref{br:example} for a specific example of how this theory can be applied to problems at finite level.

\article[Orthogonal--symplectic duality]
We show that the stable representation theory of the orthogonal group is dual to that of the symplectic group, that is, there is a natural equivalence of categories $\Rep(\bO) \cong \Rep(\Sp)$.  This equivalence is an {\bf asymmetric} tensor functor:  it commutes with tensor products but does not respect the commutativity isomorphism of the tensor product.  The proof is short enough to recapitulate here:  transposition of partitions defines an asymmetric auto-equivalence of $\Rep^{\pol}(\GL)$ which interchanges the tca's $\Sym(\Sym^2(\bV))$ and $\Sym(\lw^2(\bV))$; it therefore induces an equivalence between the module categories, the first of which is $\Rep(\bO)$, the second $\Rep(\Sp)$.

This result was first proved (to the best of our knowledge) at the level of representation rings by Koike--Terada \cite[Theorem 2.3.2]{koiketerada}.  Later, in \cite[Corollary~6.11]{koszulcategory}, the equivalence was established at the level of abelian categories (ignoring the tensor structure).  Our result is the common generalization of the two. Serganova has also obtained this result by making use of the infinite orthosymplectic Lie superalgebra \cite[\S 4.3]{serganova-inf}.

\article[The groups $\GA$ and $\fS$] \label{art:intro-degeneration}
It follows from our descriptions of the categories $\Rep^{\pol}(\GA)$ and $\Rep(\fS)$ that they are equivalent as abelian categories. In fact, we realize this equivalence as an ``infinite Schur--Weyl duality,'' see \pref{cor:sym-ga}. However, the two categories are not equivalent as tensor categories:  the structure constants for tensor products of simple objects are given by the Littlewood--Richardson coefficients in the former and the stable Kronecker coefficients in the latter.  Thus $\Rep^{\pol}(\GA)$, as a tensor category, can be regarded as a degeneration of $\Rep(\fS)$ (see \pref{ques:degen} for details).  The category $\Rep(\fS)$ seems to be the most natural categorical home of the stable Kronecker coefficients.  

\article[The Fourier transform and Koszul duality]
It follows from our results that each of the categories under consideration is Koszul.  The category $\Rep^{\pol}(\GA) \cong \Rep(\fS)$, is Koszul self-dual: this was established in \cite{symc1}, where we constructed a canonical auto-equivalence of the derived category, called the {\bf Fourier transform}, realizing the auto-duality.  Here, we extend this construction to $\Rep(\GL)$, showing that it is its own dual.  The Fourier transform now involves a small choice:  it is only canonical up to a $\bZ/2$ ambiguity.  We also show that $\Rep(\bO)$ and $\Rep(\Sp)$ are neither self-dual nor dual to each other; the dual categories are just different things.

The Koszul properties of $\Rep(\GL)$, $\Rep(\bO)$, and $\Rep(\Sp)$ were investigated in \cite{koszulcategory}.  It was shown that these categories are Koszul, and that the first is self-dual.  However, \cite{koszulcategory} does not construct the nearly canonical auto-duality of $\Rep(\GL)$ that we do:  the proof of self-duality in loc.~cit.\ is through an explicit computation with quadratic rings.

\article[Tensor product and branching rules]
We determine the multiplicity of a simple object in the tensor product of two other simple objects in the categories $\Rep(\GL)$ and $\Rep(\bO) \cong \Rep(\Sp)$.  These results had previously been obtained by Koike \cite{koike}.  However, the proof of loc.~cit.\ was by character calculations.  Our proof is more conceptual, using a general principle to reduce the problem to a simple exercise on the symmetric group.  Combined with results on the specialization functor, this recovers tensor product formulas at finite level, also obtained by Koike.

One might hope to apply the same method to the symmetric groups and obtain formulas for the stable Kronecker coefficients. Unfortunately, we have not been able to do this.

A branching rule describes how an irreducible representation of a group decomposes under a subgroup.  There are several interesting embeddings of classical groups, each giving rise to a branching rule.  For example, tensor product decompositions are branching rules for diagonal embeddings.  We discuss these rules in \S\ref{sec:relations}, and show how our viewpoint can be used to rederive some of them.

\article[Odd symplectic groups]
As discussed in \pref{art:description-tca}, the stable representation categories that we consider are equivalent to categories of finite length modules over certain tca's. Precisely, the categories $\Rep^{\pol}(\GL)$, $\Rep^{\pol}(\GA)$, $\Rep(\bO)$, $\Rep(\Sp)$, $\Rep(\GL)$ correspond to the tca's $\bC$, $\Sym(\bV)$, $\Sym(\Sym^2(\bV))$, $\Sym(\lw^2(\bV))$, and $\Sym(\bV \otimes \bV')$ (here $\bV'$ denotes a separate copy of $\bV$, and the last is a 2-variable tca). These tca's share an important property:  they are multiplicity-free.  There is another symmetric algebra tca with this property:  $\Sym(\bV \oplus \lw^2 \bV)$.  This tca corresponds to the so-called odd symplectic groups \cite{proctor}. We do not discuss these groups in this paper, but the general pattern we establish applies to them as well.

\article[Opposites of stable categories]
We describe the opposite of each of the categories listed in \S \ref{intro:cats}:  the opposite of the category of algebraic representations of $G$ is the category of pro-algebraic representations of $G$; the opposite of the category of representations of the downwards Brauer algebra is the category of representations of the upwards Brauer algebra; etc.  These identifications are natural, and come from various dualities.  Having these opposite points of view can be convenient:  for example, a certain construction we use requires coalgebras in the usual categories; in the opposite categories it uses algebras, which are easier to contemplate.

\article[Quantum variants]
We expect that the content of this paper will work in the setting of ``quantum multilinear algebra'' in the sense of \cite{hh}. In this setting, we replace the symmetric monoidal category of $\bC$-vector spaces with a braided monoidal category of $\bC(q)$-vector spaces associated to an R-matrix.  The groups $\GL(\infty)$, etc., get replaced with their quantum analogues, and the symmetric groups are replaced by certain Hecke algebras.  There are likely many other changes to be made, but we have not thought through the details.

\subsection{Relation to other work}

\article[Polynomial theory of $\GL$]
The primary antecedent of our work is the polynomial theory of the general linear group.  This theory was developed by Schur and Weyl in the 1920s.  They completely understood the connection to symmetric groups, the theory of Schur functors, and the behavior of the specialization functor.  They also understood the stability properties of the theory, in the language of symmetric functions.

\article[Stability of character theory]
It is natural to wonder if the description of the stable polynomial character theory of $\GL$ in terms of symmetric functions can be extended to other settings, such as the other classical groups.  It was known to Littlewood \cite{littlewood} that the character theory of these groups is {\it eventually} stable: calculations in an idealized infinite setting remain valid if we specialize to a sufficiently large finite setting. The question of what happens if we remove ``sufficiently large'' was studied by Koike and Terada \cite{koiketerada, koike} (see also \cite{king}).  They constructed analogues of the ring of symmetric functions, and specialization homomorphisms down to the character rings at finite level.  For symmetric groups, the existence of such a ring follows from Murnaghan's theorem (which was first completely proved in \cite{littlewood2}).  As far as we are aware, the specialization map in this context was not studied until \cite{symc1}, and even there it is phrased in a different language.  These stable character rings can be viewed as the Grothendieck groups of the categories we study.  Furthermore, the complicated behavior of the specialization maps can be seen as a reflection of the lack of semi-simplicity of these categories.

\article[Centralizer algebras]
An important point in the Schur--Weyl theory is the determination of the centralizer of $\GL(d)$ acting on tensor powers of its standard representation $\bC^d$.  It is also natural to ask how this result extends to other contexts (see \cite{weyl} for this perspective).  The situation for the orthogonal and symplectic groups was considered by Brauer \cite{brauer}. The diagram algebras are named Brauer algebras in his honor, and many of their fundamental properties were worked out by Wenzl \cite{wenzl}. The centralizer algebras for mixed tensor representations of the general linear group have been considered by many authors, but, to the best of our knowledge, the first systematic study was given in \cite{koike} and \cite{walledbrauer}. This algebra was later named the walled Brauer algebra due to the description of its diagrams. Finally, for the symmetric group, the centralizer algebra seems to have been first studied in  \cite{martin} and \cite{jones}, and is now called the partition algebra.  The ideas and results of these works are important for us, as they furnish the combinatorial descriptions of our categories.

\article[Representations of infinite rank groups]
To the best of our knowledge, the categories $\Rep(\GL)$, $\Rep(\bO)$, and $\Rep(\Sp)$ first appeared in the works of Penkov and his collaborators \cite{penkovstyrkas, penkovserganova, koszulcategory} (where they are denoted $\mathbb{T}_{\mf{g}}$) and in the work of Ol$'$shanski{\u \i} \cite{olshanskii}. In the first group of references, these categories are studied in the context of locally finite Lie algebras.  They were defined by certain Lie-theoretic conditions and then later shown to coincide with the representations which appear as subquotients of (mixed) tensor representations. (Our point of view in this paper is to ignore these characterizations and just define the categories in terms of mixed tensor representations.)  Some of the results we establish in this paper occurred earlier in these works, though often with different proofs; we have tried to be careful to point out the overlaps.  We have not found $\Rep(\fS)$ in the literature, but we would not be surprised if it is there.

\article[Twisted commutative algebras]
The commutative algebraic description of our categories is in terms of twisted commutative algebras.  Tca's seem to have been around since the 1970s, and are closely related to Joyal's theory of species.  However, to the best of knowledge, they were first treated from the perspective of commutative algebra in \cite{snowden}.  We have since developed the theory further in \cite{symc1} and \cite{expos}.  One of our motivations for studying the representation theory of $\bO(\infty)$ was to understand $\bO(\infty)$-analogues of tca's.  It came as a surprise to us that this representation theory could be described using tca's, the very objects we set out to generalize!

\article[Deligne's categories]
Deligne has introduced the idea of representation theory in ``complex rank'' \cite{deligne}: he defined family of categories $\Rep(\fS(\delta))$ depending on a complex parameter $\delta$ which, in a sense, interpolates the categories $\Rep(\fS(n))$ for $n$ a positive integer. Similar definitions exist for the classical groups.  One can interpret the categories we consider in this paper as a limit of Deligne's categories as the parameter $\delta$ goes to infinity.  Furthermore, objects of, e.g., $\Rep(\bO(\delta))$ can be defined as representations of a certain Brauer category, which is closely related to the downwards Brauer categories that we explore in this paper.

\article[Spinors and oscillators] \label{art:spin-osc}
The spinor representations of the orthogonal groups (and many related representations) are not algebraic in the sense of this paper, and so do not fit into the theory we develop here. In fact, a stable theory of spinor representations can be developed along the lines of the formalism in this paper, which we do in \cite{spincat}. This issue does not arise for the symplectic Lie algebra, but a systematic consideration of going to the infinite rank limit suggests that the role of spinors are played by oscillators, which are infinite-dimensional even for finite rank symplectic Lie algebras. This is also discussed in \cite{spincat}.

\subsection{Future directions}

\article[Classical superalgebras] \label{art:future-super}
Let $V$ be a super (i.e., $\bZ/2$-graded) vector space. Its symmetries are encoded by the general linear Lie superalgebra $\fgl(V)$.  The polynomial representations (those constructed from tensor operations on $V$) of $\fgl(V)$ are well-understood \cite{bereleregev, sergeev}, and a next natural step is to investigate the category of mixed tensor representations (those constructed from tensor operations that involve both $V$ and $V^*$). The decomposition of mixed tensor powers into indecomposable summands has been investigated in \cite{BS:IV} and \cite{comeswilson}. We believe that the category of {\it rational Schur functors}, whose study is initiated in \S\ref{sec:rationalschur}, is crucial to a further understanding of these representations. For the orthosymplectic algebra (automorphisms of $V$ preserving a non-degenerate supersymmetric form), similar remarks apply for the category of orthogonal (resp., symplectic) Schur functors studied in \S\ref{sec:orthosymschur}.

While the character problem for these algebras is in principle solved (by reducing to the combinatorics of Kazhdan--Lusztig polynomials), see for example \cite{serganova, brundan, chenglamwang}, we do not know of a general tensor construction of the irreducible representations:  they do not all occur in mixed tensor spaces. However, the notion of {\it super duality}, see for example \cite{chenglam, chenglamwang}, suggests that this problem is closely related to the categories $\Rep(\GL)$, $\Rep(\bO)$, and $\Rep(\Sp)$. In a special case relevant here this was established first in \cite{BS:IV} based on \cite{BS:III}. The involved categories for the Lie superalgebra, $\fgl_n$ and the walled Brauer algebra then all occur as some idempotent truncation of a generalized Khovanov algebra. We point out the conjecture in the end of the introduction of \cite{BS:grading} on the connection between Deligne's category and the diagrammatic aspects of the walled Brauer algebra.

\article[Quiver descriptions of categories of homogeneous bundles] 
The general affine group $\GA$ is a stabilizer subgroup of a torus bundle over a projective space (specifically, the total space of the line bundle $\cO(1)$ minus its zero section). From the equivalence between homogeneous bundles on a homogeneous space and representations of the corresponding stabilizer subgroup, our model of $\Rep^{\pol}(\GA)$ gives a quiver description of the category of ``polynomial'' homogeneous bundles on a torus bundle over infinite dimensional projective space. 

A related setup was considered in \cite{ottavianirubei}, where quiver descriptions are given for homogeneous bundles on $G/P$, with $G$ simply-laced and $P$ a parabolic subgroup of Hermitian symmetric type (in particular, $G/P$ is compact in these cases). On a finite-dimensional projective space, each homogeneous bundle can be written as a polynomial homogeneous bundle twisted by a line bundle $\cO(d)$ for some $d \in \bZ$. Pulling back this homogeneous bundle to the total space of $\cO(1)$ has the effect of forgetting the twist, so one can interpret a suitable truncation of our quiver model of $\Rep^{\pol}(\GA)$ as a quotient of the quiver considered in \cite{ottavianirubei} for projective space.

The method of \cite{ottavianirubei} was through direct calculations. By working in our general framework of diagram categories, we hope to give conceptual descriptions of categories of homogeneous bundles on homogeneous spaces in future work.

\article[Structure of tca's]
As mentioned, twisted commutative algebras play an important role in this work.  The most basic example of a tca is $\Sym(\bV)$, which can be thought of as the polynomial ring $\bC[x_1, x_2, \ldots]$ equipped with its natural $\GL(\infty)$ action.  We gave a detailed analysis of the category of modules over this tca in \cite{symc1}. This category was also studied in \cite{fimodules}, where such modules are called FI-modules.

Having understood the simplest tca, it is natural to try to understand more complicated ones.  The next most simple ones to understand are polynomial rings of the form $\Sym(\bV^{\oplus n})=\bC[x_{i,j}]$, where $1 \le i \le n$ and $j \ge 1$.  We expect that the perspective of this paper will be useful in the study of such modules.  In particular, modules over the ``generic fiber'' of these tca's should be closely related to representations of certain generalizations of the general affine group.

The examples beyond $\Sym(\bV^{\oplus n})$ are much more complicated, as they enter the realm of unbounded tca's. We expect that $\Sym(\Sym^2(\bV))$ and $\Sym(\lw^2(\bV))$ might be tractable to analyze, however, and in recent joint work with Nagpal \cite{deg2tca}, the categories $\Rep(\bO)$ and $\Rep(\Sp)$ served an essential role in establishing that these algebras are noetherian.

\article[Pure free resolutions over quadric hypersurfaces] \label{art:mot-bs}
One of our original motivations for trying to understand the algebraic framework behind the work of Koike--Terada, and to better understand the structure of tca's, was to construct ``pure free resolutions'' over the homogeneous coordinate ring of a smooth quadric hypersurface. This construction is known for polynomial rings \cite{efw, es:bs} and is the first step in the proof \cite{es:bs} of the Boij--S\"oderberg conjectures \cite{boijsoderberg}, which describe the linear inequalities that define the cone of graded Betti numbers of finitely generated modules.

The construction in \cite{efw} naturally lives in the world of Schur functors, and it was observed by the first author and Jerzy Weyman that certain formal manipulations (i.e., on the level of the character ring) of these resolutions would produce the desired resolutions over a quadric. These manipulations made use of two constructions of \cite{koiketerada}: a certain transpose operation, and the specialization homomorphism mentioned earlier.  The natural setting for the specialization homomorphism is the specialization functor studied in \S\ref{sec:orthosymschur}. Unfortunately we have not yet understood the meaning of the Koike--Terada transpose operation.

\subsection{Organization}

\noarticle
In \S\ref{sec:prelim} we develop technical results used in the rest of the paper. The material in \S\ref{ss:schurfunctors} and \S\ref{ss:tca} is used implicitly throughout. The other subsections in \S\ref{sec:prelim} are technical, and on a first reading, we suggest that the reader skip them and refer back for the referenced statements as necessary.

The next four sections are devoted to the analysis of the stable representation theory of the five families of groups previously mentioned. The general linear group is treated in \S\ref{sec:gl}.  Due to their similarity, the orthogonal and symplectic groups are treated simultaneously in \S\ref{sec:osp}.  In \S\ref{sec:alin} we handle the general affine group.  Many results on this group were developed in \cite{symc1}, so this section is brief.  We tackle the symmetric group in \S\ref{sec:symgroup}.  This case is more involved than the others since the relevant diagram category is not weakly directed.  There is little interdependence between these sections, so the reader is encouraged to skip ahead to whichever groups are of the most interest.  However, the general linear group is treated in the most detail, with similar arguments omitted in later sections.

In \S\ref{sec:relations}, we examine the relationships between different categories (branching rules, including tensor products). We close with \S \ref{sec:ques}, which lists some open problems.

\subsection{Notation and conventions}

\noarticle
We work over the field of complex numbers $\bC$. Everything in this paper can be done over the rational numbers $\bQ$ if one works with split forms of the groups.  We list here some particularly important notation and conventions used throughout the paper:
\begin{itemize}
\item $\Vec$ = category of complex vector spaces.
\item $V^*$ = dual of a vector space $V$.
\item Abelian category = $\bC$-linear abelian category.
\item $\cA^{\fin}$ = category of finite length objects in $\cA$ ($\cA$ is an abelian category).
\item Tensor category = abelian category with biadditive symmetric monoidal functor.
\item Tensor functor = additive strict \emph{symmetric} monoidal functor.
\item Asymmetric tensor functor, same as above but not symmetric.
\item $\LEx(\cA, \cB)$ = left-exact functors $\cA \to \cB$.
\item $\Fun^{\otimes}(\cA, \cB)$ = tensor functors $\cA \to \cB$.
\item $\bV=\bC^{\infty}=\bigcup_{d \ge 1} \bC^d$.  We let $e_1, e_2, \ldots$ be a basis of $\bV$ compatible with this union.
\item $\bV_*=\bigcup_{d \ge 1} \bC^d{}^*$ is the restricted dual of $\bV$.  The $e_i^*$ form a basis of $\bV_*$.
\item $S_n$, and (in \S \ref{sec:symgroup}) $\fS(n)$, is the symmetric group on $n$ letters.
\item $\un$ is the set $\{1,\ldots,n\}$.
\end{itemize}
\newpage
Some other notations defined in the body of the paper include:
\begin{multicols}{2}
\noindent $\Mod_\Lambda$, $\Mod_\Lambda^\fin$, $\Mod_\Lambda^{\gfin}$ ($\Lambda$ a category), \pref{art:repncategory}\\
$f_*$, $f_\#$ ($f$ a functor), \pref{diag:op}\\
$S_x(V)$, $P_x(V)$, $I_x(V)$, \pref{art:projinj}\\
$\hom_\Lambda$, \pref{diag:hom}\\
$\otimes^\Lambda$, \pref{diag:tens}\\
$\cK$, $\Phi$, $\Psi$, \pref{ker-func}\\
$\amalg$, $\otimes_\ast$, $\otimes_\#$, \pref{diag:conv}\\
$\Rep^{\pol}(\GL)$, $T_n$, \pref{art:polyrepGL}\\
$\bM_\lambda$, $T^{d}_n$, \pref{art:schurfunctors}\\
$c^\nu_{\lambda,\mu}$, \pref{art:LR}\\
$\Rep(S_\ast)$, $\fs$, \pref{art:vecfs}\\
$\bS_\lambda$, \pref{art:polyfunctors}\\
${}^\vee$, \pref{art:duality}\\
${}^\dagger$, \pref{gl:transp}\\
$\cV$, \pref{art:defnV}\\
$\bC\langle 1 \rangle$, \pref{art:defnc1}\\
$\ast^A$, \pref{tca:tensor}\\
$\wh{\Vec}$, \pref{art:whvecdefn}\\
$\Rep(\GL)$, \pref{art:repGL}\\
$V_{\lambda, \lambda'}$, \pref{art:glweylinf}\\
$\wh{\Rep}(\GL)$, \pref{gl:pro}\\
$\cB_{n,m}$, \pref{gl:walgebra}\\
$\dwb$, \pref{dwb:def}\\
$\uwb$, \pref{uwb:def}\\
$\Sym(\bC\langle 1,1 \rangle)$, \pref{gl:amod}\\
$\Gamma_d$, \pref{gl:special}, \pref{c:special}, \pref{art:ga-special}, \pref{sym:special} \\
$T_0$, $T_1$, \pref{art:gl-T0T1}, \pref{art:osp-T0T1}, \pref{ga:T0}, \pref{sym:T0} \\
$\Rep(\bO)$, $\wh{\Rep}(\bO)$, \pref{art:Oinfdefn}\\
$V_\lambda$, \pref{art:weyl-inf-orth}, \pref{art:ga-weyl}, \pref{symgp:weyl:inf} \\
$\Rep(\Sp)$, $\wh{\Rep}(\Sp)$, \pref{art:Spinfdefn}\\
$\cB_n$, \pref{art:brauer-alg}\\
$\db$, $\ub$ \pref{art:def:db}\\
$\dsb$, $\usb$, \pref{art:dsbdefn}\\
$\GA(n)$, $\GA(\infty)$, \pref{art:gadefn}\\
$\ds$, $\us$, \pref{art:defnds}\\
$\fS$, $\Rep(\fS)$, $\wh{\Rep}(\fS)$, \pref{art:symgpdefn}\\
$\cA_n$, \pref{art:partition-algebra}\\
$\dpc$, $\upc$, \pref{art:defndp}
\end{multicols}

\vskip.6\baselineskip\par\noindent
{\bf Acknowledgements.}
We thank Sarah Kitchen, Ivan Penkov, Vera Serganova, and Catharina Stroppel for helpful conversations.  We also thank Stroppel for her comments on a draft of the paper.

\section{Preliminaries} \label{sec:prelim}

\subsection{Representations of categories} \label{ss:repsofcategory}

\article[Categorical conditions]
\label{cat-cond}
We consider the following conditions on a category $\Lambda$:
\begin{itemize}
\item {\bf Hom-finite:} for all objects $x$ and $y$, the set $\Hom_{\Lambda}(x, y)$ is finite.
\item {\bf Weakly directed:} any self-map is an isomorphism.  When this condition holds, there is a natural partial order on the isomorphism classes:  $x \le y$ if there exists a morphism $x \to y$.   
\item {\bf Inwards finite:} for any $x$ there exists only finitely many $y$, up to isomorphism, for which there exists a map $y \to x$.  (There is an obvious dual condition, called {\bf outwards finite}.)
\end{itemize}
We assume in this section that $\Lambda$ is Hom-finite, weakly directed and either inwards or outwards finite.  (We assume the same for similarly named categories, e.g., $\Lambda'$.)  Some of the results in this section do not require these conditions, but most categories we are interested in do satisfy these conditions, and assuming them allows for some simplifications in the discussion.

\article[Representations of categories]
\label{art:repncategory}
Let $\cA$ be an abelian category.  A {\bf representation} of $\Lambda$ valued in $\cA$ is a functor $\Lambda \to \cA$.  A morphism of representations is a natural transformation of functors.  We let $\cA^{\Lambda}$ denote the category of representations; it is an abelian category.  We typically denote the value of an object $M$ of $\cA^{\Lambda}$ on an object $x$ of $\Lambda$ by $M_x$.  We often write $\Hom_{\Lambda}$ in place of $\Hom_{\cA^{\Lambda}}$.  In the special case where $\cA=\Vec$, we write $\Mod_{\Lambda}$ in place of $\cA^{\Lambda}$.  As usual, we write $\Mod_{\Lambda}^{\fin}$ for the objects of $\Mod_{\Lambda}$ of finite length.  We write $\Mod_{\Lambda}^{\gfin} = (\Vec^{\fin})^{\Lambda}$ for the {\bf graded finite} objects of $\Mod_{\Lambda}$, i.e., those whose values at each object of $\Lambda$ are finite dimensional vector spaces.

\article[Duality]
We have a natural equivalence $\cA^{\Lambda^{\op}}=((\cA^{\op})^{\Lambda})^{\op}$.  If we have an equivalence $\cA=\cA^{\op}$, then this yields an equivalence $\cA^{\Lambda^{\op}}=(\cA^{\Lambda})^{\op}$.  In particular, we have an equivalence $\Mod_{\Lambda^{\op}}^{\gfin}=(\Mod_{\Lambda}^{\gfin})^{\op}$ defined by taking an object $M$ of $\Mod_{\Lambda^{\op}}^{\gfin}$ to the object $M^*$ of $\Mod_{\Lambda}^{\gfin}$ given by $(M^*)_x=M_x^*$.

\article[Push-forwards and pull-backs]
\label{diag:op}
Let $f \colon \Lambda \to \Lambda'$ be a functor.  We then get a pull-back functor $f^* \colon \cA^{\Lambda'} \to \cA^{\Lambda}$.  For an object $y$ of $\Lambda'$, let $\Lambda_{/y}$ denote the category of pairs $(x, \alpha)$ where $x$ is an object of $\Lambda$ and $\alpha \colon f(x) \to y$ is a morphism in $\Lambda'$.  Define ${}_{y \bs} \Lambda$ similarly, but with $\alpha \colon y \to f(x)$.  For $M \in \cA^{\Lambda}$ and $y \in \Lambda'$, define
\begin{displaymath}
f_*(M)_y=\lim(M \mid {}_{y \bs}\Lambda), \qquad f_{\#}(M)_y=\colim(M \mid \Lambda_{/y}).
\end{displaymath}
We assume in this section that these limits and colimits always exist, as they will in all cases of interest.  It is clear then that $f_*(M)$ and $f_{\#}(M)$ define objects of $\cA^{\Lambda'}$.  In fact, $f_*$ and $f_{\#}$ define functors $\cA^{\Lambda} \to \cA^{\Lambda'}$, and are the right and left adjoints of $f^*$, respectively.  Note that $f_*$ and $f_{\#}$ are interchanged under duality, i.e., the diagram
\begin{displaymath}
\xymatrix{
\cA^{\Lambda^{\op}} \ar[r]^{(f^{\op})_*} \ar@{=}[d] &
\cA^{(\Lambda')^{\op}} \ar@{=}[d] \\
((\cA^{\op})^{\Lambda})^{\op} \ar[r]^{(f_{\#})^{\op}} &
((\cA^{\op})^{\Lambda'})^{\op} }
\end{displaymath}
commutes (up to natural isomorphism).

\article[Simples, projectives and injectives] \label{art:projinj}
Let $x$ be an object of $\Lambda$ and let $G=\Aut(x)$, a finite group.  Let $V$ be an irreducible representation of $G$.  There is a unique (up to isomorphism) object $S_x(V)$ of $\Mod_{\Lambda}$ such that $S_x(V)_x=V$ and $S_x(V)_y=0$ if $y$ is not isomorphic to $x$.  The objects $S_x(V)$ are simple, and one easily sees that they exhaust the simple objects of $\Mod_{\Lambda}$.  From this description of simple objects, one finds that an object $M$ of $\Mod_{\Lambda}$ is of finite length if and only if $M_x$ is finite dimensional for all $x$ and non-zero for only finitely many $x$, up to isomorphism.  These statements depend crucially on $\Lambda$ being weakly directed.

Let $\rB{G}$ be the category with one object with automorphism group $G$, and let $i \colon \rB{G} \to \Lambda$ be the natural fully faithful functor.  We regard $V$ as an object of $\Mod_{\rB{G}}$, and we can thus form $P_x(V)=i_{\#}(V)$.  The object $P_x(V)$ is projective, since $\Mod_{\rB{G}}$ is semi-simple and $i_{\#}$ takes projectives to projectives.  It follows immediately from the definition of $i_{\#}$ and the weakly directed hypothesis that $P_x(V)_x=V$.  In fact, there is a natural surjection $P_x(V) \to S_x(V)$, which realizes $P_x(V)$ as the projective cover of the simple object $S_x(V)$.  The kernel of this surjection is supported on objects larger than $x$ (in the partial order).  If $\Lambda$ is outwards finite, then $P_x(V)$ has finite length and every object of $\Mod_{\Lambda}^{\fin}$ has finite projective dimension.

Similarly, we can form $I_x(V)=i_*(V)$.  The same discussion applies:  this is the injective envelope of $S_x(V)$, and if $\Lambda$ is inwards finite then $I_x(V)$ has finite length and every object of $\Mod_{\Lambda}^{\fin}$ has finite injective dimension.

\begin{Proposition}
\label{pf-finite}
Let $f \colon \Lambda \to \Lambda'$ be a functor and suppose that $\Lambda'$ is outwards finite.  Then $f_{\#}$ takes finite length objects of $\Mod_{\Lambda}$ to finite length objects of $\Mod_{\Lambda'}$.
\end{Proposition}

\begin{proof}
Since $f_{\#}$ is right exact, it suffices to show that it takes simple objects to finite length objects.  It follows from the definition of $f_{\#}$ that $f_{\#}(S_x(V))_y$ is a quotient of $V \otimes \bC[\Hom(f(x), y)]$.  Since $\Lambda'$ is Hom-finite, the space $V \otimes \bC[\Hom(f(x), y)]$ is finite dimensional for all $y$, and since $\Lambda'$ is outwards finite, it is non-zero for only finitely many isomorphism classes $y$.  It follows that $f_{\#}(S_x(V))$ is finite length, which completes the proof.
\end{proof}

\article[The $\Vec$-module structure on $\cA$]
\label{vecmodule}
Let $A$ be an object of $\cA$ and let $V$ be a vector space of finite dimension $d$.  We define objects $V \otimes A$ and $\Hom(V, A)$ of $\cA$ by the functors they represent:
\begin{align*}
\Hom_{\cA}(-, V \otimes A) &= V \otimes \Hom_{\cA}(-, A),\\
\Hom_{\cA}(-, \Hom(V, A)) &= \Hom(V, \Hom_{\cA}(-, A)).
\end{align*}
After picking a basis for $V$, both $V \otimes A$ and $\Hom(V, A)$ are canonically isomorphic to $A^{\oplus d}$, which shows that the above functors are representable.  Note that $\Hom(V, A)$ is canonically isomorphic to $V^* \otimes A$.  This construction appears in \cite[\S 2.9]{DeligneTannak}.

\article[Structured $\Hom$ spaces]
\label{diag:hom}
Suppose that $M$ is an object of $\Mod^{\fin}_{\Lambda}$ and $N$ is an object of $\cA^{\Lambda}$.  We define an object $\Hom_{\Lambda}(M, N)$ of $\cA$ as follows:
\begin{displaymath}
\Hom_{\Lambda}(M, N)=\lim_{(x, y) \in \Lambda^{\op} \times \Lambda} \Hom(M_x, N_y).
\end{displaymath}
One can show that this limit is equivalent to a finite limit, and therefore exists.  As this  definition is a bit abstract, we now give a more straightforward, though less intrinsic, definition.  Suppose that $\cA$ is a subcategory of $\Mod_R$ for some $\bC$-algebra $R$; this can essentially always be arranged by the Freyd--Mitchell embedding theorem.  We can then think of $N$ as an object of $\Mod_{\Lambda}$ such that each $N_x$ has the structure of an $R$-module, in a compatible manner.  We can thus form $\Hom_{\Mod_{\Lambda}}(M, N)$, and the result will have the structure of an $R$-module.  This is $\Hom_{\Lambda}(M, N)$.

\article[Structured tensor products]
\label{diag:tens}
There is a covariant version of the previous construction.  Suppose that $M$ is an object of $\Mod_{\Lambda^{\op}}^{\fin}$ and $N$ is an object of $\cA^{\Lambda}$.  We then put
\begin{displaymath}
M \otimes^{\Lambda} N = \lim_{(x, y) \in \Lambda^{\op} \times \Lambda} M_x \otimes N_y,
\end{displaymath}
which is an object of $\cA$.  The identifications $\Hom_{\Lambda}(M, N)=M^* \otimes^{\Lambda} N$ and $M \otimes^{\Lambda} N=\Hom_{\Lambda}(M^*, N)$ hold.

\article[Transforms defined by kernels]
\label{ker-func}
Let $\cK$ be an object of $\cA^{\Lambda}$.  We have contravariant functors
\begin{displaymath}
\Phi \colon \Mod_{\Lambda}^{\fin} \to \cA, \qquad \Phi(M) = \Hom_{\Lambda}(M, \cK)
\end{displaymath}
and
\begin{displaymath}
\Psi \colon \cA \to \Mod_{\Lambda}, \qquad \Psi(N) = \Hom_{\cA}(N, \cK).
\end{displaymath}
We call $\cK$ the {\bf kernel} of these functors.

\begin{proposition}
The contravariant functors $\Phi$ and $\Psi$ are adjoint on the right, that is, for $M \in \Mod_{\Lambda}^{\fin}$ and $N \in \cA$ there is a natural isomorphism
\begin{displaymath}
\Hom_{\Lambda}(M, \Psi(N))=\Hom_{\cA}(N, \Phi(M)).
\end{displaymath}
Furthermore, the adjunctions $M \to \Psi(\Phi(M))$ and $N \to \Phi(\Psi(N))$ are injective.
\end{proposition}

\begin{proof}
This is completely formal and left to the reader.
\end{proof}

\article[A criterion for equivalence]
\label{thm:ker-equiv}
Suppose $\Lambda$ is outwards finite and let $\cK \in \cA^{\Lambda}$.  For an object $x$ of $\Lambda$, we put
\begin{displaymath}
\cK_{[x]}=\bigcap_{f \colon x \to y} \ker(\cK_x \to \cK_y),
\end{displaymath}
where the intersection is taken over all non-isomorphisms $f$.  We then have the following general criterion for $\Phi$ and $\Psi$ to be equivalences.

\begin{theorem}
Suppose the following conditions hold:
\begin{enumerate}[\rm (a)]
\item For any object $x \in \Lambda$ and any irreducible representation $V$ of $\Aut(x)$, the space $\Hom_{\Aut(x)}(V, \cK_{[x]})$ is a simple object of $\cA$.
\item For each simple object $A$ of $\cA$ there is a unique object $x$ of $\Lambda$ (up to isomorphism) such that $\Hom_{\cA}(A, \cK_x)$ is non-zero, and it is then an irreducible representation of $\Aut(x)$.
\end{enumerate}
Then $\Phi \colon \Mod_{\Lambda}^{\fin} \to \cA^{\fin}$ and $\Psi \colon \cA^{\fin} \to \Mod_{\Lambda}^{\fin}$ are mutually quasi-inverse equivalences.
\end{theorem}

\begin{proof}
Let $x$ be an object of $\Lambda$ and let $V$ be an irreducible representation of $\Aut(x)$.  Then $\Hom_{\Lambda}(S_x(V), \cK)=\Hom_{\Aut(x)}(V, \cK_{[x]})$, and so (a) shows that $\Phi$ takes simple objects of $\Mod_{\Lambda}$ to simple objects of $\cA$.  Condition (b) exactly shows that $\Psi$ takes simple objects of $\cA$ to simple objects of $\Mod_{\Lambda}$.  Since $\Phi$ is left-exact and takes simples to simples, an easy inductive argument shows that $\len(\Phi(M)) \le \len(M)$; in particular, $\Phi$ takes finite length objects to finite length objects.  The same holds for $\Psi$.

Now, we have natural injective maps $\eta \colon \id \to \Phi \Psi$ and $\eta' \colon \id \to \Psi \Phi$.  For any finite length object $A$ of $\cA$, we thus have an injection $A \to \Phi(\Psi(A))$.  Since $\len(\Phi(\Psi(A))) \le \len(A)$, this map is necessarily an isomorphism.  Thus $\eta$ is an isomorphism of functors.  A similar argument shows that $\eta'$ is an isomorphism of functors, which completes the proof.
\end{proof}

\begin{Corollary}
In the setting of Theorem~\pref{thm:ker-equiv}, we have a covariant equivalence of categories $\Mod_{\Lambda^{\op}}^{\fin} \to \cA^{\fin}$ given by $M \mapsto M \otimes^{\Lambda} \cK$.
\end{Corollary}

\article[The pointwise tensor product]
\label{art:ptwise}
Suppose now that $\cA$ has a tensor product $\otimes$.  Given two objects $M$ and $N$ of $\cA^{\Lambda}$, we let $M \boxtimes N$ be the object of $\cA^{\Lambda}$ defined by $x \mapsto M_x \otimes N_x$.  We call this the {\bf pointwise tensor product} of $M$ and $N$.  This tensor product preserves finite length objects of $\Mod_{\Lambda}$, by the characterization of such objects given in \pref{art:projinj}.

\article[Convolution tensor products]
\label{diag:conv}
Suppose now that $\Lambda$ is equipped with a symmetric monoidal functor $\amalg$.  Let $p_1,p_2 \colon \Lambda \times \Lambda \to \Lambda$ be the projection maps.  We then have two {\bf convolution tensor products} on $\cA^{\Lambda}$, denoted $\otimes_{\#}$ and $\otimes_*$, and defined as follows
\begin{displaymath}
M \otimes_{\#} N=\amalg_{\#}(p_1^*M \boxtimes p_2^*N), \qquad
M \otimes_* N = \amalg_*(p_1^*M \boxtimes p_2^*N).
\end{displaymath}
When $\Lambda$ is outwards (resp.\ inwards) finite we put $\otimes=\otimes_{\#}$ (resp.\ $\otimes=\otimes_*$).  This tensor product preserves finite length objects of $\Mod_{\Lambda}$ by Proposition~\pref{pf-finite}.

\begin{Lemma}
\label{ker-tens-1}
Assume $\Lambda$ is outwards finite, let $M$ and $N$ be objects of $\Mod_{\Lambda}^{\fin}$ and let $M'$ and $N'$ be objects of $\cA^{\Lambda}$.  Then the natural map
\begin{displaymath}
\Hom_{\Lambda}(M, M') \otimes \Hom_{\Lambda}(N, N') \to \Hom_{\Lambda \times \Lambda}(p_1^*M \boxtimes p_2^*N, p_1^*M' \boxtimes p_2^*N')
\end{displaymath}
is an isomorphism.
\end{Lemma}

\begin{proof}
Fix $M'$ and $N'$, and define functors $F,G \colon \Mod_{\Lambda}^{\fin} \times \Mod_{\Lambda}^{\fin} \to \cA$ by
\begin{displaymath}
\begin{split}
F(M,N) &= \Hom_{\Lambda}(M,M') \otimes \Hom_{\Lambda}(N,N') \\
G(M,N) &= \Hom_{\Lambda \times \Lambda}(p_1^*M \boxtimes p_2^*N, p_1^*M' \boxtimes p_2^*N')
\end{split}
\end{displaymath}
There is a natural map $F(M,N) \to G(M,N)$, which we must show is an isomorphism.

We first treat the case where $M$ and $N$ are projective.  It suffices to treat the indecomposable case, so say $M=P_x(U)$ and $N=P_y(V)$.  We then have
\begin{displaymath}
\Hom_{\Lambda}(M, M')=\Hom_{\Aut(x)}(U, M'_x), \qquad
\Hom_{\Lambda}(N, N')=\Hom_{\Aut(y)}(V, M'_y).
\end{displaymath}
Then $p_1^*M \boxtimes p_2^*N=P_{(x,y)}(U \boxtimes V)$, and so
\begin{displaymath}
\Hom_{\Lambda \times \Lambda}(p_1^*M \boxtimes p_2^*N, p_1^*M' \boxtimes p_2^*N')=
\Hom_{\Aut(x) \times \Aut(y)}(U \boxtimes V, M'_x \boxtimes N'_y).
\end{displaymath}
The natural map $F(M,N) \to G(M,N)$ is the obvious one, and is an isomorphism by standard finite group representation theory.

Now we treat the case where $N$ is a projective and $M$ is arbitrary.  Pick a presentation
\begin{displaymath}
P' \to P \to M \to 0
\end{displaymath}
with $P$ and $P'$ finite length projectives.  We then have a commutative square
\begin{displaymath}
\xymatrix{
0 \ar[r] & F(M, N) \ar[r] \ar[d] & F(P, N) \ar[r] \ar[d] & F(P', N) \ar[d] \\
0 \ar[r] & G(M, N) \ar[r] & G(P, N) \ar[r] & G(P', N) }
\end{displaymath}
The two right vertical arrows are isomorphisms, and so the left arrow is an isomorphism as well.

Finally, we treat the case where both $M$ and $N$ are arbitrary.  Use the same reasoning as in the above paragraph:  pick a presentation for $N$, and use the fact that we know $F(M,P) \to G(M,P)$ is an isomorphism when $P$ is a finite projective object.
\end{proof}

\article[Tensor kernels]
\label{ker-tens}
Suppose $\Lambda$ is outwards finite. We say that an object $\cK \in \cA^{\Lambda}$ is a {\bf tensor kernel} if the functor $\cK \colon \Lambda \to \cA$ is a monoidal functor, that is, we require a functorial isomorphism $\cK(L \amalg L') \to \cK(L) \otimes \cK(L')$ which is compatible with the commutativity and associativity structures.  Equivalently, $\cK$ is a tensor kernel if there is an isomorphism $\amalg^*(\cK) \to p_1^*\cK \boxtimes p_2^*\cK$ which is compatible with the associativity and commutativity structures.

\begin{proposition} \label{prop:tensorkernel}
Let $\cK$ be a tensor kernel.  Then $\Phi$ defines a tensor functor $\Mod_{\Lambda}^{\fin} \to \cA$, that is, for any objects $M$ and $N$ of $\Mod_{\Lambda}^{\fin}$ there is a natural isomorphism $\Phi(M \otimes N)=\Phi(M) \otimes \Phi(N)$.
\end{proposition}

\begin{proof}
We have the following identifications:
\begin{displaymath}
\begin{split}
\Phi(M) \otimes \Phi(N) &= \Hom_{\Lambda}(M, \cK) \otimes \Hom_{\Lambda}(N, \cK)
= \Hom_{\Lambda \times \Lambda}(p_1^*M \boxtimes p_2^*N, p_1^* \cK \boxtimes p_2^* \cK) \\
&=\Hom_{\Lambda \times \Lambda}(p_1^*M \boxtimes p_2^* N, \amalg^*(\cK))
=\Hom_{\Lambda}(M \otimes N, \cK)=\Phi(M \otimes N).
\end{split}
\end{displaymath}
In the second equality we used Lemma~\pref{ker-tens-1}, in the third we used the fact that $\cK$ is a tensor kernel and in the fourth we used the adjunction between $\amalg^*$ and $\amalg_{\#}$.
\end{proof}

\begin{Corollary}
\label{ker-tens3}
In the setting of Proposition~\pref{ker-tens}, the functor $\Mod_{\Lambda^{\op}}^{\fin} \to \cA^{\fin}$ defined by $M \mapsto M \otimes^{\Lambda} \cK$ is a tensor functor.
\end{Corollary}

\subsection{\texorpdfstring{Polynomial representations of $\GL(\infty)$ and category $\cV$}{Polynomial representations of GL(inf)}} \label{ss:schurfunctors}

\noarticle
In this section, we review the category $\cV$ and some of its models.  We refer to \cite[Part~2]{expos} for a more thorough discussion.  There are many similarities between this theory and the more difficult ones developed later, so this material serves as a good warm-up for the rest of the paper.

\article[Polynomial representations]
\label{art:polyrepGL}
Let $\bV=\bC^{\infty}=\bigcup_{n \ge 1} \bC^n$ and let $T_n=\bV^{\otimes n}$ be its $n$th tensor power.  The group $\GL(\infty)=\bigcup_{n \ge 1} \GL(n)$ acts on $\bV$ and on $T_n$.  We say that a representation of $\GL(\infty)$ is {\bf polynomial} if it is a subquotient of a finite direct sum of $T_n$'s.  We denote by $\Rep^{\pol}(\GL)$ the category of polynomial representations.  It is abelian and stable under tensor products.

\article[The action of the center]
\label{gl:center}
The group $\GL(\infty)$ does not contain the scalar matrices and thus has trivial center.  However, large diagonal matrices can be used to approximate scalar matrices, and this allows us to define an action of the ``central $\bG_m$'' on polynomial representations.  Precisely, let $V$ be a polynomial representation of $\GL(\infty)$ and let $z$ be an element of $\bG_m$.  Let $g_{n,z} \in \GL(\infty)$ be the diagonal matrix whose first $n$ entries are $z$ and whose remaining entries are 1.  Given $v \in V$, we define $zv$ to be $g_{n,z}v$ for $n \gg 0$.  Since $v$ only involves a finite number of the basis vectors of $\bV$, this formula is well-defined, and one easily verifies that it defines an action of $\bG_m$ on $V$ which commutes with that of $\GL(\infty)$.

An action of $\bG_m$ is equivalent to a $\bZ$-grading, so the above paragraph can be rephrased as:  every polynomial representation $V$ of $\GL(\infty)$ admits a canonical grading $V=\bigoplus_{n \in \bZ} V_n$.  The space $V_n$ is the subspace where $\bG_m$ acts through its $n$th power.  The representation $T_n$ is concentrated in degree $n$.  It follows that every polynomial representation is graded by $\bZ_{\ge 0}$.

\article[Weyl's construction (finite case)]
\label{art:schurfunctors}
To determine the structure of $\Rep^{\pol}(\GL)$ we use Schur--Weyl duality, which we now recall.  Let $T_n^d=(\bC^d)^{\otimes n}$.  The group $S_n$ acts on $T_n^d$ by permuting coordinates, and this action commutes with that of $\GL(d)$.  For a partition $\lambda$ of $n$, let $\bM_\lambda$ be the irreducible representation of $S_n$ associated to $\lambda$; our conventions are such that $\lambda=(n)$ gives the trivial representation and $\lambda=(1^n)$ the sign representation.  Put
\begin{displaymath}
V_{\lambda}^d=\Hom_{S_n}(\bM_{\lambda}, T_n^d).
\end{displaymath}
We then have the following result:

\begin{proposition}
Let $r=\ell(\lambda)$.  If $r \le d$ then $V_{\lambda}^d$ is the irreducible representation of $\GL(d)$ with highest weight $(\lambda_1, \ldots, \lambda_r, 0, \ldots, 0)$.  If $r>d$ then $V_{\lambda}^d=0$.
\end{proposition}

\article[Weyl's construction (infinite case)]
\label{art:schurfunctor-inf}
For a partition $\lambda$ of $n$ we again put
\begin{displaymath}
V_{\lambda}=\Hom_{S_n}(\bM_{\lambda}, T_n).
\end{displaymath}
Note that we have a decomposition of $S_n \times \GL(\infty)$ representations
\begin{equation}
\label{pol:decomp}
T_n=\bigoplus_{\vert \lambda \vert=n} \bM_{\lambda} \boxtimes V_{\lambda}
\end{equation}
The following result classifies the simple objects of $\Rep^{\pol}(\GL)$.

\begin{Proposition}
\label{pol:simples}
The $V_{\lambda}$ constitute a complete irredundant set of simple objects of $\Rep^{\pol}(\GL)$.
\end{Proposition}

\begin{proof}
Since $V_{\lambda}=\bigcup_{d \ge 0} V_{\lambda}^d$ and $V_{\lambda}^d$ is non-zero and irreducible for $d \gg 0$, it follows that the $V_{\lambda}$ are simple.  Every simple object of $\Rep^{\pol}(\GL)$ is a constituent of some $T_n$, and \eqref{pol:decomp} shows that every simple constituent of $T_n$ is isomorphic to some $V_{\lambda}$.  Thus the $V_{\lambda}$ are a complete set of simples.  Finally, to prove that they are irredundant, we note that the character of $V_{\lambda}$ is the Schur function $s_{\lambda}$, and $s_{\lambda} \ne s_{\mu}$ for $\lambda \ne \mu$.
\end{proof}

\begin{Proposition} \label{prop:polGLss}
The category $\Rep^{\pol}(\GL)$ is semi-simple, i.e., every polynomial representation is a finite direct sum of $V_{\lambda}$'s.
\end{Proposition}

\begin{proof}
By \eqref{pol:decomp} and \pref{pol:simples}, $T_n$ is semi-simple.  Since any finite direct sum or quotient of semi-simple objects is again semi-simple, the result follows.
\end{proof}

\article[Tensor product decompositions] \label{art:LR}
By Proposition~\pref{prop:polGLss}, a tensor product $V_{\lambda} \otimes V_{\mu}$ of simples decomposes into a direct sum of $V_{\nu}$'s with certain multiplicities.  The multiplicity of $V_{\nu}$ in this decomposition is called the {\bf Littlewood--Richardson coefficient}, and is denoted by $c^\nu_{\lambda, \mu}$. See \cite[(2.14)]{expos} for some basic discussion of, and references for, these coefficients.

\article[The categories $\Mod_{\fs}$ and $\Rep(S_{\ast})$] 
\label{art:vecfs}
Let $\fs$ denote the category whose objects are finite sets and whose morphisms are bijections.  This category satisfies all the conditions of \pref{cat-cond}.  Let $\Mod_{\fs}$ denote the representation category, see \pref{art:repncategory}.  This category is equivalent to the category $\Rep(S_{\ast})$ consisting of sequences $(M_n)_{n \ge 0}$ where $M_n$ is a representation of the symmetric group $S_n$.  Disjoint union of finite sets gives $\fs$ a monoidal structure and endows $\Mod_{\fs}$ with convolution tensor products, as discussed in \pref{diag:conv}.  Since $\fs$ is a groupoid, we have $\otimes_*=\otimes_{\#}$.  In terms of the category $\Rep(S_{\ast})$, this tensor product is given by:
\begin{displaymath}
(M \otimes N)_n=\bigoplus_{i+j=n} \Ind_{S_i \times S_j}^{S_n} (M_i \otimes N_j),
\end{displaymath}
where $\Ind$ denotes induction.

\article[The equivalence between $\Mod_{\fs}^{\fin}$ and $\Rep^{\pol}(\GL)$]
\label{art:schurweyl}
For a finite set $L$, let $\cK_L=(\bC^{\infty})^{\otimes L}$.  Then $\cK$ is naturally an object of $\Rep^{\pol}(\GL)^{\fs}$, and is obviously a tensor kernel in the sense of \pref{ker-tens}.  It follows easily from the results of this section that the functors $\Phi$ and $\Psi$ of \pref{ker-func} induce equivalences between $\Mod^{\fin}_{\fs}$ and $\Rep^{\pol}(\GL)$.

\article[Polynomial functors]
\label{art:polyfunctors}
A functor $F \colon \Vec^{\fin} \to \Vec^{\fin}$ is {\bf polynomial} if for every pair of finite dimensional vector spaces $V$ and $W$, the induced map
\begin{displaymath}
F \colon \Hom(V, W) \to \Hom(F(V), F(W))
\end{displaymath}
is a polynomial map of vector spaces, the degree of which is bounded independently of $V$ or $W$.  Let $\cS$ denote the category of polynomial functors.  Given a finite dimensional representation $M$ of $S_n$ and a vector space $V$, put
\begin{displaymath}
S_M(V)=\Hom_{S_n}(M, V^{\otimes n}).
\end{displaymath}
Then $V \mapsto S_M(V)$ is a polynomial functor.  The functor $S_{\bM_{\lambda}}$ is denoted $\bS_{\lambda}$ and called the {\bf Schur functor} associated to $\lambda$.  The rule $M \mapsto S_M$ defines a functor $\Rep^{\fin}(S_n) \to \cS$ which extends additively to a functor $\Rep^{\fin}(S_{\ast}) \to \cS$.  We have the following result \cite[(5.4.4)]{expos}:

\begin{proposition}
The functor $\Rep^{\fin}(S_{\ast}) \to \cS$ is an equivalence of categories.
\end{proposition}

\begin{remark}
The definition of $S_M(V)$ clearly makes sense when $V$ is infinite dimensional.  Thus, by the above proposition, every polynomial functor extends canonically to a functor $\Vec \to \Vec$.
\end{remark}

\article[Universal description] \label{art:schur-univ}
The category $\Rep^{\pol}(\GL)$ has the following universal description:

\begin{proposition}
Let $\cA$ be a tensor category. To give a tensor functor $\Rep^{\pol}(\GL) \to \cA$ is the same as to give an object of $\cA$.  The equivalence takes a functor $F$ to the object $F(\bV)$.
\end{proposition}

The proposition can be phrased equivalently as saying that the functor
\begin{displaymath}
\Phi \colon \Fun^{\otimes}(\Rep^{\fin}(S_{\ast}), \cA) \to \cA, \qquad \Phi(F)=F(\bM_{(1)})
\end{displaymath}
is an equivalence of categories.  Let us now explain why this is.  Let $V$ be an object of $\cA$.  Given a finite dimensional representation $M$ of $S_n$, put
\begin{displaymath}
S_M(V)=\Hom_{S_n}(M, V^{\otimes n}).
\end{displaymath}
(See \pref{vecmodule} for how to make sense of this.) The definition of $S_M(V)$ extends additively to all $M \in \Rep(S_{\ast})$, and $M \mapsto S_M(V)$ is a tensor functor.  We have thus defined a functor
\begin{displaymath}
\Psi \colon \cA \to \Fun^{\otimes}(\Rep^{\fin}(S_{\ast}), \cA).
\end{displaymath}
We leave it to the reader to show that $\Phi$ and $\Psi$ are mutually quasi-inverse.

\begin{remark}
Let $\Cat$ be the 2-category whose objects are categories, whose 1-morphisms are functors and whose 2-morphisms are natural transformations of functors.  Let $\TCat$ be the 2-category whose objects are tensor categories, whose 1-morphisms are left-exact tensor functors and whose 2-morphisms are natural transformations of tensor functors.  Let $T \colon \TCat \to \Cat$ be the forgetful functor.  The above proposition can be rephrased more abstractly as: the functor $T$ is corepresented by $\Rep^{\pol}(\GL)$, with the universal object being $\bV \in T(\Rep^{\pol}(\GL))$.  (For this statement, there is no need to restrict to left-exact tensor functors, but the restriction is necessary for similar statements occurring below.)
\end{remark}

\article[The category $\cV$]
We showed that the categories $\Rep^{\pol}(\GL)$, $\Mod_{\fs}^{\fin}$, $\Rep^{\fin}(S_{\ast})$ and $\cS$ are all equivalent.  We write $\cV^{\fin}$ for any of these categories.  We think of $\cV^{\fin}$ abstractly, and regard the four specific categories just mentioned as concrete models for it:  we call them the GL, fs, sequence and Schur models.  It will be convenient to introduce a category $\cV$ whose objects need not be finite length.  The fs-model of $\cV$ is  $\Mod_{\fs}$.  The GL-model of $\cV$ consists of representations of $\GL(\infty)$ which occur as a subquotient of a (possibly infinite) direct sum of $T_n$'s.  We let $\cV^{\gfin}$ be the full subcategory of $\cV$ on objects in which each simple has finite multiplicity; it is equivalent to $\Mod_{\fs}^{\gfin}$.

\article[Duality] \label{art:duality}
Let $M=(M_n)$ be an object of $\Rep^{\gfin}(S_{\ast})$.  We define an object $M^{\vee}$, called the {\bf dual} of $M$, by $(M^{\vee})_n=M_n^*$.  Duality gives an equivalence $(\cV^{\gfin})^{\op} \to \cV^{\gfin}$ of tensor categories.  There is a canonical isomorphism $M \to (M^{\vee})^{\vee}$.  There is also a non-canonical isomorphism $M \cong M^{\vee}$ for each object $M$, since irreducible representations of symmetric groups are self-dual.  Note that in the GL-model, duality is {\bf not} the usual linear dual; see \cite[\S 6.1.6]{expos}. 

\article[Transpose]
\label{gl:transp}
Let $M=(M_n)$ be an object of $\Rep(S_{\ast})$.  We define an object $M^{\dag}$, called the {\bf transpose} of $M$, by $(M^{\dag})_n=M_n \otimes \sgn$, where $\sgn$ is the sign character of $S_n$.  Transpose gives an {\bf asymmetric} equivalence $\cV \to \cV$ of tensor categories:  it behaves in the expected way with respect to tensor products, but does not respect the symmetric structure on tensor products.  For example, it interchanges certain symmetric and exterior powers.  See \cite[\S 7.4]{expos} for more details.

\article[The category $\cV^{\otimes 2}$]
\label{art:V2}
We will occasionally need to use the category $\cV^{\otimes 2}$.  Objects in this category can be thought of in three equivalent ways:
\begin{itemize}
\item functors $\fs \times \fs \to \Vec$, or
\item representations of $\GL(\infty) \times \GL(\infty)$ which are polynomial in each group, or 
\item polynomial functors $\Vec \times \Vec \to \Vec$.  
\end{itemize}
See \cite[\S 6.2]{expos} for further discussion. The transpose duality in \pref{gl:transp} extends in an obvious way to $\cV^{\otimes 2}$ by twisting a representation of $S_n \times S_m$ by the product of its sign characters. This gives an asymmetric equivalence $\cV^{\otimes 2} \to \cV^{\otimes 2}$ of tensor categories. Similarly, it is sometimes useful to apply partial transpose duality just to one of the factors. To distinguish, the first one will be called the full transpose.

\subsection{Twisted commutative algebras} \label{ss:tca}

\article[Twisted commutative algebras] \label{art:defnV}
The category $\cV$ is an abelian tensor category, and so there is a notion of commutative algebra in it.  A {\bf twisted commutative algebra} (tca) is an associative commutative unital algebra in $\cV$. Each model of $\cV$ provides a different way to think about tca's:
\begin{itemize}
\item In $\Rep^{\pol}(\GL)$, a tca is a commutative associative unital $\bC$-algebra equipped with an action of $\GL(\infty)$ by algebra homomorphisms, under which it decomposes as an infinite direct sum of polynomial representations.
\item In $\Mod_{\fs}$, a tca is a functor $A \colon \fs \to \Vec$ equipped with a multiplication map $A_L \otimes A_{L'} \to A_{L \amalg L'}$ that satisfies the relevant conditions.
\item In $\cS$, a tca is a functor from $\Vec^{\fin}$ to the category of commutative associative unital $\bC$-algebras.  A finitely generated tca will take values in finitely generated $\bC$-algebras. 
\end{itemize}
A {\bf 2-variable tca} is an algebra in $\cV^{\otimes 2}$.  For a more detailed discussion of this section, see \cite[\S 8]{expos}.

\article[An example] \label{art:defnc1}
We write $\bC\langle 1 \rangle$ for the object of $\cV$ given as follows:
\begin{itemize}
\item the standard representation $\bV$ in $\Rep^{\pol}(\GL)$, or
\item the identity functor $\Vec \to \Vec$ in $\cS$, or
\item the trivial representation of $S_1$ in $\Rep(S_{\ast})$.  
\end{itemize}
The symmetric algebra $\Sym(\bC\langle 1 \rangle)$ is the simplest nontrivial example of a tca.  In $\Rep^{\pol}(\GL)$, it corresponds to the algebra $\bC[x_1, x_2, \ldots]$ equipped with the usual action of $\GL(\infty)$ by linear substitutions. We refer to \cite{symc1} for an in-depth treatment of the structure of this tca.

\article[Tensor products of modules]
\label{tca:tensor}
Let $A$ be a tca.  Then $\Mod_A$ naturally has a tensor product $\otimes_A$ induced from the one on $\cV$.  This is usually the correct tensor product to use, but is not for the purposes of this paper.  We now define an alternative tensor product.  Let $\bE$ and $\bE'$ be two copies of $\bC^{\infty}$.  Suppose $M$ is an $A$-module.  Then $M(\bE \oplus \bE')$ is naturally an $A(\bE \oplus \bE')$-module.  There is a natural ring homomorphism $A(\bE) \otimes A(\bE') \to A(\bE \oplus \bE')$, and so we can regard $M(\bE \oplus \bE')$ as a module over this ring.  We have thus defined a functor $a^* \colon \Mod_A \to \Mod_{A \boxtimes A}$, where here $\boxtimes$ denotes the external tensor product, given by $M \boxtimes N$ for $M(\bE) \otimes N(\bE')$.  The functor $a^*$ has a right adjoint $a_*$.  For two $A$-modules $M$ and $N$, we define $M \ast^A N$ to be $a_*(M \boxtimes N)$.  Then $\ast^A$ gives the category $\Mod_A$ a second tensor structure.  (We do not prove the existence of $a_*$, but it can be deduced from Proposition~\pref{sg:equiv}.)

\begin{remark}
There is an opposite version of this:  if $A$ is a twisted cocommutative coalgebra (see \pref{tc-coalg} below), then $\CoMod_A$ has a natural tensor product $\otimes^A$, and an analogue of the above procedure yields an alternative tensor product $\ast_A$.
\end{remark}

\article[Computation of certain $\Ext$ groups]
\label{tca:ext}
Let $U$ be an object of $\cV$ with $U_0=0$ and let $A$ be the tca $\Sym(U)$.  The simple $A$-modules are the objects $\bS_{\lambda}$ of $\cV$ endowed with the trivial $A$-module structure (i.e., the action comes via the homomorphism $A \to \bC$).  We now compute the $\Ext$'s of these modules.

\begin{proposition}
We have a natural identification
\begin{displaymath}
\Ext^i_A(\bS_{\lambda}, \bS_{\mu})=\Hom_{\cV}(\bS_{\lambda} \otimes \lw^i{U}, \bS_{\mu}).
\end{displaymath}
In particular, if $U$ is concentrated in degree $d$ then this $\Ext$ group vanishes unless $\vert \mu \vert-\vert \lambda \vert=di$.
\end{proposition}

\begin{proof}
The Koszul complex gives a projective resolution $A \otimes \lw^{\bullet}{U}$ of the residue field $\bC$.  Tensoring with $\bS_{\lambda}$ gives a projective resolution of the simple module $\bS_{\lambda}$.  Applying $\Hom_A(-,\bS_{\mu})$ gives a complex whose terms are $\Hom_{\cV}(\bS_{\lambda} \otimes \lw^{\bullet}{U}, \bS_{\mu})$ and whose differentials vanish. 
\end{proof}

\article[Coalgebras and duality]
\label{tc-coalg}
A {\bf twisted cocommutative coalgebra} is a coassociative cocommutative counital coalgebra in $\cV$.  Recall that the duality ${}^\vee$ \pref{art:duality} is a contravariant equivalence of tensor categories from $\cV^{\gfin}$ to itself.  It follows that if $A$ is a graded-finite tca, then $A^\vee$ is naturally a twisted cocommutative coalgebra, and if $M$ is an $A$-module, then $M^\vee$ becomes an $A^\vee$-comodule.  We thus find that duality yields an equivalence
\[
{}^\vee \colon (\Mod_A^{\gfin})^{\op} \to \CoMod^{\gfin}_{A^\vee}
\]
of abelian categories.  We note in particular that this equivalence preserves finite length objects, takes finitely generated modules to finitely cogenerated comodules and interchanges projective and injective objects.

\article[Weyl algebras and projective modules] \label{art:bialgebras}
Let $U$ be an object of $\cV$, thought of as a representation of $\GL(\infty)$, and let $U^*$ be the full linear dual of $U$, which is a non-polynomial representation of $\GL(\infty)$.  Let $A=\Sym(U)$, a tca.  Let $A'$ be the quotient of the tensor algebra on $U \oplus U^*$ by two-sided ideal generated by the following relations: (a) $xy=yx$ for $x, y \in U$; (b) $\lambda \mu=\mu \lambda$ for $\lambda,\mu \in U^*$; (c) $\lambda x - x \lambda=\lambda(x)$ for $x \in U$ and $\lambda \in U^*$.  One can think of $A'$ as a Weyl algebra (algebra of differential operators).  The group $\GL(\infty)$ acts on $A'$, and by an $A'$-module we mean one with a compatible $\GL(\infty)$ action.  It is clear that $A$, as well as any projective module over $A$, has the structure of an $A'$-module.  Conversely:

\begin{proposition} 
A finitely generated $A$-module $M$ that has a compatible $A'$-module structure is projective as an $A$-module. 
\end{proposition}

\begin{proof}
Let $V$ be an irreducible $\GL(\infty)$-submodule of $M$ which is annihilated by all of the partial derivatives in $A'$.  The $A'$-submodule of $M$ generated by $V$ is a quotient of $A \otimes V$. Since $V$ is irreducible, it follows that $A \otimes V$ is simple as an $A'$-module, and hence the map $A \otimes V \to M$ is injective. It is clear that $V$ is part of a minimal generating set of $M$, so $M / (A \otimes V)$ has fewer generators.  It follows by induction that $M$ has a filtration whose associated graded is projective, which implies that $M$ is projective.
\end{proof}

\subsection{Semi-group tca's and diagram categories}
\label{ss:sg-tca}

\article
\label{sg:def}
Let $\cG$ be a twisted commutative monoid; that is, $\cG$ is a functor $\fs \to \fs$ equipped with a multiplication map
\begin{displaymath}
\cG_L \amalg \cG_{L'} \to \cG_{L \amalg L'}
\end{displaymath}
which is associative, commutative, and has an identity, in the same sense as tca's.  We assume that $\cG_L$ is finite for all $L$.  Put $A_L=\bC[\cG_L]$.  Then $A$ is a tca.

Let $\Lambda$ be the following category.  Objects are finite sets.  A morphism $L \to L'$ consists of a triple $(U, \Gamma, f)$ where $U$ is a subset of $L'$, $\Gamma$ is an element of $\cG_U$ and $f$ is a bijection $L \to L' \setminus U$.  Given a morphism $L \to L'$ corresponding to $(U, \Gamma, f)$ and a morphism $L' \to L''$ corresponding to $(V, \Delta, g)$, the composition corresponds to the data $(g(U) \amalg V, g(\Gamma) \amalg \Delta, gf)$.   This category is $\Hom$-finite, weakly directed, and inwards finite.  Disjoint union of sets endows $\Lambda$ with a monoidal operation $\amalg$.  We let $\otimes=\otimes_*$ be the convolution tensor product on $\Mod_{\Lambda}$, see \pref{diag:conv}.  There is a natural equivalence $\Mod_A=\Mod_{\Lambda}$; in fact, a representation of $\Lambda$ is simply an $A$-module from the point of view of the $\fs$-model.

We note that everything said above works in the multivariate case, i.e., if $\cG \colon \fs^n \to \fs$ is an $n$-variable tc monoid then $A=\bC[\cG]$ is an $n$-variable tca, the definition of $\Lambda$ still makes sense (objects are now $n$-tuples of finite sets) and we have an equivalence $\Mod_A=\Mod_{\Lambda}$.

\begin{Proposition}
\label{sg:equiv}
Under the equivalence $\Mod_A=\Mod_{\Lambda}$, the tensor products $\ast^A$ and $\otimes$ coincide.
\end{Proposition}

\begin{proof}
Let $\cG'$ be the 2-variable tc monoid given by $\cG'_{L,L'}=\cG_{L \amalg L'}$, and let $\cG''$ be the one given by $\cG''_{L,L'}=\cG_L \times \cG_{L'}$.  Let $\Lambda'$ and $\Lambda''$ be the categories associated to $\cG'$ and $\cG''$.  Also, let $A'$ be the 2-variable tca $A(\bE \oplus \bE')$ and let $A \boxtimes A$ be as in \pref{tca:tensor}.  We have the following commutative diagram
\begin{displaymath}
\xymatrix{
\Mod_{\Lambda} \ar@{=}[d] \ar[r] & \Mod_{\Lambda'} \ar@{=}[d] \ar[r] & \Mod_{\Lambda''} \ar@{=}[d] \\
\Mod_A \ar[r] & \Mod_{A'} \ar[r] & \Mod_{A \boxtimes A} }
\end{displaymath}
We now elaborate on the diagram.  The functors in the bottom row are as in \pref{tca:tensor}.  The first functor on the top row takes $M \in \Mod_{\Lambda}$ to the functor $(L, L') \mapsto M_{L \amalg L'}$ in $\Mod_{\Lambda'}$.  The second functor in the top row is pullback along the functor $\Lambda'' \to \Lambda'$ corresponding to the homomorphism $\cG'' \to \cG$ given by the monoidal operation on $\cG$.  The vertical equivalences all come from viewing the bottom categories in the fs-model.  Commutativity of the diagram is an exercise left to the reader.

Now, the composition of the bottom two horizontal functors is the functor $a^* \colon \Mod_A \to \Mod_{A \boxtimes A}$ discussed in \pref{tca:tensor}.  After identifying $\Lambda''$ with $\Lambda \times \Lambda$, the composition of the top horizontal functors is identified with $\amalg^*$, where $\amalg$ is the monoidal functor on $\Lambda$.  It follows that $a_*$ coincides with $\amalg_*$, which proves the proposition.
\end{proof}

\begin{remark}
The proposition also holds in the multivariate case.
\end{remark}

\article
\label{sg:op}
Everything we just did has an opposite version, as follows.  The tca $A$ is naturally a twisted commutative coalgebra.  Let $\Lambda'$ be defined like $\Lambda$ but with $U$ a subset of $L$ instead of $L'$.  Clearly, $\Lambda'$ is just the opposite category of $\Lambda$.  We let $\otimes=\otimes_{\#}$ on $\Mod_{\Lambda'}$.  Then $\CoMod_A$ is equivalent to $\Mod_{\Lambda'}$, with $\ast_A$ corresponding to $\otimes$.

\subsection{Profinite vector spaces} \label{ss:profinite}

\noarticle
Throughout this section, we treat $\bC$ as a topological field endowed with the discrete topology.  All finite dimensional vector spaces are endowed with the discrete topology as well.

\article[Inverse limits]
Suppose that we have an inverse system $(V_i)_{i \in I}$ of discrete vector spaces, indexed by some directed category $I$.  Let $V$ be the inverse limit in the category of topological vector spaces.  The topology on $V$ can be described as follows.  Let $\pi_i \colon V \to V_i$ be the natural map and let $U_i$ be its kernel.  Then the $U_i$ form a neighborhood basis of $0 \in V$, and so a subset of $V$ is open if and only if it is a union of translates of the $U_i$'s.

\article[Profinite vector spaces] \label{art:whvecdefn}
Let $V$ be a topological vector space.  We say that $V$ is {\bf profinite} if the natural map
\begin{displaymath}
V \to \varprojlim_{V \to V'} V'
\end{displaymath}
is an isomorphism of topological vector spaces, where the limit is taken over the continuous surjections from $V$ to finite dimensional vector spaces $V'$.  Then $V$ is profinite if and only if it has a neighborhood basis of the identity consisting of open subspaces of finite codimension.  We denote by $\wh{\Vec}$ the category of profinite vector spaces, with morphisms being continuous linear maps.

\article[Continuous dual]
Let $V$ be a profinite vector space.  We define the {\bf continuous dual} of $V$, denoted $V^{\vee}$, to be the space of continuous linear functionals $V \to \bC$.  If $V$ is the inverse limit of the system $(V_i)$, with $V_i$ finite dimensional, then $V^{\vee}$ is the direct limit of the system $(V_i^*)$.  Similarly, if $V$ is a discrete vector space we define its continuous dual, denoted $V^{\vee}$, to be the usual dual $V^*$ but regarded as a profinite vector space; that is, write $V=\varinjlim V_i$ with $V_i$ finite dimensional, and then $V^{\vee}=\varprojlim V_i^*$.  If $V$ is either profinite or discrete then the natural map $V \to (V^{\vee})^{\vee}$ is an isomorphism.  We thus see that ``continuous dual'' provides a contravariant equivalence of categories between $\Vec$ and $\wh{\Vec}$.  In particular, $\wh{\Vec}$ is abelian.

\article[Completed tensor product]
Let $V$ and $W$ be profinite vector spaces.  The tensor product $V \otimes W$ is not a profinite vector space in a natural way.  We define the {\bf completed tensor product}, denoted $V \hatotimes W$, as $\varprojlim(V_i \otimes W_j)$, where $V_i=\varprojlim V_i$ and $W=\varprojlim W_j$ with $V_i$ and $W_j$ discrete.  The functor $\hatotimes$ endows $\wh{\Vec}$ with the structure of a tensor category.  Furthermore, it is compatible with duality:  if $V$ and $W$ are profinite then $(V \hatotimes W)^{\vee}=V^{\vee} \otimes W^{\vee}$, while if $V$ and $W$ are both discrete then $(V \otimes W)^{\vee}=V^{\vee} \hatotimes W^{\vee}$.

\section{The general linear group} \label{sec:gl}

\subsection{\texorpdfstring{Representations of $\GL(\infty)$}{Representations of GL(inf)}}

\noarticle
In this section, we develop what we require of the algebraic representation theory of $\GL(\infty)$.  The most important results of the section are Proposition~\pref{prop:glmaxmag} and its consequences.  These results can be deduced from those in \cite[\S 2]{penkovstyrkas}, however, our proofs are different and shorter (though we prove less than \cite[\S 2]{penkovstyrkas}).

\article[Algebraic representations]
\label{art:repGL}
Let $\bV_*=\bigcup_{d \ge 1} \bC^d{}^*$ be the {\bf restricted dual} of $\bV$, where $\bC^d{}^*$ is the subspace of $(\bC^{d+1})^*$ which kills the final basis vector.  Let $T_{n,m}=\bV^{\otimes n} \otimes \bV_*^{\otimes m}$.  We say that a representation of $\GL(\infty)$ is {\bf algebraic} if it appears as a subquotient of a finite direct sum of the $T_{n,m}$.  This definition is somewhat ad hoc; see \cite[\S\S 3, 4]{koszulcategory} for a more natural characterization of these representations.  We denote by $\Rep(\GL)$ the category of algebraic representations of $\GL(\infty)$.  It is an abelian category and stable under tensor products.  Two remarks:
\begin{enumerate}
\item As discussed in \pref{intro:ss}, the category $\Rep(\GL)$ is not semi-simple:  the pairing $\bV \otimes \bV_* \to \bC$ is a non-split surjection.
\item As in \pref{gl:center}, there is a ``central $\bG_m$'' that acts on every algebraic representation; in other words, every algebraic representation admits a canonical $\bZ$-grading.  The representation $T_{n,m}$ is concentrated in degree $n-m$.
\end{enumerate}

\article[Weyl's construction (finite case)] \label{art:glweylfin}
Before we begin our study of $\Rep(\GL)$, we recall the traceless tensor construction of the irreducible representations of $\GL(d)$.  Let $T_{n,m}^d=(\bC^d)^{\otimes n} \otimes (\bC^d{}^*)^{\otimes m}$.  We note that $S_n \times S_m \times \GL(d)$ acts on $T_{n,m}^d$.  For integers $1 \le i \le n$ and $1 \le j \le m$, we obtain a map
\begin{displaymath}
t_{i,j} \colon T_{n,m}^d \to T_{n-1,m-1}^d
\end{displaymath}
by applying the pairing $\bC^d \otimes \bC^d{}^* \to \bC$ to the $i$th $\bC^d$ factor and $j$th $\bC^d{}^*$ factor.  We let $T_{[n,m]}^d$ denote the intersection of the kernels of the maps $t_{i,j}$.  If $n=0$ or $m=0$ then $T_{[n,m]}^d=T_{n,m}^d$.  This is clearly stable under the action of $S_n \times S_m \times \GL(d)$.  For a partition $\lambda$ of $n$ and $\lambda'$ of $m$, we put
\begin{displaymath}
V_{\lambda,\lambda'}^d=\Hom_{S_n \times S_m}(\bM_{\lambda} \boxtimes \bM_{\lambda'}, T_{[n,m]}^d).
\end{displaymath}
This space carries an action of $\GL(d)$.  We have the following fundamental result, which is an analogue of Weyl's construction for the irreducible representations of the classical groups (see \cite[Theorem 1.1]{koike}):

\begin{proposition}
Let $r=\ell(\lambda)$ and $s=\ell(\lambda')$.  If $r+s \le d$ then $V_{\lambda,\lambda'}$ is the irreducible representation of $\GL(d)$ with highest weight $(\lambda_1, \ldots, \lambda_r, 0, \ldots, 0, -\lambda'_s, \ldots, -\lambda'_1)$.  If $r+s>d$ then $V_{\lambda,\lambda'}=0$.
\end{proposition}

\article[Weyl's construction (infinite case)] \label{art:glweylinf}
Much of the discussion of the previous paragraph carries over to the $d=\infty$ case.  Let $t_{i,j} \colon T_{n,m} \to T_{n-1,m-1}$ and $T_{[n,m]}$ be defined as before.  Then $T_{[n,m]}$ is stable under the natural action of $S_n \times S_m \times \GL(\infty)$ on $T_{n,m}$.  For partitions $\lambda$ of $n$ and $\lambda'$ of $m$, put
\begin{displaymath}
V_{\lambda,\lambda'}=\Hom_{S_n \times S_m}(\bM_{\lambda} \boxtimes \bM_{\lambda'}, T_{[n,m]}).
\end{displaymath}
We have the decomposition
\begin{equation}
\label{gl:decomp}
T_{[n,m]}=\bigoplus_{\vert \lambda \vert=n,\ \vert \lambda' \vert=m} \bM_{\lambda} \boxtimes \bM_{\lambda'} \boxtimes V_{\lambda,\lambda'}
\end{equation}
Note that we also have an exact sequence
\begin{equation}
\label{gl:seq}
0 \to T_{[n,m]} \to T_{n,m} \to (T_{n-1,m-1})^{\bigoplus nm},
\end{equation}
where the right map is made of the $nm$ trace maps.

\begin{Proposition}
\label{gl:simple}
The $V_{\lambda,\lambda'}$ constitute a complete irredundant set of simple objects of $\Rep(\GL)$.
\end{Proposition}

\begin{proof}
Since
\begin{displaymath}
V_{\lambda,\lambda'}=\bigcup_{d \ge \ell(\lambda) + \ell(\lambda')} V^d_{\lambda,\lambda'}
\end{displaymath}
and each $V^d_{\lambda,\lambda'}$ is an irreducible representation of $\GL(d)$, it follows that $V_{\lambda,\lambda'}$ is an irreducible representation of $\GL(\infty)$.

From \eqref{gl:seq}, we see that every simple constituent of $T_{n,m}$ occurs as a constituent of $T_{[n,m]}$ or $T_{n-1,m-1}$.  From \eqref{gl:decomp}, we see that the simple constituents of $T_{[n,m]}$ are all of the form $V_{\lambda,\lambda'}$.  It thus follows from induction that the same is true for $T_{n,m}$.  Since every simple object of $\Rep(\GL)$ is a constituent of some $T_{n,m}$, we see that every simple is isomorphic to some $V_{\lambda,\lambda'}$.

Finally, $V_{\lambda,\lambda'}$ is not isomorphic to $V_{\mu,\mu'}$ if $(\lambda,\lambda') \ne (\mu,\mu')$ since their characters are distinct.  This is a modification of the theory of symmetric functions, see \cite[\S 2]{koike} for details.
\end{proof}

\begin{Proposition}
\label{gl:finlen}
Every object of $\Rep(\GL)$ has finite length.
\end{Proposition}

\begin{proof}
If suffices to show that $T_{n,m}$ has finite length.  The decomposition \eqref{gl:decomp} together with Proposition~\pref{gl:simple} shows that $T_{[n,m]}$ has finite length.  The sequence \eqref{gl:seq} thus shows, inductively, that $T_{n,m}$ has finite length.
\end{proof}

\begin{Proposition}
\label{gl:const}
The simple constituents of $T_{n,m}$ are those $V_{\lambda,\lambda'}$ with $\vert \lambda \vert \le n$, $\vert \lambda' \vert \le m$ and $\vert \lambda \vert-\vert \lambda' \vert=n-m$.
\end{Proposition}

\begin{proof}
Let $\sC(M)$ denote the simple constituents of an object $M$ of $\Rep(\GL)$.  The sequence \eqref{gl:seq} shows that
\begin{displaymath}
\sC(T_{n,m}) \subset \sC(T_{n-1,m-1}) \cup \sC(T_{[n,m]}).
\end{displaymath}
Since any individual trace map $t_{i,j} \colon T_{n,m} \to T_{n-1,m-1}$ is surjective, the above inclusion is an equality.  The result now follows from \eqref{gl:decomp} and an easy inductive argument.
\end{proof}

\article
\label{gl:mag}
Let $T \subset \GL(\infty)$ be the diagonal torus.  A {\bf weight} is a homomorphism $T \to \bG_m$ which only depends on finitely many coordinates, i.e., it is of the form $[a_1, a_2, \ldots] \mapsto a_1^{n_1} \cdots a_r^{n_r}$ for some integers $n_1, \ldots, n_r$.  The group of weights is isomorphic to the group of integer sequences $(a_1, a_2, \ldots)$ which are eventually zero.  An algebraic representation of $\GL(\infty)$ decomposes into weight spaces, as usual.  The {\bf magnitude} of a weight $\lambda$ is the sum of the absolute values of the $a_i$.  It is clear that the magnitude of any weight of $T_{n,m}$ is at most $n+m$.  The following result shows that this maximum is achieved for every non-zero submodule.

\begin{Proposition} \label{prop:glmaxmag}
Every non-zero submodule of $T_{n,m}$ has a weight of magnitude $n+m$.
\end{Proposition}

\begin{proof}
For $\alpha \in \bZ_{>0}^n$ and $\beta \in \bZ_{>0}^m$, put
\begin{displaymath}
e_{\alpha|\beta}=e_{\alpha_1} \otimes \cdots \otimes e_{\alpha_n} \otimes e_{\beta_1}^* \otimes \cdots \otimes e_{\beta_m}^*.
\end{displaymath}
Then the $e_{\alpha|\beta}$ form a basis of $T_{n,m}$.  We say that a basis vector ``occurs'' in a non-zero element $x$ of $T_{n,m}$ if its coefficient is non-zero in the expression of $x$ in this basis.  For an integer $k$, we let $N(k; \alpha)$ denote the number of coordinates of $\alpha$ which are equal to $k$.  We say that an integer $k$ ``occurs'' in $x$ if there is some basis vector $e_{\alpha|\beta}$ which occurs in $x$ with $N(k;\alpha)$ or $N(k;\beta)$ non-zero.

Let $M$ be a non-zero submodule of $T_{n,m}$ and let $x$ be a non-zero vector in $M$.  Let $I$ be the set of indices occurring in $x$.  We show that $M$ contains an element in the span of the $e_{\alpha|\beta}$ where no entry of $\alpha$ belongs to $I$ and every entry of $\beta$ belongs to $I$.  As such basis elements have weight of magnitude $n+m$, this will establish the result.

For a subset $J \subseteq I$, let $M_J$ denote the set of elements of $M$ which belong to the span of the $e_{\alpha|\beta}$ where the entries of $\alpha$ belong to $J \cup (\bZ_{>0} \setminus I)$ and the entries of $\beta$ belong to $I$.  We must show that $M_{\emptyset}$ is non-zero.  We show by descending induction on $|J|$ that $M_J$ is non-zero for each $J \subseteq I$.  We are given that $M_I$ is non-zero, as it contains $x$.

Let $y$ be a non-zero element of $M_J$, for some non-empty $J \subseteq I$, and let $j$ be an element of $J$.  We will show that $M_{J \setminus \{j\}}$ is non-zero.  Let $N$ be the maximum value of $N(j; \alpha)$ over those $(\alpha, \beta)$ for which $e_{\alpha|\beta}$ occurs in $y$.  If $N=0$ then $y$ already belongs to $M_{J \setminus \{j\}}$ and we are done; thus assume that $N>0$.  Let $k \in \bZ_{>0} \setminus I$ be an index not occurring in $y$.  Let $g \in \GL(\infty)$ be defined by $ge_r=e_r$ for $r \ne j$ and $ge_j=e_j+e_k$.  Then $ge_r^*=e_r^*$ for $r \ne k$ and $ge_k^*=e_k^*-e_j^*$.  In particular, if $e_{\alpha|\beta}$ occurs in $y$ then $ge_{\alpha|\beta}$ is computed by changing each $e_j$ in $e_{\alpha|\beta}$ to $e_j+e_k$ and then expanding the tensor product; the ``$\beta$ part'' remains unchanged.

Let $\lambda$ be the weight of $y$ and let $\lambda'$ be the weight obtained by subtracting $N$ from the $j$th spot and adding $N$ to the $k$th spot of $\lambda$.  We claim that the $\lambda'$-component of $gy$ is non-zero and belongs to $M_{J \setminus \{j\}}$.  Let us now prove this.  For $\alpha \in \bZ_{>0}^n$, let $\alpha'$ be obtained by changing each $j$ in $\alpha$ to $k$.  If $e_{\alpha|\beta}$ has weight $\lambda$ and $N(j;\alpha)=N$ (resp.\ $N(j; \alpha)<N$) then the $\lambda'$-component of $ge_{\alpha|\beta}$ is $e_{\alpha'|\beta}$ (resp.\ 0).  Thus if we write
\begin{displaymath}
y=\sum c_{\alpha,\beta} e_{\alpha|\beta}
\end{displaymath}
then the $\lambda'$ component of $gy$ is
\begin{displaymath}
\sum_{N(j; \alpha)=N} c_{\alpha,\beta} e_{\alpha'|\beta}.
\end{displaymath}
This is non-zero by the definition of $N$, and belongs to $M_{J \setminus \{j\}}$.  This completes the proof.
\end{proof}

\begin{Proposition}
\label{gl:nohom}
Let $M$ be a submodule of $T_{n,m}$.  Then $\Hom_{\GL}(M, T_{n+r,m+r})=0$ for $r>0$.
\end{Proposition}

\begin{proof}
The image of a non-zero map $M \to T_{n+r, m+r}$ would contain a weight of magnitude $n+m+2r$; as $M$ has no weight of this magnitude, such a map cannot exist.
\end{proof}

\begin{Proposition}
\label{gl:oneway}
We have
\begin{displaymath}
\Hom_{\GL}(V_{\lambda,\lambda'}, T_{n,m})=\begin{cases}
\bM_{\lambda} \boxtimes \bM_{\lambda'} & \textrm{if $n=\vert \lambda \vert$ and $m=\vert \lambda' \vert$} \\
0 & \textrm{otherwise.}
\end{cases}
\end{displaymath}
\end{Proposition}

\begin{proof}
Propositions~\pref{gl:const} and~\pref{gl:nohom} show that the $\Hom$ vanishes unless $n=\vert \lambda \vert$ and $m=\vert \lambda' \vert$.  From this observation and the exact sequence \eqref{gl:seq}, we see that the map
\begin{displaymath}
\Hom_{\GL}(V_{\lambda,\lambda'}, T_{[n,m]}) \to \Hom_{\GL}(V_{\lambda,\lambda'}, T_{n,m})
\end{displaymath}
is an isomorphism.  As the left group is $\bM_{\lambda} \boxtimes \bM_{\lambda'}$ by \eqref{gl:decomp}, the result follows.
\end{proof}

\article[Representations of $\SL(\infty)$]
\label{gl:sl}
Define a representation of $\SL(\infty)$ to be {\bf algebraic} if it is a subquotient of a finite direct sum of $T_{n,m}$'s.  We claim that any such representation naturally extends to a representation of $\GL(\infty)$.  To see this, let $V$ be an algebraic representation of $\SL(\infty)$.  For $z \in \bC^{\times}$, let $g_{n,z}$ be the diagonal matrix with $z$ in the $(n,n)$ entry and 1's in all other diagonal entries.  Given an element $v$ of $V$, we define $g_{n,z}v$ to be $g_{n,z}g_{m,z}^{-1} v$ for $m \gg 0$.  Since $v$ only involves a finite number of the basis vectors of $\bC^{\infty}$, this is well-defined.  One easily verifies that it defines an algebraic action of $\GL(\infty)$ on $V$.  Thus, letting $\Rep(\SL)$ be the category of algebraic representations of $\SL(\infty)$, we see that the restriction map $\Rep(\GL) \to \Rep(\SL)$ is an equivalence of categories.

\article[Pro-algebraic representations]
\label{gl:pro}
Let $\wh{\bV}=\varprojlim \bC^d$ and $\wh{\bV}_*=\varprojlim \bC^d{}^*$ and put $\wh{T}_{n,m}=\wh{\bV}^{\otimes n} \hatotimes \wh{\bV}_*^{\otimes m}$.  These are profinite vector spaces, and carry actions of $\GL(\infty)$.  We say that a representation of $\GL(\infty)$ is {\bf pro-algebraic} if it occurs as a subquotient of a finite direct sum of $\wh{T}_{n,m}$'s.  We let $\wh{\Rep}(\GL)$ denote the category of pro-algebraic representations.  ``Continuous dual'' provides an equivalence of tensor categories $\Rep(\GL)^{\op}=\wh{\Rep}(\GL)$.  Thus all our results on $\Rep(\GL)$ apply to $\wh{\Rep}(\GL)$, but with arrows turned around.  As we will see, it is sometimes more convenient to work in the pro category.

\begin{remark}
We can think of $\wh{\bV} \hatotimes \wh{\bV}_*$ as the space of all endomorphisms of $\bC^{\infty}$.  In particular, the scalar matrices provide a $\GL$-equivariant inclusion $\bC \to \wh{\bV} \hatotimes \wh{\bV}_*$.  As we saw, there is no copy of $\bC$ inside of $\bV \otimes \bV_*$.  On the other hand, while $\bV \otimes \bV_*$ admits a trace map to $\bC$, the space $\wh{\bV} \hatotimes \wh{\bV}_*$ does not, as one cannot form the trace of an infinite matrix in general.
\end{remark}

\subsection{The walled Brauer algebra and category}

\article[The monoid $\cG_{n,m}$]
\label{gl:gmonoid}
Let $\cV_{n,m}$ be the set of vertices $\{x_i, y_j, x'_i, y'_j\}$ with $1 \le i \le n$ and $1 \le j \le m$.  We imagine the $x_i$ and $y_j$ in one row and the $x'_i$ and $y'_j$ in a parallel row below the previous row.  We also imagine a wall dividing the $x$'s from the $y$'s.  We call an edge between two vertices ``horizontal'' if it stays within the same row and ``vertical'' if it goes between the two rows.

Let $\cG_{n,m}$ be the set of graphs $\Gamma$ on the vertex set $\cV_{n,m}$ with the following properties:  (a) $\Gamma$ is a complete matching, i.e., every vertex has valence 1; (b) no vertical edge crosses the wall; (c) every horizontal edge crosses the wall.  We give $\cG_{n,m}$ the structure of a monoid, as follows.  Let $\Gamma$ and $\Gamma'$ be two elements of $\cG_{n,m}$.  Put $\Gamma$ above $\Gamma'$, i.e., identify the bottom row of vertices of $\Gamma$ with the top row of vertices of $\Gamma'$.  In doing so, there might be some components of the resulting graph which only touch vertices in the middle row.  We write $n(\Gamma, \Gamma')$ for the number of such components; this number will be important in a moment, but for now we simply discard these components and ignore the middle vertices.  The resulting graph is the composition $\Gamma \Gamma'$.

The identity element of $\cG_{n,m}$ is the graph with all edges perfectly vertical, i.e., $x_i$ connected to $x'_i$ and $y_j$ to $y'_j$.  The group $S_n \times S_m$ is embedded in $\cG_{n,m}$ as the set of graphs with only vertical edges.  Explicitly, the pair of permutations $(\sigma, \tau)$ corresponds to the graph which has edges from $x_i$ to $x'_{\sigma(i)}$ and $y_j$ to $y'_{\tau(j)}$.

\article
\label{wb:ex}
We now give an example of composition in $\cG_{n,m}$.  Suppose $n=m=3$.  Let $\Gamma$ be the graph
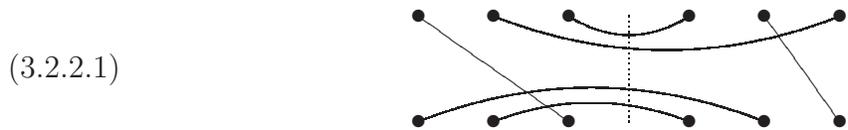
\begin{equation}
\label{wb:ex:1}
\begin{xy}
(-28, 7)*{}="A1"; (-18, 7)*{}="B1"; (-8, 7)*{}="C1"; (8, 7)*{}="D1"; (18, 7)*{}="E1"; (28, 7)*{}="F1";
(-28, -7)*{}="A2"; (-18, -7)*{}="B2"; (-8, -7)*{}="C2"; (8, -7)*{}="D2"; (18, -7)*{}="E2"; (28, -7)*{}="F2";
"A1"*{\bullet}; "B1"*{\bullet}; "C1"*{\bullet}; "D1"*{\bullet}; "E1"*{\bullet}; "F1"*{\bullet};
"A2"*{\bullet}; "B2"*{\bullet}; "C2"*{\bullet}; "D2"*{\bullet}; "E2"*{\bullet}; "F2"*{\bullet};
(0, 7)*{}; (0, -7)*{}; **\dir{..};
"A1"; "C2"; **\dir{-};
"B1"; "F1"; **\crv{(5, -2)};
"C1"; "D1"; **\crv{(0, 2)};
"E1"; "F2"; **\dir{-};
"A2"; "E2"; **\crv{(-5, 2)};
"B2"; "D2"; **\crv{(-5, -2)};
\end{xy}
\end{equation}
and let $\Gamma'$ be the graph
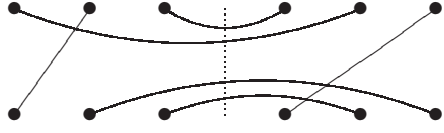
\begin{displaymath}
\begin{xy}
(-28, 7)*{}="A1"; (-18, 7)*{}="B1"; (-8, 7)*{}="C1"; (8, 7)*{}="D1"; (18, 7)*{}="E1"; (28, 7)*{}="F1";
(-28, -7)*{}="A2"; (-18, -7)*{}="B2"; (-8, -7)*{}="C2"; (8, -7)*{}="D2"; (18, -7)*{}="E2"; (28, -7)*{}="F2";
"A1"*{\bullet}; "B1"*{\bullet}; "C1"*{\bullet}; "D1"*{\bullet}; "E1"*{\bullet}; "F1"*{\bullet};
"A2"*{\bullet}; "B2"*{\bullet}; "C2"*{\bullet}; "D2"*{\bullet}; "E2"*{\bullet}; "F2"*{\bullet};
(0, 7)*{}; (0, -7)*{}; **\dir{..};
"A1"; "E1"; **\crv{(-5, -2)};
"B1"; "A2"; **\dir{-};
"C1"; "D1"; **\crv{(0, 2)};
"F1"; "D2"; **\dir{-};
"B2"; "F2"; **\crv{(5, 2)};
"C2"; "E2"; **\crv{(5, -2)};
\end{xy}
\end{displaymath}
The dotted line denotes the wall.  After putting $\Gamma$ above $\Gamma'$, we obtain the graph
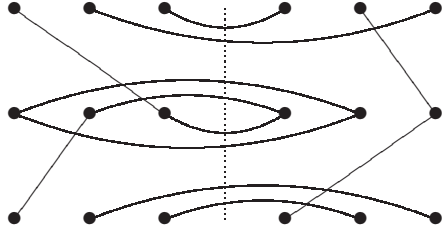
\begin{displaymath}
\begin{xy}
(-28, 14)*{}="A1"; (-18, 14)*{}="B1"; (-8, 14)*{}="C1"; (8, 14)*{}="D1"; (18, 14)*{}="E1"; (28, 14)*{}="F1";
(-28, 0)*{}="A2"; (-18, 0)*{}="B2"; (-8, 0)*{}="C2"; (8, 0)*{}="D2"; (18, 0)*{}="E2"; (28, 0)*{}="F2";
(-28, -14)*{}="A3"; (-18, -14)*{}="B3"; (-8, -14)*{}="C3"; (8, -14)*{}="D3"; (18, -14)*{}="E3"; (28, -14)*{}="F3";
"A1"*{\bullet}; "B1"*{\bullet}; "C1"*{\bullet}; "D1"*{\bullet}; "E1"*{\bullet}; "F1"*{\bullet};
"A2"*{\bullet}; "B2"*{\bullet}; "C2"*{\bullet}; "D2"*{\bullet}; "E2"*{\bullet}; "F2"*{\bullet};
"A3"*{\bullet}; "B3"*{\bullet}; "C3"*{\bullet}; "D3"*{\bullet}; "E3"*{\bullet}; "F3"*{\bullet};
(0, 14)*{}; (0, -14)*{}; **\dir{..};
"A1"; "C2"; **\dir{-};
"B1"; "F1"; **\crv{(5, 5)};
"C1"; "D1"; **\crv{(0, 9)};
"E1"; "F2"; **\dir{-};
"A2"; "E2"; **\crv{(-5, 9)};
"B2"; "D2"; **\crv{(-5, 5)};
"A2"; "E2"; **\crv{(-5, -9)};
"B2"; "A3"; **\dir{-};
"C2"; "D2"; **\crv{(0, -5)};
"F2"; "D3"; **\dir{-};
"B3"; "F3"; **\crv{(5, -5)};
"C3"; "E3"; **\crv{(5, -9)};
\end{xy}
\end{displaymath}
There is one component that only touches middle vertices, and so $n(\Gamma, \Gamma')=1$.  Discarding it and ignoring the middle vertices, we are left with
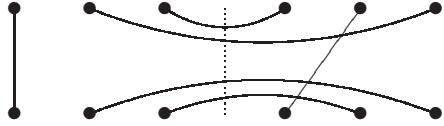
\begin{displaymath}
\begin{xy}
(-28, 7)*{}="A1"; (-18, 7)*{}="B1"; (-8, 7)*{}="C1"; (8, 7)*{}="D1"; (18, 7)*{}="E1"; (28, 7)*{}="F1";
(-28, -7)*{}="A2"; (-18, -7)*{}="B2"; (-8, -7)*{}="C2"; (8, -7)*{}="D2"; (18, -7)*{}="E2"; (28, -7)*{}="F2";
"A1"*{\bullet}; "B1"*{\bullet}; "C1"*{\bullet}; "D1"*{\bullet}; "E1"*{\bullet}; "F1"*{\bullet};
"A2"*{\bullet}; "B2"*{\bullet}; "C2"*{\bullet}; "D2"*{\bullet}; "E2"*{\bullet}; "F2"*{\bullet};
(0, 7)*{}; (0, -7)*{}; **\dir{..};
"A1"; "A2"; **\dir{-};
"B1"; "F1"; **\crv{(5, -2)};
"C1"; "D1"; **\crv{(0, 2)};
"E1"; "D2"; **\dir{-};
"B2"; "F2"; **\crv{(5, 2)};
"C2"; "E2"; **\crv{(5, -2)};
\end{xy}
\end{displaymath}
and this is $\Gamma \Gamma'$.

\article[The algebra $\cB_{n,m}$]
\label{gl:walgebra}
Let $\cB_{n,m}$ be the free $\bC[t]$-module spanned by $\cG_{n,m}$.  We write $X_{\Gamma}$ for the element of $\cB_{n,m}$ corresponding to $\Gamma \in \cG_{n,m}$.  We give $\cB_{n,m}$ the structure of an algebra by defining $X_{\Gamma} X_{\Gamma'}$ to be $t^{n(\Gamma, \Gamma')} X_{\Gamma \Gamma'}$.  For a number $\alpha \in \bC$, we let $\cB_{n,m}(\alpha)$ be defined similarly to $\cB_{n,m}$, but with $\alpha$ in place of $t$; of course, $\cB_{n,m}(\alpha)$ is just the quotient of $\cB_{n,m}$ by the two-sided ideal generated by $t-\alpha$.  These algebras are the {\bf walled Brauer algebras}, and were introduced in \cite{walledbrauer} (see also \cite[Lemma 1.2]{koike}).

\article[The $\cB_{n,m}(d)$-module $T_{n,m}^d$]
\label{gl:wmodule}
We now give $T_{n,m}^d=(\bC^d)^{\otimes n} \otimes (\bC^d{}^*)^{\otimes m}$ the structure of a $\cB_{n,m}(d)$-module.  We first introduce some notation.  For a vector $u$ in $\bC^d$, we write $f_i(u)$ for the same vector regarded in the $i$th tensor slot of $(\bC^d)^{\otimes n}$; similarly, for $u$ in $\bC^d{}^*$, we write $f'_j(u)$ for the same vector regarded in the $j$th tensor slot of $(\bC^d{}^*)^{\otimes m}$.  This notation does not really make sense on its own, but only when in an expression involving a product over all $i$ or $j$.  For example, if $n=3$ then $f_2(u) f_3(v) f_1(w)$ means $w \otimes u \otimes v$.

We now define the module structure.  Thus let $\Gamma$ be an element of $\cG_{n,m}$ and let $v$ be an element of $T_{n,m}^d$.  We assume $v$ is a pure tensor, and write $v=v_1 \otimes \cdots \otimes v_n \otimes v'_1 \otimes \cdots v'_m$.  The element $w=X_{\Gamma} v$ is defined as a product over the edges of $\Gamma$, so we just have to describe the contribution of each edge.
\begin{itemize}
\item The vertical edge $(x_i, x'_j)$ contributes $f_j(v_i)$.
\item The vertical edge $(y_i, y'_j)$ contributes $f'_j(v'_i)$.
\item The horizontal edge $(x_i, y_j)$ contributes the scalar factor $\langle v_i, v'_j \rangle$.
\item The horizontal edge $(x'_i, y'_j)$ contributes $(f_i \otimes f'_j)(\id)$, where $\id \in \bC^d \otimes \bC^d{}^*$ is the identity element.
\end{itemize}
We leave it to the reader to verify that this defines an action; we remark that the reason it is important to use the parameter $t=d$ is that the trace of the identity endomorphism on $\bC^d$ is $d$.

\article
We now give an example of the procedure described in the previous paragraph.  Suppose $\Gamma$ is the graph in \eqref{wb:ex:1}.  Then
\begin{displaymath}
X_{\Gamma}(v_1 \otimes v_2 \otimes v_3 \otimes v'_1 \otimes v'_2 \otimes v'_3)=
\langle v_2, v'_3 \rangle \langle v_3, v'_1 \rangle (f_1 \otimes f'_2)(\id) (f_2 \otimes f'_1)(\id) f_3(v_1) f'_3(v'_2)
\end{displaymath}

\article
The action of $\cB_{n,m}(d)$ on $T_{n,m}^d$ obviously commutes with that of $\GL(d)$. The following is the main result on how these actions relate.  See \cite[Theorem 5.8]{walledbrauer} and the discussion following it for details.

\begin{theorem}[Benkart et al.]
The natural map $\cB_{n,m}(d) \to \End_{\GL(d)}(T_{n,m}^d)$ is surjective, and is an isomorphism when $d \ge n+m$.
\end{theorem}

\article
We now wish to apply the theory of the walled Brauer algebra to the infinite case, and obtain an equivalence of categories analogous to that given in \pref{art:schurweyl}.  However, there is a problem:  the walled Brauer algebra does not naturally act on $T_{n,m}$.  The reason for this is that a horizontal edge in the bottom row is supposed to act by inserting the invariant of $\bV \otimes \bV_*$ into the appropriate pair of tensor factors, but this space does not have any invariant (see \pref{intro:ss}).  Our solution to this problem is simple:  we disallow graphs that have horizontal edges in the bottom row.  We do allow horizontal edges in the top row, and therefore allow the two rows to have different cardinalities.  Rather than try to form this structure into a single algebra, we find it more natural to represent it as a category, with the graphs playing the role of morphisms.  The spaces $T_{n,m}$, for all $n$ and $m$ at once, will form a representation of this category.

\article
\label{dwb:def}
The {\bf downwards walled Brauer category}, denoted $\dwb$, is the following category:
\begin{itemize}
\item The objects are {\bf bisets}, i.e., pairs $L=(L_+, L_-)$ where $L_+$ and $L_-$ are finite sets.
\item A morphism $L \to L'$ is a pair $(\Gamma, f)$, where $\Gamma$ is a bipartite matching on $L$ and $f$ is an isomorphism of bisets $L \setminus V(\Gamma) \to L'$.  Thus $\Gamma$ is a graph whose vertices $V(\Gamma)$ are elements of $L$ and all edges go between $L_+$ and $L_-$.
\item Suppose $(\Gamma, f)$ defines a morphism $L \to L'$ and $(\Delta, g)$ defines a morphism $L' \to L''$.  The composition is the pair $(\Phi, h)$ defined as follows.  The matching $\Phi$ is the union of $\Gamma$ and $f^{-1}(\Delta)$.  The bijection $h$ is the composition of the bijection $L \setminus V(\Phi) \to L' \setminus V(\Delta)$ induced by $f$ with $g$.
\end{itemize}
If $L \to L'$ is a morphism in $\dwb$ then $\# L' \le \# L$ (hence ``downwards''), with equality if and only if the morphism is an isomorphism.  The automorphism group of $L$ in $\dwb$ is the product of symmetric groups $\Aut(L_+) \times \Aut(L_-)$.  Disjoint union defines a symmetric monoidal functor $\amalg$ on $\dwb$.  We let $\otimes=\otimes_{\#}$ be the resulting convolution tensor product on $\Mod_{\dwb}$ (see \pref{diag:conv}).

\article
\label{uwb:def}
There is also an {\bf upwards walled Brauer category}, denoted $\uwb$.  Its definition is the same, except that a morphism $L \to L'$ is defined by a pair $(\Gamma, f)$ where $\Gamma$ is a bipartite matching on $L'$ and $f$ is a bijection $L \to L' \setminus V(\Gamma)$.  As with $\dwb$, we have a monoidal functor $\amalg$ on $\uwb$.  We let $\otimes=\otimes_*$; note that this is reversed from $\dwb$.  There is an obvious equivalence $\uwb=\dwb^{\op}$.  The resulting equivalence $\Mod_{\uwb}^{\fin}=(\Mod_{\dwb}^{\fin})^{\op}$ is one of tensor categories.

\article
\label{gl:ker}
Given an object $L$ of $\dwb$, put
\begin{displaymath}
\cK_L=\bV^{\otimes L_+} \otimes \bV_*^{\otimes L_-}.
\end{displaymath}
Given a morphism $L \to L'$ in $\dwb$ represented by $(\Gamma, f)$, we obtain a morphism
\begin{displaymath}
\cK_L=\left( \bigotimes_{(x, y) \in E(\Gamma)} \bV^{\otimes \{x\}} \otimes \bV_*^{\otimes \{y\}} \right) \otimes \cK_{L \setminus V(\Gamma)} \to \cK_{L'},
\end{displaymath}
where we apply the pairing $\bV \otimes \bV_* \to \bC$ to the factors in the parentheses and use $f$ to identify $\cK_{L \setminus V(\Gamma)}$ with $\cK_{L'}$.  As $\cK_L$ belongs to $\Rep(\GL)$ and the morphisms $\cK_L \to \cK_{L'}$ are $\GL$-linear, we have thus defined an object $\cK$ of $\Rep(\GL)^{\dwb}$.

\begin{Theorem}
\label{gl:dwb}
The functors of \pref{ker-func} associated to the kernel $\cK$ provide contravariant mutually quasi-inverse equivalences of tensor categories between $\Rep(\GL)$ and $\Mod_{\dwb}^{\fin}$.
\end{Theorem}

\begin{proof}
To show that $\Phi$ and $\Psi$ and mutually quasi-inverse, we apply Theorem~\pref{thm:ker-equiv}.  The first hypothesis follows from the fact that $V_{\lambda,\lambda'}$ is simple; note that $\cK_{[L, L']}$ is isomorphic to $T_{[n,m]}$ if $n=\#L$ and $m=\#L'$.  The second hypothesis follows from Proposition~\pref{gl:oneway}.  It is clear that $\amalg^*\cK$ is naturally identified with $\cK \boxtimes \cK$, i.e., $\cK$ is a tensor kernel, and so Proposition~\pref{ker-tens} implies that $\Phi$ is a tensor functor.
\end{proof}

\begin{remark}
This result is closely related to \cite[Corollary 5.2]{koszulcategory}.  Specifically, we can think of $\dwb$ as a locally finite quiver with relations, and the path algebra of this quiver is isomorphic to the algebra $\cA_{\fsl_\infty}$ in \cite[\S 5]{koszulcategory}.  Note, however, that \cite[Corollary 5.2]{koszulcategory} does not describe the tensor product from this point of view.
\end{remark}

\begin{Corollary}
\label{gl:uwb}
The tensor categories $\Rep(\GL)$ and $\Mod_{\uwb}^{\fin}$ are equivalent.
\end{Corollary}

This comes from the identification $(\Mod_{\dwb}^{\fin})^{\op}=\Mod_{\uwb}^{\fin}$.  A direct equivalence $\Mod_{\uwb}^{\fin} \to \Rep(\GL)$ is given by $M \mapsto M \otimes^{\dwb} \cK$ (see \pref{diag:tens} for notation). 

\begin{Corollary}
The tensor categories $\wh{\Rep}(\GL)$ and $\Mod_{\dwb}^{\fin}$ are equivalent.
\end{Corollary}

This comes from the identification $\wh{\Rep}(\GL)=\Rep(\GL)^{\op}$ provided by the continuous dual.

\article[Classification of injectives] \label{art:gl-injectives}
As an application of Theorem~\pref{gl:dwb}, we use \pref{art:projinj} to obtain a description of the injective objects of $\Rep(\GL)$, which recovers \cite[Corollary~4.6]{koszulcategory}.

\begin{proposition}
The Schur functor $\bS_{\lambda}(\bV) \otimes \bS_{\lambda'}(\bV_*)$ is the injective envelope of the simple module $V_{\lambda, \lambda'}$.  The representations $\bS_{\lambda}(\bV) \otimes \bS_{\lambda'}(\bV_*)$ form a complete irredundant set of indecomposable injectives.
\end{proposition}

\begin{proof}
Let $n=\vert \lambda \vert$, let $m=\vert \lambda' \vert$, let $G=S_n \times S_m$, let $i \colon \rB G \to \dwb$ be the inclusion at the object $x = (\ul{n}, \ul{m})$ and let $V=\bM_{\lambda} \boxtimes \bM_{\lambda'}$ be the irreducible representation of $G$ indexed by $(\lambda, \lambda')$.  We then have the simple object $S_x(V)$ of $\Mod_{\dwb}$ and its projective cover $P_x(V)=i_{\#}(V)$, see \pref{art:projinj}.  We have $\Phi(S_x(V))=V_{\lambda,\lambda'}$ by \pref{gl:dwb}, where $\Phi$ is as in Theorem~\pref{gl:dwb}.  It follows that $\Phi(P_x(V))$ is the injective envelope of $V_{\lambda,\lambda'}$ (note that $\Phi$ is contravariant).  We have
\[
\Phi(i_\#(V)) = \hom_\dwb(i_\#(V), \cK) = \hom_G(V, i^*\cK) = \hom_G(V, \bV^{\otimes n} \otimes \bV_*^{\otimes m}) = \bS_\lambda(\bV) \otimes \bS_{\lambda'}(\bV_*).
\]
In the second equality we used the adjunction between $i_\#$ and $i^*$, and in the fourth equality we used the construction of Schur functors from \pref{art:schurfunctor-inf}.
\end{proof}

\subsection{\texorpdfstring{Modules over $\Sym(\bC\langle 1, 1 \rangle)$}{Modules over Sym(C<1,1>)}}

\article
\label{gl:amod}
For a biset $L$, let $\cG_L$ be the set of bipartite perfect matchings on $L$.  Then $\cG$ is a (2-variable) tc monoid.  The category $\Lambda$ associated to $\cG$ in \pref{sg:def} is exactly $\uwb$.  The (2-variable) tca $A$ associated to $\cG$ is $\Sym(\bC\langle 1, 1 \rangle)$.  For clarity, let $\bE$ and $\bE'$ be two copies of $\bC^{\infty}$, so that we can identify $A$ with $\Sym(\bE \otimes \bE')$.

\begin{theorem}
We have an equivalence of tensor categories between $\Rep(\GL)$ and $\Mod_A^{\fin}$, where the latter is endowed with the tensor product $\ast^A$.
\end{theorem}

\begin{proof}
From Proposition~\pref{sg:equiv}, we have an equivalence of categories $\Mod_{\uwb}=\Mod_A$, under which $\otimes$ corresponds to $\ast^A$.  The result follows from Corollary~\pref{gl:uwb}.
\end{proof}

\begin{remark}
The category $\wh{\Rep}(\GL)$ is equivalent to $\CoMod_A^{\fin}$.
\end{remark}

\article
We now explain how to construct the equivalence between $\Mod_A^{\fin}$ and $\Rep(\GL)$ directly.  Let $U=\bE \otimes \bE'$, so that $A=\Sym(U)$, and put $B=\Sym((\bE \otimes \bV) \oplus (\bE' \otimes \bV_*))$.  We regard $A$ and $B$ as both algebras and coalgebras.  We have a natural linear map $B \to U$ given by
\begin{displaymath}
B = \Sym^2((\bE \otimes \bV) \oplus (\bE' \otimes \bV_*)) \to (\bE \otimes \bV) \otimes (\bE' \otimes \bV_*) \to \bE \otimes \bE' = U,
\end{displaymath}
where the final map makes use of the map $\bV \otimes \bV_* \to \bC$.  This induces a coalgebra homomorphism $B \to A$ which is $\GL(\bE) \times \GL(\bE')$ equivariant, and gives $B$ the structure of an $A$-comodule.  Note that if $M$ is an $A$-module then $M^{\vee}$ is an $A$-comodule.  We thus obtain functors
\begin{displaymath}
\Phi \colon \Mod^{\fin}_A \to \Rep(\GL), \qquad M \mapsto \Hom_A(M^{\vee}, B)
\end{displaymath}
and
\begin{displaymath}
\Psi \colon \Rep(\GL) \to \Mod^{\fin}_A, \qquad V \mapsto \Hom_{\GL}(V, B)^{\vee}
\end{displaymath}
(It is not immediately clear that $\Psi$ takes values in finite length modules, but this is indeed the case.)  These functors are the equivalences between $\Mod^{\fin}_A$ and $\Rep(\GL)$ that we have already constructed.  To see this, one must simply show that the object $B$ of $\CoMod_A$ corresponds to the kernel $\cK$ in $\Rep(\GL)^{\dwb}$ of \pref{gl:ker}. We omit the details.

\article[Computation of $\Ext$ groups]
We now give an application of Theorem~\pref{gl:amod}:  the computation of the $\Ext$ groups between simple objects of $\Rep(\GL)$.  This recovers \cite[Corollary~6.5]{koszulcategory}.

\begin{proposition}
We have
\begin{displaymath}
\dim \Ext^i_{\GL}(V_{\mu,\mu'}, V_{\lambda,\lambda'})=\sum_{\vert \nu \vert=i} c^{\lambda}_{\mu,\nu} c^{\lambda'}_{\mu',\nu^{\dag}}.
\end{displaymath}
In particular, this $\Ext$ group vanishes unless $i=\vert \lambda \vert-\vert \mu \vert=\vert \lambda' \vert-\vert \mu' \vert$.
\end{proposition}

\begin{proof}
Under the equivalence of Theorem~\pref{gl:amod}, the $\Ext$ group in the statement of the proposition corresponds to
\begin{displaymath}
\Ext^i_A(\bS_{\mu}(\bE) \otimes \bS_{\mu'}(\bE'), \bS_{\lambda}(\bE) \otimes \bS_{\lambda'}(\bE')),
\end{displaymath}
where this $\Ext$ is taken in the category $\Mod_A$.  According to Proposition~\pref{tca:ext} (or, rather, its generalization to the 2-variable setting), this coincides with
\begin{displaymath}
\Hom_{\GL(\bE) \times \GL(\bE')}(\bS_{\mu}(\bE) \otimes \bS_{\mu'}(\bE') \otimes \lw^i(U), \bS_{\lambda}(\bE) \otimes \bS_{\lambda'}(\bE')).
\end{displaymath}
Now, we have the Cauchy formula (see \cite[(6.2.8)]{expos})
\begin{displaymath}
\lw^i(U)=\lw^i(\bE \otimes \bE')=\bigoplus_{\nu} \bS_{\nu}(\bE) \otimes \bS_{\nu^{\dag}}(\bE').
\end{displaymath}
Combining this with the previous identity completes the proof.
\end{proof}

\begin{remark}
An interesting point in this proof is that it uses projective $A$-modules, which have infinite length and thus do not appear in the category $\Rep(\GL)$.
\end{remark}

\article[Classification of blocks] \label{art:gl-blocks}
As a corollary, we get an easy description of the block structure of $\Rep(\GL)$, which recovers \cite[Corollary~6.6]{koszulcategory}:

\begin{corollary}
Two simples $V_{\lambda, \lambda'}$ and $V_{\mu, \mu'}$ belong to the same block of the category $\Rep(\GL)$ if and only if $|\lambda| - |\lambda| = |\mu| - |\mu'|$.
\end{corollary}

\begin{remark}
The blocks are naturally indexed by $\bZ$: given an integer $d$, the corresponding block $\Rep(\GL)_d$ is the one containing simples $V_{\lambda,\lambda'}$ for which $|\lambda| - |\lambda'| = d$. It is easy to see that there are equivalences $\Rep(\GL)_d \simeq \Rep(\GL)_{-d}$ and that no other non-trivial equivalences exist amongst the blocks (this is a statement about the combinatorics of the underlying quivers).
\end{remark}

\article[Littlewood varieties] \label{art:gl-lw}
We now give a second application of Theorem~\pref{gl:amod}:  the construction of injective resolutions of simple objects of $\Rep(\GL)$.  For this it will actually be more natural to use the pro version $\wh{\Rep}(\GL)$ of the category (see \pref{gl:pro}).  We write $\wh{V}_{\lambda,\lambda'}$ for the simple object of this category corresponding to $(\lambda, \lambda')$.  Let $E$ and $E'$ be finite dimensional vector spaces, let $U=E \otimes E'$, let $A=\Sym(U)$ and let $B=\Sym((E \otimes \wh{\bV}) \oplus (E' \otimes \wh{\bV}_*))$, where each symmetric power is taken in $\wh{\Vec}$. The space $U$ is a subspace of $B$, and so we have an algebra homomorphism $A \to B$.  We let $C$ be the quotient of $B$ by the ideal generated by $U$.  We have maps
\begin{displaymath}
\Spec(C) \to \Spec(B) \to \Spec(A).
\end{displaymath}
To give a geometric interpretation of these maps, we ignore subtleties caused by spaces being infinite dimensional.  The space $\Spec(A)$ is identified with the space of forms $E \otimes E' \to \bC$, while $\Spec(B)$ is identified with the space $\Hom(E, \bV) \times \Hom(E', \bV_*)$ of pairs of maps $(\varphi \colon E \to \bV, \psi \colon E' \to \bV_*)$.  The map $\Spec(B) \to \Spec(A)$ takes a pair of linear maps $(\varphi, \psi)$ to the form $(\varphi \otimes \psi)^* \bomega$, where $\bomega$ is the natural pairing on $\bV \otimes \bV_*$.  The space $\Spec(C)$, which we call the {\bf Littlewood variety}, is the scheme-theoretic fiber of this map above 0, i.e., it consists of those pairs $(\varphi, \psi)$ for which $(\varphi \otimes \psi)^* \bomega=0$.  Alternatively, we can think of this as the variety defined by the condition $\psi^* \phi = 0$. Let $K_{\bullet}=B \otimes \lw^{\bullet}(U)$ be the Koszul complex of the Littlewood variety.

\begin{remark}
We would like to take $E$ and $E'$ to be $\bC^{\infty}$, but this presents the technical annoyance of mixing ind-finite and pro-finite vector spaces, which we prefer to avoid.
\end{remark}

\begin{Proposition}
\label{gl:koszul}
The augmented complex $K_{\bullet} \to C$ is exact. We have a decomposition
\begin{displaymath}
C=\bigoplus_{\lambda,\lambda'} \bS_{\lambda}(E) \otimes \bS_{\lambda'}(E') \otimes \wh{V}_{\lambda,\lambda'}.
\end{displaymath}
\end{Proposition}

\begin{proof}
See \cite[\S 5.3]{ssw}.
\end{proof}

\article[Littlewood complexes] \label{art:littlewoodcx}
We now give the projective resolutions of simple objects.  For a pair of partitions $(\lambda, \lambda')$, define the {\bf Littlewood complex} $L^{\lambda,\lambda'}_{\bullet}$ by
\begin{displaymath}
L^{\lambda,\lambda'}_{\bullet}=\Hom_{\GL(E) \times \GL(E')}(\bS_{\lambda}(E) \otimes \bS_{\lambda'}(E'), K_{\bullet}),
\end{displaymath}
where the dimensions of $E$ and $E'$ are sufficiently large (the definition is then independent of $E$ and $E'$).  Proposition~\pref{gl:koszul} shows that
\begin{displaymath}
\rH_i(L^{\lambda,\lambda'}_{\bullet})=\begin{cases}
\wh{V}_{\lambda,\lambda'} & \textrm{if $i=0$} \\
0 & \textrm{otherwise,} \end{cases}
\end{displaymath}
and so $L^{\lambda,\lambda'}_{\bullet}$ is a resolution of the simple object $\wh{V}_{\lambda,\lambda'}$.  Furthermore, it is clear that $K_{\bullet}$ is built from polynomial Schur functors applied to $\wh{\bV}$ and $\wh{\bV}_*$, and so each $K_i$ is projective in $\wh{\Rep}(\GL)$.  In fact, we have
\begin{displaymath}
L^{\lambda,\lambda'}_i=\bigoplus_{\vert \mu \vert=i} \bS_{\lambda/\mu}(\wh{\bV}) \otimes \bS_{\lambda'/\mu^{\dag}}(\wh{\bV}_*).
\end{displaymath}
Thus the Littlewood complex $L^{\lambda,\lambda'}_{\bullet}$ is a projective resolution of $\wh{V}_{\lambda,\lambda'}$; in fact, it is a minimal resolution.

\article[Transpose duality] \label{art:gl-transpose}
As an application of Theorem~\pref{gl:amod}, we can find a symmetry of $\Rep(\GL)$:

\begin{theorem}
There is a natural asymmetric auto-equivalence of tensor categories $\Rep(\GL) \to \Rep(\GL)$. This equivalence takes the simple object $V_{\lambda, \mu}$ to $V_{\lambda^\dagger, \mu^\dagger}$.
\end{theorem}

\begin{proof}
As discussed in \pref{art:V2}, full transpose is an asymmetric autoequivalence on the tensor category $\cV^{\otimes 2}$. Since it takes $\bE$ and $\bE'$ to themselves, it takes $A$ to itself and induces an equivalence of categories $\Mod_A \cong \Mod_A$. Since the tensor products $\ast^A$ are defined using only the tensor structure on $\cV^{\otimes 2}$, it maps to itself under this equivalence.
\end{proof}

\article[The Fourier transform] \label{rmk:glkoszul}
We give another application of Theorem~\pref{gl:amod}:  the construction of a (nearly canonical) derived auto-equivalence of $\Rep(\GL)$, which we call the {\bf Fourier transform}.  By the theorem, it is enough to construct such an auto-equivalence on the derived category of $\Mod_A^{\fin}$.  To do this, we use a variant of the construction of the Fourier transform on $\Sym(\bC\langle 1 \rangle)$ given in \cite[\S 6]{symc1}.  We give only the main ideas here, details will appear in \cite{sskoszul}.  Given an $A$-module $M$, the module
\[
\cT_n(M) = \bigoplus_{p \ge 0} \Tor^A_p(M, \bC)_{p+n}
\]
is naturally a comodule over $B = \lw(\bE \otimes \bE')$, and so $\cT_n(M)^\vee$ is a module over $B$. We can apply the partial transpose functor \pref{art:V2} with respect to $\bE'$. By the Cauchy identity \cite[(6.2.8)]{expos}, this turns $B$ into $A$, and the $B$-module $\cT_n(M)$ into an $A$-module $\cF_n(M)$.  In fact, $\cF_n$ is the $n$th homology of an equivalence of triangulated categories
\[
\cF \colon \rD(\Mod_A)^{\rm op} \to \rD(\Mod_A).
\]
This equivalence takes $\rD^b(\Mod_A^{\fin})$ to the category of perfect complexes in $\rD(\Mod_A)$.  Finally, the category of finitely generated projective objects of $\Mod_A$ is equivalent to the category of finite length injective objects of $\Mod_A$; this is a variant of \cite[\S 2.5]{symc1}.  We can therefore move from perfect complexes back to $\rD^b(\Mod_A^{\fin})$.

\begin{remark}
\begin{enumerate}[(a)]
\item Obviously, one could have applied transpose with respect to $\bE$ instead of $\bE'$.  Thus there are naturally two Fourier transforms, and they differ by the full transpose.

\item The Fourier transform realizes $\Rep(\GL)$ as its own Koszul dual.  The fact that $\Rep(\GL)$ is Koszul self-dual was also proved in \cite[Corollary~6.4]{koszulcategory}, though in a different way: the proof of loc.\ cit.\ identifies $\Rep(\GL)$ with modules over a quadratic ring, and then shows that this ring is Koszul self-dual by computation.
\end{enumerate}
\end{remark}

\subsection{Rational Schur functors, universal property and specialization} 
\label{sec:rationalschur}

\article[Rational Schur functors]
\label{gl:schur:intro}
Let $\cA$ be a tensor category.  Let $T(\cA)$ be the category whose objects are triples $(A, A', \omega)$, where $A$ and $A'$ are objects of $\cA$ and $\omega$ is a pairing $A \otimes A' \to \bC$, and whose morphisms are the obvious things.  We typically suppress $\omega$ from the notation.  Given $(A, A') \in T(\cA)$, define $\cK(A, A')$ to be the object of $\cA^{\dwb}$ given by $L \mapsto A^{\otimes L_+} \otimes (A')^{\otimes L_-}$.  Functoriality with respect to morphisms in $\dwb$ makes use of the pairing $\omega$, and is defined just like in \pref{gl:ker}.  For an object $M$ of $\Mod_{\uwb}^{\fin}$, define
\begin{displaymath}
S_M(A, A')=M \otimes^{\dwb} \cK(A, A')
\end{displaymath}
(see \pref{diag:tens} for notation). Then $M \mapsto S_M(A, A')$ defines a covariant functor $\Mod_{\uwb} \to \cA$ which is left-exact (see \pref{diag:tens}) and a tensor functor (since $\cK(A, A')$ is obviously a tensor kernel, see \pref{ker-tens}).  We call $S_M$ the {\bf rational Schur functor} associated to $M$.  

\begin{Theorem}
\label{gl:univ}
To give a left-exact tensor functor from $\Rep(\GL)$ to an arbitrary tensor category $\cA$ is the same as to give an object of $T(\cA)$.  More precisely, letting $\bM$ and $\bM_*$ be the objects of $\Mod_{\uwb}$ corresponding to $\bV$ and $\bV_*$, the functors
\begin{displaymath}
\Phi_\cA \colon \LEx^{\otimes}(\Mod_{\uwb}^{\fin}, \cA) \to T(\cA), \qquad F \mapsto (F(\bM), F(\bM_*))
\end{displaymath}
and
\begin{displaymath}
\Psi_\cA \colon T(\cA) \to \LEx^{\otimes}(\Mod_{\uwb}^{\fin}, \cA), \qquad (A, A') \mapsto (M \mapsto S_M(A, A'))
\end{displaymath}
are mutually quasi-inverse equivalences.
\end{Theorem}

\begin{proof}
Let $(A, A')$ be an object of $T(\cA)$.  Applying $\Psi_{\cA}$, we obtain a functor $F \colon \Mod_{\uwb}^{\fin} \to \cA$.  It is obvious from the construction of this functor that $F(\bM)=A$ and $F(\bM_*)=A'$.  Thus $\Phi_{\cA}(F) \cong (A, A')$.  This reasoning shows that the natural morphism $\id \to \Phi_{\cA} \Psi_{\cA}$ is an isomorphism.

Now suppose that $F \colon \Mod_{\uwb}^{\fin} \to \cA$ is a left-exact tensor functor.  Applying $\Phi_{\cA}$ and then $\Psi_{\cA}$, we obtain the functor $\Mod_{\uwb}^{\fin} \to \cA$ given by
\begin{displaymath}
M \mapsto M \otimes^{\dwb} \cK(F(\bM), F(\bM_*)).
\end{displaymath}
Since $F$ is a tensor functor, we have $\cK(F(\bM), F(\bM_*))=F(\cK(\bM, \bM_*))$, and since $F$ is left-exact it pulls out of $\otimes^{\dwb}$.  We thus see that
\begin{displaymath}
M \otimes^{\dwb} \cK(F(\bM), F(\bM_*)) = F(M \otimes^{\dwb} \cK(\bM, \bM_*))=F(M).
\end{displaymath}
Here we have used that $M \otimes^{\dwb} \cK(\bM, \bM_*)$ is naturally identified with $M$, which follows from the proof of Theorem~\pref{gl:dwb}.  This shows that $\Psi_{\cA}(\Phi_{\cA}(F))=F$, from which one deduces that the natural morphism $\id \to \Psi_{\cA} \Phi_{\cA}$ is an isomorphism.
\end{proof}

\begin{remark}
This theorem can be rephrased as follows:  the functor $T \colon \TCat \to \Cat$ is corepresented by $\Rep(\GL)$, with the universal object in $T(\Rep(\GL))$ being $(\bV, \bV_*)$.  See \pref{art:schur-univ} for notation.
\end{remark}

\article[The specialization functor] \label{gl:special}
The pair $(\bC^d, \bC^d{}^*)$ defines an object of $T(\Rep(\GL(d)))$, and so by Theorem~\pref{gl:univ} we obtain a left-exact tensor functor
\begin{displaymath}
\Gamma_d \colon \Rep(\GL) \to \Rep(\GL(d)),
\end{displaymath}
which we call the {\bf specialization functor}.  This functor is given by $\Gamma_d(V)=S_V(\bC^d)$, where $S_V$ is the rational Schur functor associated to $V$.  The results of \pref{art:glweylfin} describe this functor on simple objects:  $\Gamma_d(V_{\lambda,\lambda'})$ is the irreducible $V_{\lambda,\lambda'}^d$ if $d \ge \ell(\lambda)+\ell(\lambda')$ and 0 otherwise.

\article[Specialization via invariants]
We now give a more direct description of specialization.  Let $G_d$ be the subgroup of $\GL(\infty)$ which coincides with the identity matrix away from the top left $d \times d$ block, and let $H_d$ be the subgroup which coincides with the identity matrix away from the complementary diagonal block.  Of course, $G_d=\GL(d)$ and $H_d$ is isomorphic to $\GL(\infty)$.  The subgroups $G_d$ and $H_d$ commute, and so $G_d \times H_d$ is a subgroup of $\GL(\infty)$; in particular, the $H_d$-invariants of any representation of $\GL(\infty)$ form a representation of $G_d$.

\begin{proposition} \label{prop:gl_invts}
We have a natural identification $\Gamma_d(V)=V^{H_d}$.
\end{proposition}

\begin{proof}
Let $M$ be the object of $\Mod_{\uwb}^{\fin}$ corresponding to $V \in \Rep(\GL)$.  We have identifications
\begin{displaymath}
\Gamma_d(V)=M \otimes^{\dwb} \cK(\bC^d, \bC^d{}^*), \qquad
V=M \otimes^{\dwb} \cK(\bV, \bV_*).
\end{displaymath}
As $\otimes^{\dwb}$ is a limit, it commutes with the formation of invariants, and so
\begin{displaymath}
V^{H_d}=M \otimes^{\dwb} \cK(\bV, \bV_*)^{H_d},
\end{displaymath}
and so it suffices to show that $\cK(\bV, \bV_*)^{H_d}=\cK(\bC^d, \bC^d{}^*)$.  For this, it is enough to show that $T_{n,m}^{H_d}=T_{n,m}^d$.  Let $\bV'$ be the span of the $e_i$ with $i>d$, define $\bV'_*$ similarly, and let $T'_{n,m}$ be defined using $\bV'$ and $\bV'_*$.  We have $\bV=\bC^d \oplus \bV'$, and similarly for $\bV_*$.  Then
\begin{displaymath}
T_{n,m}=\bigoplus_{\substack{0 \le i \le n \\ 0 \le j \le m}} W_{i,j} \otimes (\bC^d)^{\otimes n-i} \otimes (\bC^d{}^*)^{\otimes m-j} \otimes T'_{i,j},
\end{displaymath}
where $W_{i,j}$ is a multiplicity space of dimension $\binom{n}{i} \binom{m}{j}$.  Now, $H_d$ is isomorphic to $\GL(\infty)$, and under this isomorphism $T'_{n,m}$ corresponds to $T_{n,m}$.  It follows from Proposition~\pref{gl:oneway} that $(T'_{n,m})^{H_d}=0$ unless $n=m=0$.  Applying this to the above decomposition, we see that $T_{n,m}^{H_d}=T_{n,m}^d$, which completes the proof.
\end{proof}

\begin{remark}
Since $\Gamma_d$ is a tensor functor, the above proposition shows that if $V$ and $W$ are algebraic representations of $\GL(\infty)$ then $(V \otimes W)^{H_d}=V^{H_d} \otimes W^{H_d}$.
\end{remark}

\article[Derived specialization]
As the category $\Rep(\GL)$ has enough injectives, the derived functor $\rR\Gamma_d$ of $\Gamma_d$ exists.  The injective resolution of the simple object $V_{\lambda,\lambda'}$ is given by the Littlewood complex $L^{\lambda,\lambda'}_{\bullet}$ (see \pref{art:littlewoodcx} for the dual picture).  Since specialization behaves in the obvious manner on polynomial Schur functors (see \pref{gl:special}), $\rR\Gamma_d(V_{\lambda,\lambda'})=\Gamma_d(L^{\lambda,\lambda'}_{\bullet})$ is just $L^{\lambda,\lambda'}_{\bullet}(\bC^d)$, which is by definition the result of evaluating the Schur functors in $L^{\lambda, \lambda'}_\bullet$ on $\bC^d$.  The cohomology of this complex is computed in \cite[\S 5.5]{ssw}, the result being:

\begin{theorem}
Let $(\lambda,\lambda')$ be a pair of partitions and let $d \ge 1$ be an integer.  Then $\rR^i\Gamma_d(V_{\lambda,\lambda'})$ either vanishes identically or else there exists a unique $i$ for which it is non-zero, and it is then an irreducible representation of $\GL(d)$.
\end{theorem}

Furthermore, there is a rule, the {\bf modification rule}, which calculates where the cohomology is non-zero, and what the resulting irreducible of $\GL(d)$ is.  See \cite[\S 5.4]{ssw} for details. The Euler characteristic of this complex was previously computed by \cite[Proposition 2.2]{koike}, which suggested the result of \cite{ssw}.

\article \label{art:gl-T0T1}
Let $T_0=T(\Vec^{\fin})$ be the category of triples $(V, V', \omega)$ where $V$ and $V'$ are finite dimensional vector spaces and $\omega \colon V \otimes V' \to \bC$ is a bilinear form.  There are no conditions placed on $\omega$; it could even be zero.  Let $T_1$ be the category whose objects are vector spaces and whose morphisms are split injections; that is, a morphism $V \to V'$ consists of a pair $(i, p)$ where $i \colon V \to V'$ and $p \colon V' \to V$ are linear maps with $pi=\id_V$.  There is a natural functor $T_1 \to T_0$ taking $V$ to $(V, V^*, \omega)$, where $\omega$ is the canonical pairing between $V$ and $V^*$.  In this way, $T_1$ is identified with the full subcategory of $T_0$ where the form $\omega$ is perfect.

\article
\label{gl:acat}
Let $\cA$ be the category of representations of $\GL(\infty)$ such that every element is stabilized by $H_d$ for some $d$. Define a functor
\begin{displaymath}
\sF \colon \cA \to \Fun(T_1, \Vec)
\end{displaymath}
as follows.  For $U \in \cA$ and $V \in T_1$, pick an isomorphism $V \cong \bC^d$ and put $\sF(U)(V)=U^{H_d}$.  This can be said more canonically as follows.  Let $S(V)$ be the groupoid of split injections $V \to \bC^{\infty}$.  Given such an injection, let $V' \subset \bC^{\infty}$ be the complement of $V$.  We obtain a functor $S(V) \to \Vec$ by sending a split injection to $U^{\GL(V')}$.  The space $\sF(U)$ is then canonically the limit of this functor.  It is clear that $\sF$ is left-exact.  We also define a functor
\begin{displaymath}
\sG \colon \Fun(T_1, \Vec) \to \cA
\end{displaymath}
by $\sG(F)=\varinjlim F(\bC^d)$, where the transition map $\bC^d \to \bC^{d+1}$ is the standard inclusion using the first $d$ elements of a basis (with its standard splitting).  Basic properties of direct limits show that $\sG$ is exact and respects tensor products.  There are natural maps $\sG \sF \to \id$ and $\id \to \sF \sG$, the first of which is an isomorphism, essentially by the definition of $\cA$.  We thus see that $\sF$ and $\sG$ are naturally adjoint to each other (with $\sF$ being the right adjoint).

\begin{Lemma}
The functor $\sF$ is fully faithful.
\end{Lemma}

\begin{proof}
The isomorphism $\sG \sF = \id$ shows that $\sF$ is faithful.  We claim that $\sG$ is faithful on the image of $\sF$.  To see this, suppose that $V$ and $W$ belong to $\cA$ and let $f \colon \sF(V) \to \sF(W)$ be a map in $\Fun(T_1, \Vec)$.  Let $f'=\sG(f) \colon V \to W$.  We obtain a commutative square
\begin{displaymath}
\xymatrix{
\sF(V)(\bC^d) \ar[r] \ar[d]_{f(\bC^d)} & V \ar[d]^{f'} \\
\sG(W)(\bC^d) \ar[r] & W }
\end{displaymath}
The horizontal maps are injective, since $\sF(V)(\bC^d)$ is just $V^{H_d}$.  Thus if $f'=0$ then $f(\bC^d)=0$ for all $d$, and so $f=0$.  This shows that $\sG$ is faithful on the image of $\sF$.

We now show that $\sF$ is full.  Suppose $f \colon \sF(V) \to \sF(W)$ is a map in $\Fun(T_1, \Vec)$.  Let $g=\sG(f)$.  Then $\sG(f)=\sG(\sF(g))$, since $\sG \sF=\id$, and so $\sG(f-\sF(g))=0$.  Since $\sG$ is faithful on the image of $\sF$, this gives $f=\sF(g)$, and shows that $\sF$ is full.
\end{proof}

\begin{Theorem}
\label{gl:schur}
The functor $\sF$ induces a left-exact fully faithful tensor functor $\Rep(\GL) \to \Fun(T_1, \Vec^{\fin})$.
\end{Theorem}

\begin{proof}
All that remains to be shown is that the restriction of $\sF$ to $\Rep(\GL)$ is a tensor functor and takes values in $\Fun(T_1, \Vec^{\fin})$.  This follows from Proposition~\pref{prop:gl_invts} and basic properties of the specialization functor.
\end{proof}

\article \label{gl:catc}
The above theorem shows that we can regard $\Rep(\GL)$ as a category of functors $T_1 \to \Vec^{\fin}$.  However, it is not an abelian subcategory of the functor category, since there are surjections in $\Rep(\GL)$ which are not surjections of functors (as specialization is not exact).  The abelian closure $\cC$ of the image should be an interesting category.  We now explain what we expect to be true of it.  We intend to prove these statements, and more, in a subsequent work.

Call a functor $F \colon T_1 \to \Vec^{\fin}$ {\bf algebraic} if the maps it induces on $\Hom$ spaces are maps of varieties (the $\Hom$ sets in $T_1$ and $\Vec^{\fin}$ are naturally quasi-affine varieties).  The following conditions on a functor $F$ are equivalent:  (1) $F$ belongs to $\cC$; (2) $F$ is algebraic and finitely generated; (3) $F$ is a subquotient of a finite direct sum of mixed tensor powers (functors of the form $V \mapsto V^{\otimes n} \otimes (V^*)^{\otimes m})$.  The functor $\sG$ is an exact tensor functor mapping $\cC$ to $\Rep(\GL)$; in fact, it realizes $\Rep(\GL)$ as the Serre quotient of $\cC$ by its subcategory of finite length objects.  The functor $\sF$ maps $\Rep(\GL)$ fully faithfully into $\cC$; the essential image consists of ``saturated'' objects.

Thus the pair $(\cC, \Rep(\GL))$ is very analogous to the pair $(\Mod_A, \Mod_K)$, where $A=\Sym(\bC\langle 1 \rangle)$.  We expect that more of the picture developed in \cite{symc1} will hold for the former pair.  Furthermore, one should be able to replace $T_1$ with $T_0$ and maintain some of this picture.

\section{The orthogonal and symplectic groups} \label{sec:osp}

\subsection{\texorpdfstring{Representations of $\bO(\infty)$ and $\Sp(\infty)$}{Representations of O(inf) and Sp(inf)}}

\article \label{art:Oinfdefn}
Let $\bomega$ be a non-degenerate symmetric bilinear on $\bV$ such that each $e_i$ is orthogonal to all but finitely many $e_j$.  Let $\bO(\infty)$ be the subgroup of $\GL(\infty)$ stabilizing $\bomega$. For concreteness, we define $\bomega$ by
\[
\bomega\bigg(\sum_{i \ge 1} a_i e_i, \sum_{j \ge 1} b_j e_j \bigg) = \sum_{k \ge 1} (a_{2k-1} b_{2k} + a_{2k} b_{2k-1}).
\]
We say that a representation of $\bO(\infty)$ is {\bf algebraic} if it appears as a subquotient of a finite direct sum of the spaces $T_n=\bV^{\otimes n}$.  Again, a more intrinsic definition is given in \cite[\S 4]{penkovstyrkas}.  We denote by $\Rep(\bO)$ the category of algebraic representations.  It is an abelian category and stable under tensor product.  Some remarks:
\begin{enumerate}[(a)]
\item The form $\bomega$ provides an isomorphism $\bV \to \bV_*$, which is why we do not need to consider subquotients of tensor powers involving $\bV_*$.
\item Just like $\Rep(\GL)$, the category $\Rep(\bO)$ is not semi-simple:  the surjection $\bomega \colon \Sym^2(\bV) \to \bC$ is not split.
\item There is a ``central $\mu_2$'' that acts on every algebraic representation of $\bO(\infty)$.  This is just the restriction to $\mu_2$ of the ``central $\bG_m$'' of $\GL(\infty)$, as defined in \pref{gl:center}.  This action endows every algebraic representation with a canonical $\bZ/2$ grading.  The representation $T_n$ is concentrated in degree $n \pmod{2}$.
\item Define a representation of $\SO(\infty)$ to be algebraic if it appears as a subquotient of a finite direct sum of $T_n$'s.  Just as in \pref{gl:sl}, every such representation canonically extends to an algebraic representation of $\bO(\infty)$, and so the restriction functor $\Rep(\bO) \to \Rep(\SO)$ is an equivalence.
\item One can also consider pro-algebraic representations of $\bO(\infty)$, and this leads to a category $\wh{\Rep}(\bO)$ (see \pref{gl:pro} for details).  In this category, there is a non-split injection $\bC \to \Sym^2(\bV)$.  ``Continuous dual'' provides an equivalence of $\wh{\Rep}(\bO)$ with $\Rep(\bO)^{\op}$.
\end{enumerate}

\article[Weyl's construction (finite case)]
\label{o:weylfin}
Before studying $\Rep(\bO)$, we recall Weyl's construction of the irreducible representations of $\bO(d)$.  Let $T_n^d=(\bC^d)^{\otimes n}$.  The group $S_n \times \bO(d)$ acts on $T_n^d$.  For integers $1 \le i, j \le n$ we obtain a map
\begin{displaymath}
t_{i,j} \colon T_n^d \to T_{n-2}^d
\end{displaymath}
by applying the pairing $\bomega$ to the $i$th and $j$th tensor factors.  We let $T_{[n]}^d$ denote the intersection of the kernels of the maps $t_{i,j}$.  If $n=0$ or $m=0$ then $T_{[n]}^d=T_n^d$.  The space $T_{[n]}^d$ is clearly stable under the action of $S_n \times \bO(d)$.  For a partition $\lambda$ of $n$, we put
\begin{displaymath}
V_{\lambda}^d=\Hom_{S_n}(\bM_{\lambda}, T_{[n]}^d).
\end{displaymath}
This space carries an action of $\bO(d)$, and we have the following fundamental result of Weyl (see \cite[\S 19.5]{fultonharris}):

\begin{proposition}
If the first two columns of $\lambda$ have at most $d$ boxes in total then $V_{\lambda}^d$ is an irreducible representation of $\bO(d)$.  Otherwise, $V_{\lambda}^d=0$.  
\end{proposition}

\article[Weyl's construction (infinite case)] \label{art:weyl-inf-orth}
Much of the preceding discussion carries over to the infinite case.  We let $t_{i,j} \colon T_n \to T_{n-2}$ and $T_{[n]}$ be defined as before.  Then $T_{[n]}$ is stable under $S_n \times \bO(\infty)$.  For a partition $\lambda$ of $n$, put
\begin{displaymath}
V_{\lambda}=\Hom_{S_n}(\bM_{\lambda}, T_{[n]}).
\end{displaymath}
We have the decomposition
\begin{equation}
\label{o:decomp}
T_{[n]}=\bigoplus_{\vert \lambda \vert=n} \bM_{\lambda} \boxtimes V_{\lambda}.
\end{equation}
Note that we also have an exact sequence
\begin{displaymath}
0 \to T_{[n]} \to T_n \to T_{n-2}^{\oplus n(n+1)/2}.
\end{displaymath}

\begin{Proposition}
The $V_{\lambda}$ constitute a complete irredundant set of simple objects of $\Rep(\bO)$.
\end{Proposition}

\begin{proof}
This is just like the proof of Proposition~\pref{gl:simple}. For the relevant character theory, see \cite[\S 2.1]{koiketerada}.
\end{proof}

\begin{Proposition}
\label{o:finlen}
Every object of $\Rep(\bO)$ has finite length.
\end{Proposition}

\begin{proof}
This is just like the proof of Proposition~\pref{gl:finlen}.
\end{proof}

\begin{Proposition}
\label{o:const}
The simple constituents of $T_n$ are those $V_{\lambda}$ with $\vert \lambda \vert \le n$ and $\vert \lambda \vert=n \pmod{2}$.
\end{Proposition}

\begin{proof}
This is just like the proof of Proposition~\pref{gl:const}.
\end{proof}

\article
Let $T \subset \bO(\infty)$ be the diagonal torus. Given our choice of $\bomega$ in \pref{art:Oinfdefn}, $T$ is the subgroup of diagonal matrices of the form ${\rm diag}(a_1, a_1^{-1}, a_2, a_2^{-1}, \dots)$, which we abbreviate by $[a_1, a_2, \dots]$. A {\bf weight} is a homomorphism $T \to \bG_m$ which only depends on finitely many coordinates, i.e., it is of the form $[a_1, a_2, \ldots] \mapsto a_1^{n_1} \cdots a_r^{n_r}$ for some integers $n_1, \ldots, n_r$.  The group of weights is isomorphic to the group of integer sequences $(a_1, a_2, \ldots)$ which are eventually zero.  An algebraic representation of $\bO(\infty)$ decomposes into weight spaces, as usual.  The {\bf magnitude} of a weight $\lambda$ is the sum of the absolute values of the $a_i$.  It is clear that the magnitude of any weight of $T_n$ is at most $n$. The following result shows that this maximum is achieved for every non-zero submodule.

\begin{Proposition} \label{prop:Oweight}
Every non-zero submodule of $T_n$ has a weight of magnitude $n$.
\end{Proposition}

\begin{proof}
Let $\bV'$ be the isotropic subspace of $\bV$ spanned by $e_1, e_3, e_5, \dots$.  Using the orthogonal form $\bomega$, the span of $e_2, e_4, e_6, \dots$ is identified with $\bV'_\ast$.  The representation of $\GL(\bV')$ on $\bV=\bV' \oplus \bV'_*$ realizes $\GL(\bV')$ as a subgroup of $\bO(\infty)$.  Furthermore, the diagonal torus of $\GL(\bV')$ coincides with the maximal torus $T$ of $\bO(\infty)$, and the notion of magnitude defined in \pref{gl:mag} agrees with the one defined above.  We have a decomposition of $\GL(\bV')$ representations
\[
\bV^{\otimes n} = \bigoplus_{a + b = n} W_{a,b} \otimes {\bV'}^{\otimes a} \otimes {\bV'_\ast}^{\otimes b}
\]
where $W_{a,b}$ is a multiplicity space of dimension $\binom{n}{a}$.  If $M$ is a non-zero $\bO(\infty)$-submodule of $\bV^{\otimes n}$ then it has a non-zero projection $M'$ to some ${\bV'}^{\otimes a} \otimes {\bV'_{\ast}}^{\otimes b}$.  Since $M'$ is a $\GL(\bV')$-submodule of ${\bV'}^{\otimes a} \otimes {\bV'_{\ast}}^{\otimes b}$, it has a weight of magnitude $n=a+b$ by Proposition~\pref{prop:glmaxmag}.
\end{proof}

\begin{Proposition}
\label{o:nohom}
Let $M$ be a submodule of $T_n$.  Then $\Hom_{\bO}(M, T_{n+r})=0$ for $r>0$.
\end{Proposition}

\begin{proof}
This is just like the proof of Proposition~\pref{gl:nohom}.
\end{proof}

\begin{Proposition}
\label{o:oneway}
We have
\begin{displaymath}
\Hom_{\bO}(V_{\lambda}, T_n)=\begin{cases}
\bM_{\lambda} & \textrm{if $n=\vert \lambda \vert$} \\
0 & \textrm{otherwise.}
\end{cases}
\end{displaymath}
\end{Proposition}

\begin{proof}
This follows immediately from Propositions~\pref{o:const} and~\pref{o:nohom}.
\end{proof}

\article[Representations of $\Sp(\infty)$] \label{art:Spinfdefn}
The situation for the symplectic group $\Sp(\infty)$ is very similar to the situation of the orthogonal group $\bO(\infty)$, and essentially everything goes through without change. For concreteness, we choose our symplectic form $\bomega$ to be
\begin{displaymath}
\bomega \bigg( \sum_i a_i e_i, \sum_j b_j e_j \bigg) = \sum_k (a_{2k-1} b_{2k} - a_{2k} b_{2k-1}).
\end{displaymath}
We define $T_n$ and $T_{[n]}$ and their finite even dimensional versions as in the orthogonal case. We denote the category of algebraic representations of $\Sp(\infty)$ by $\Rep(\Sp)$.

The result of Weyl's construction in the finite case is slightly different:  the representation $V_{\lambda}^{2d}$ of $\Sp(2d)$ is irreducible if $\ell(\lambda) \le d$ and 0 otherwise, see \cite[\S 17.3]{fultonharris}.  The result in the infinite case is the same:  the $V_{\lambda}$ are a complete irredundant set of irreducibles.  The remaining propositions carry through unchanged.

\subsection{The Brauer algebra and category}

\article[The monoid $\cG_n$]
This discussion is similar to that of \pref{gl:gmonoid}, and thus abbreviated somewhat.  Let $\cV_n$ be the set of vertices $\{x_i, y_i\}$ with $1 \le i \le n$.  We imagine the $x_i$ in the top row and the $y_i$ in the bottom row.  We thus have a notion of horizontal and vertical edges, as before, but there is no longer any wall.

Let $\cG_n$ be the set of complete matchings on the vertex set $\cV_n$.  We give $\cG_n$ the structure of a monoid.  The definition is similar to before:  to multiply $\Gamma$ and $\Gamma'$, put $\Gamma$ on top of $\Gamma'$ and ignore the middle vertices, discarding any components that use only the middle vertices.  Again, we write $n(\Gamma, \Gamma')$ for the number of discarded components.  As before, the symmetric group $S_n$ is identified with the submonoid of $\cG_n$ consisting of graphs with only vertical edges.

\article[The algebra $\cB_n$ and the module $T_n^d$] \label{art:brauer-alg}
We define algebras $\cB_n$ and $\cB_n(\alpha)$ exactly as in \pref{gl:walgebra}.  These algebras are the {\bf Brauer algebras}, introduced in \cite{brauer} (see also \cite{wenzl} for some of its fundamental properties).  We give $T_n^d$ the structure of a $\cB_n(d)$-module.  The definition is similar to that given in \pref{gl:wmodule}, but we provide details (using the same notation). Recall that for $u \in \bC^d$, we defined $f_i(u)$ to be the vector $u$ regarded in the $i$th tensor slot of $(\bC^d)^{\otimes n}$. Let $\Gamma$ be an element of $\cG_n$ and let $v=v_1 \otimes \cdots \otimes v_n$ be an element of $T_n^d$.  The element $w=X_{\Gamma} v$ is defined as a product over the edges of $\Gamma$.  The contribution of edges is as follows:
\begin{itemize}
\item The vertical edges $(x_i, x_j')$ contributes $f_j(v_i)$.
\item The horizontal edge $(x_i, x_j)$ contributes $\langle v_i, v_j \rangle$.
\item The horizontal edge $(x'_i, x'_j)$ contributes $(f_i \otimes f_j)(\omega)$, where $\omega \in \bC^d \otimes \bC^d$ is the form on $\bC^d$.  (We have used the form to identify $\bC^d$ with its dual.) Here $\omega$ can be either symmetric or skew-symmetric, so that we can treat the orthogonal and symplectic cases uniformly.
\end{itemize}
As before, it is critical in this definition that we have specialized the parameter $t$ of $\cB_n$ to $d$, the reason being that when we evaluate the pairing on $\bC^d$ on ``itself'' (regarding it as an element of $\bC^d \otimes \bC^d$), we get $d$.

\article
\label{thm:brauer}
The action of $\cB_n(d)$ on $T_n^d$ obviously commutes with that of $\bO(d)$.  The following is the main result on how these actions relate.  See \cite[Proposition~10.1.3, Corollary~10.1.4]{goodmanwallach} for a proof.

\begin{theorem}[Brauer]
The map $\cB_n(d) \to \End_{\bO(d)}(T_n^d)$ is surjective, and is bijective when $d \ge n$.
\end{theorem}

\article
We now wish to apply the theory of the Brauer algebra in the infinite case, to obtain a diagrammatic description of $\Rep(\bO)$.  As with the walled Brauer algebra, there is no natural way to give $T_n$ the structure of a module over a Brauer algebra since $\Sym^2(\bV)$ does not contain an invariant.  Our solution, again, is to simply disallow horizontal edges in the top row.  As before, we find it more convenient to work with a diagram category than to attempt to form a single algebra.

\article \label{art:def:db}
The {\bf downwards Brauer category}, denoted $\db$, is the following category:
\begin{itemize}
\item The objects are finite sets.
\item A morphism $L \to L'$ is a pair $(\Gamma, f)$, where $\Gamma$ is a matching on $L$ and $f$ is a bijection $L \setminus V(\Gamma) \to L'$. Here $V(\Gamma)$ is the set of vertices adjacent to edges of $\Gamma$.
\item Composition is defined exactly as in $\dwb$, see \pref{dwb:def}.
\end{itemize}
If $L \to L'$ is a morphism in $\db$ then  $\#L' \le \#L$ (hence ``downwards''), with equality if and only if the morphism is an isomorphism.  The automorphism group of $L$ in $\db$ is the full symmetric group on $L$.  Disjoint union endows $\db$ with a symmetric monoidal functor $\amalg$.  We let $\otimes=\otimes_{\#}$ be the resulting convolution tensor product on $\Mod_{\db}$, as defined in \pref{diag:conv}.  We also have the {\bf upwards Brauer category}, denoted $\ub$, and defined analogously to $\uwb$ (see \pref{uwb:def}).  There is a natural equivalence of tensor categories $\Mod_{\ub}^{\fin}=(\Mod_{\db}^{\fin})^{\op}$.

\article
\label{o:ker}
Given an object $L$ of $\db$, put $\cK_L=\bV^{\otimes L}$.  Given a morphism $L \to L'$ in $\db$ we obtain a morphism $\cK_L \to \cK_{L'}$ using the pairing $\bomega$ (similar to \pref{gl:ker}).  We have thus defined an object $\cK$ of $\Rep(\bO)^{\db}$.

\begin{theorem}
\label{o:db}
The functors of \pref{ker-func} associated to the kernel $\cK$ provide contravariant mutually quasi-inverse equivalences of tensor categories between $\Rep(\bO)$ and $\Mod_{\db}^{\fin}$.
\end{theorem}

\begin{proof}
The proof is essentially the same as that of Theorem~\pref{gl:dwb}, but we provide details.  To show that $\Phi$ and $\Psi$ are mutually quasi-inverse, we apply Theorem~\pref{thm:ker-equiv}.  The first hypothesis follows from the fact that $V_{\lambda}$ is simple; note that $\cK_{[L]}$ is isomorphic to $T_{[n]}$ if $n=\# L$.  The second hypothesis follows from Proposition~\pref{o:oneway}.  Finally, it is clear that $\amalg^* \cK=\cK \boxtimes \cK$, and so Proposition~\pref{ker-tens} implies that $\Phi$ is a tensor functor.
\end{proof}

\begin{remark}
This result is closely related to \cite[Corollary~5.2]{koszulcategory}. Specifically, we can think of $\db$ as a locally finite quiver with relations, and the path algebra of this quiver is isomorphic to the algebra $\cA_{\fso_\infty}$ in \cite[\S 5]{koszulcategory}.  Note, however, that \cite[Corollary~5.2]{koszulcategory} does not describe the tensor product from this point of view.
\end{remark}

\begin{Corollary}
\label{o:ub}
The tensor categories $\Rep(\bO)$ and $\Mod_{\ub}^{\fin}$ are equivalent.
\end{Corollary}

\begin{proof}
This follows from the identification $(\Mod_{\db}^{\fin})^{\op}=\Mod_{\ub}^{\fin}$.  A direct equivalence $\Mod_{\ub}^{\fin} \to \Rep(\bO)$ is given by $M \mapsto M \otimes^{\db} \cK$ (see \pref{diag:tens} for notation).
\end{proof}

\begin{Corollary}
The tensor categories $\wh{\Rep}(\bO)$ and $\Mod_{\db}^{\fin}$ are equivalent.
\end{Corollary}

\begin{proof}
This follows from the equivalence $\wh{\Rep}(\bO)=\Rep(\bO)^{\op}$ provided by continuous dual.
\end{proof}

\article[Classification of injectives]
\label{art:o-injectives}
As an application of Theorem~\pref{o:db}, we use \pref{art:projinj} to describe the injective objects of $\Rep(\bO)$, recovering \cite[Corollary~4.6]{koszulcategory}.

\begin{proposition}
The Schur functor $\bS_{\lambda}(\bV)$ is the injective envelope of the simple module $V_{\lambda}$.  The representations $\bS_{\lambda}(\bV)$ constitute a complete irredundant set of indecomposable injectives.
\end{proposition}

\begin{proof}
The proof is analogous to the proof of Proposition~\pref{art:gl-injectives}, so we omit it.
\end{proof}

\article
We now explain how the results of this section can be extended to the symplectic group.  We introduce a modification $\cB'_n$ of the Brauer algebra.  The algebra is spanned by elements $X_{\Gamma}$ where $\Gamma$ is a graph as before, but now the horizontal edges are directed (the vertical edges are undirected).  The orientation of a horizontal edge can be reversed at the cost of a sign, that is, the relation $X_{\Gamma'}=-X_{\Gamma}$ holds if $\Gamma'$ is obtained from $\Gamma$ by flipping the orientation of a single horizontal edge.  To multiply $X_{\Gamma}$ and $X_{\Gamma'}$ in $\cB'_n$, first flip horizontal edges so that the orientations in $\Gamma$ and $\Gamma'$ are compatible, and then proceed as usual.  The $\cB'_n(-d)$-module structure on $T_n^d$ is as before, but with one modification:  the orientation of a horizontal edge indicates the order of the tensor factors for contractions or co-contractions (so we have to specialize our parameter to $-d$ since a coherently oriented cycle in the middle of a composition represents evaluating the symplectic form on its negative).  The obvious analogue of Theorem~\pref{thm:brauer} holds.

\article \label{art:dsbdefn}
We must accordingly modify the downwards Brauer category.  The {\bf downwards signed Brauer category}, denoted $\dsb$, is the following category:
\begin{itemize}
\item Objects are finite sets.
\item A morphism $L \to L'$ is $(\Gamma, f)$ as in $\db$ (see \pref{art:def:db}), but now $\Gamma$ is a directed matching.
\item Composition is defined exactly as in $\dwb$, see \pref{dwb:def}.
\end{itemize}
We let $\Mod_{\dsb}^-$ be the full subcategory of $\Mod_{\dsb}$ on those functors $M \colon \dsb \to \Vec$ which satisfy the following hypothesis:  if $\Gamma'$ is gotten from $\Gamma$ by flipping the orientation of $n$ edges, then $M(\Gamma', f)=(-1)^nM(\Gamma, f)$.  This category is stable under the convolution tensor product $\otimes=\otimes_{\#}$.  One similarly has an {\bf upwards signed Brauer category} $\usb$ and the category $\Mod_{\usb}^-$.

\article
\label{sp:ker}
Given an object $L$ of $\dsb$, put $\cK_L=\bV^{\otimes L}$.  Given a morphism $L \to L'$ in $\dsb$, we obtain a morphism $\cK_L \to \cK_{L'}$ using the pairing $\bomega$.  This is similar to \pref{gl:ker}, but now the directions of the edges of $\Gamma$ indicate the orders in which to pair vectors. We have thus defined an object $\cK$ of $\Rep(\Sp)^{\dsb}$.  The following theorem is proved just like Theorem~\pref{o:db}.

\begin{theorem}
The functors of \pref{ker-func} associated to the kernel $\cK$ provide contravariant mutually quasi-inverse equivalences of tensor categories between $\Rep(\Sp)$ and $\Mod_{\dsb}^{-,\fin}$.
\end{theorem}

\subsection{\texorpdfstring{Modules over $\Sym(\Sym^2)$ and $\Sym(\lw^2)$}{Modules over Sym(Sym2) and Sym(wedge2)}}

\article
\label{o:tca}
For a finite set $L$, let $\cG_L$ denote the set of perfect matchings on $L$.  Then $\cG$ is a tc monoid.  The category $\Lambda$ associated to $\cG$ in \pref{sg:def} is exactly $\ub$.  The tca $A$ associated to $\cG$ in \pref{sg:def} is $\Sym(\Sym^2(\bE))$, where, for clarity, $\bE$ is a copy of $\bC^{\infty}$.

\begin{theorem}
We have an equivalence of tensor categories between $\Rep(\bO)$ and $\Mod^{\fin}_A$, where the latter is endowed with the tensor product $\ast^A$.
\end{theorem}

\begin{proof}
From Proposition~\pref{sg:equiv}, we have an equivalence of categories $\Mod_{\ub}=\Mod_A$, under which $\ast$ corresponds to $\ast^A$.  The result follows from Corollary~\pref{o:ub}.
\end{proof}

\article
\label{sp:tca}
There is also a signed version of the above result.  Let $\cG_L$ be the set of directed perfect matchings on $L$, so that $\cG$ is a tc monoid.  The category associated to $\cG$ is exactly $\usb$.  Let $A'$ be the tca associated to $\cG$, but where one imposes the relations that flipping the orientation of an edge changes a sign.  Then $A'=\Sym(\lw^2(\bE))$.  Combining Theorem~\pref{sp:ker} and an appropriate signed version of Proposition~\pref{sg:equiv}, we obtain the following theorem:

\begin{theorem}
We have an equivalence of tensor categories between $\Rep(\Sp)$ and $\Mod_{A'}^{\fin}$, where the latter is endowed with the tensor product $\ast^{A'}$.
\end{theorem}

\article
We now explain how to construct the equivalence between $\Mod_A^{\fin}$ and $\Rep(\bO)$ directly.  Let $U=\Sym^2(\bE)$, so that $A=\Sym(U)$, and put $B=\Sym(\bE \otimes \bV)$.  We regard $A$ and $B$ as both algebras and coalgebras.  We have a natural linear map $B \to U$ given by
\begin{displaymath}
B \to \Sym^2(\bE \otimes \bV) \to \Sym^2(\bE) \otimes \Sym^2(\bV) \to \Sym^2(\bE)=U,
\end{displaymath}
where the final map makes use of the form $\bomega$.  This induces a coalgebra homomorphism $B \to A$ which is $\GL(\bE)$ equivariant, and gives $B$ the structure of an $A$-comodule.  We thus obtain functors
\begin{displaymath}
\Phi \colon \Mod_A^{\fin} \to \Rep(\bO), \qquad M \mapsto \Hom_A(M^{\vee}, B)
\end{displaymath}
and
\begin{displaymath}
\Psi \colon \Rep(\bO) \to \Mod_A^{\fin}, \qquad V \mapsto \Hom_{\bO}(V, B)^{\vee}.
\end{displaymath}
(It is not immediately clear that $\Psi$ takes values in finite length modules, but this is indeed the case.)  A similar discussion holds in the symplectic case.

\article[Orthogonal--symplectic duality]
Perhaps the most important application of Theorems~\pref{o:tca} and~\pref{sp:tca} is orthogonal--symplectic duality:

\begin{theorem} \label{thm:ospequiv}
There is a natural asymmetric equivalence of tensor categories $\Rep(\bO) \cong \Rep(\Sp)$.  This equivalence takes the simple object $V_{\lambda}$ to the simple object $V_{\lambda^{\dag}}$.
\end{theorem}

\begin{proof}
As discussed in \pref{gl:transp}, transpose is an asymmetric autoequivalence of the tensor category $\cV$.  Since it takes $\Sym^2(\bV)$ to $\lw^2(\bV)$, it takes $A$ to $A'$, and thus induces an equivalence of categories $\Mod_A \cong \Mod_{A'}$.  As the tensor products $\ast^A$ and $\ast^{A'}$ are defined using only the tensor structure of $\cV$, they correspond under this equivalence.
\end{proof}

\article[Computation of $\Ext$ groups] \label{art:o-lw}
We now give a second application of Theorem~\pref{o:tca}:  the computation of the $\Ext$ groups between simple objects of $\Rep(\bO) = \Rep(\Sp)$.  Define $Q_1$ to be the set of partitions $\lambda$ such that for each box $b$ along the main diagonal, the number of boxes to the right of $b$ in the same row is 1 more than the number of boxes below $b$ in the same column. We define $Q_{-1}$ in the same way, except that the roles of rows and columns are swapped. The relevance of these definitions comes from the following two decompositions \cite[Example I.8.6]{macdonald}:
\begin{equation} \label{eqn:q1q-1}
\begin{split}
\lw^i(\Sym^2(\bE)) &= \bigoplus_{\nu \in Q_1,\ \vert \nu \vert=2i} \bS_{\nu}(\bE),\\
\lw^i(\lw^2(\bE)) &= \bigoplus_{\nu \in Q_{-1},\ \vert \nu \vert=2i} \bS_{\nu}(\bE).
\end{split}
\end{equation}
We have the following result:

\begin{proposition}
In $\Rep(\bO)$ we have
\begin{displaymath}
\dim{\Ext^i_{\Rep(\bO)}(V_{\mu}, V_{\lambda})}=\sum_{\nu \in Q_1,\ \vert \nu \vert=2i} c^{\lambda}_{\mu,\nu}.
\end{displaymath}
Equivalently, in $\Rep(\Sp)$ we have
\begin{displaymath}
\dim{\Ext^i_{\Rep(\Sp)}(V_{\mu}, V_{\lambda})}=\sum_{\nu \in Q_{-1},\ \vert \nu \vert=2i} c^{\lambda}_{\mu,\nu}.
\end{displaymath}
In particular, these $\Ext$ groups vanish unless $i=(\vert \lambda \vert-\vert \mu \vert)/2$.
\end{proposition}

\begin{proof}
We work in the orthogonal case.  Under the equivalence of Corollary~\pref{o:tca}, the $\Ext$ group in the statement of the proposition corresponds to
\begin{displaymath}
\Ext^i_A(\bS_{\mu}(\bE), \bS_{\lambda}(\bE)),
\end{displaymath}
where this $\Ext$ is taken in $\Mod_A$.  According to Proposition~\pref{tca:ext}, this coincides with
\begin{displaymath}
\Hom_{\GL(E)}(\bS_{\mu}(\bE) \otimes \lw^i(U), \bS_{\lambda}(E)),
\end{displaymath}
where $U = \Sym^2(\bE)$. Now use \eqref{eqn:q1q-1}.
\end{proof}

\article[Classification of blocks] \label{art:osp-blocks}
We get an easy description of the block structure of $\Rep(\bO) = \Rep(\Sp)$, which recovers \cite[Proposition 6.12]{koszulcategory}:

\begin{corollary}
Two simples $V_{\lambda}$ and $V_{\mu}$ belong to the same block of the category $\Rep(\bO) = \Rep(\Sp)$ if and only if $|\lambda| = |\mu| \pmod 2$.
\end{corollary}

\begin{remark}
The blocks are naturally indexed by $\bZ/2$, and one can show that the two blocks are not equivalent to one another.
\end{remark}

\article[Littlewood varieties]
We give a third application of Theorem~\pref{o:tca}:  the construction of injective resolutions of simple objects of $\Rep(\bO) = \Rep(\Sp)$. To avoid certain technicalities arising from the difference between $\bO(d)$ and $\SO(d)$, we will stick with $\Rep(\Sp)$. As in \pref{art:gl-lw}, it will be more natural to use $\wh{\Rep}(\Sp)$.  Let $E$ be a finite dimensional vector space, let $U=\lw^2(E)$, let $A=\Sym(U)$ and let $B=\Sym(E \otimes \wh{\bV})$, where each symmetric power is taken in $\wh{\Vec}$.  The space $U$ is a subspace of $B$, and so we have an algebra homomorphism $A \to B$. We let $C$ be the quotient of $B$ by the ideal generated by $U$.  We have maps
\begin{displaymath}
\Spec(C) \to \Spec(B) \to \Spec(A).
\end{displaymath}
To give a geometric interpretation of these maps, we ignore subtleties caused by spaces being infinite dimensional.  The space $\Spec(A)$ is identified with the space of forms $\lw^2(E) \to \bC$, while $\Spec(B)$ is identified with the space $\Hom(E, \bV)$ of maps $\varphi \colon E \to \bV$.  The map $\Spec(B) \to \Spec(A)$ takes a linear maps $\varphi$ to the form $\varphi^* \bomega$, where $\bomega$ is the symplectic form on $\bV \otimes \bV_*$.  The space $\Spec(C)$, which we call the {\bf Littlewood variety}, is the scheme-theoretic fiber of this map above 0, i.e., it consists of maps $\varphi$ for which $\varphi^* \bomega=0$. Alternatively, it consists of those maps $\phi$ whose image is totally isotropic with respect to $\bomega$.  Let $K_{\bullet}=B \otimes \lw^{\bullet}(U)$ be the Koszul complex of the Littlewood variety.

\begin{Proposition}
\label{sp:koszul}
The augmented complex $K_{\bullet} \to C$ is exact. We have a decomposition
\begin{displaymath}
C=\bigoplus_{\lambda} \bS_{\lambda}(E) \otimes \wh{V}_{\lambda}.
\end{displaymath}
\end{Proposition}

\begin{proof}
See \cite[\S 3.3]{ssw}.
\end{proof}

\article[Littlewood complexes] \label{art:sp_littlewoodcx}
We now give the projective resolutions of simple objects.  For a partition $\lambda$, define the {\bf Littlewood complex} $L^{\lambda}_{\bullet}$ by
\begin{displaymath}
L^{\lambda}_{\bullet}=\Hom_{\GL(E)}(\bS_{\lambda}(E), K_{\bullet}),
\end{displaymath}
where $E$ is of sufficient dimension (the definition is then independent of $E$).  Proposition~\pref{sp:koszul} shows that
\begin{displaymath}
\rH_i(L^{\lambda}_{\bullet})=\begin{cases}
\wh{V}_{\lambda} & \textrm{if $i=0$} \\
0 & \textrm{otherwise,} \end{cases}
\end{displaymath}
and so $L^{\lambda}_{\bullet}$ is a resolution of the simple object $\wh{V}_{\lambda}$.  Furthermore, it is clear that $K_{\bullet}$ is built from polynomial Schur functors applied to $\wh{\bV}$, and so each $K_i$ is projective in $\wh{\Rep}(\Sp)$.  In fact, by \eqref{eqn:q1q-1}, we have
\begin{displaymath}
L^{\lambda}_i=\bigoplus_{\mu \in Q_{-1},\ \vert \mu \vert=i} \bS_{\lambda/\mu}(\wh{\bV}).
\end{displaymath}
Thus the Littlewood complex $L^{\lambda}_{\bullet}$ is a projective resolution of $\wh{V}_{\lambda}$; in fact, it is a minimal resolution.

\article[(Lack of) Fourier transform]
\label{o:fourier}
Consider the four algebras:
\begin{displaymath}
A=\Sym(\Sym^2(\bE)), \quad
B=\Sym(\lw^2(\bE)), \quad
C=\lw(\Sym^2(\bE)), \quad
D=\lw(\lw^2(\bE)).
\end{displaymath}
As we have discussed, transpose interchanges $A$ and $B$, and also $C$ and $D$.  Thus we have equivalences $\Mod_A^{\fin} \cong \Mod_B^{\fin}$ and $\Mod_C^{\fin} \cong \Mod_D^{\fin}$, as abelian categories.  A modification of Koszul duality, similar to that occurring in \pref{rmk:glkoszul}, shows that $\Mod_A^{\fin}$ is Koszul dual to $\Mod_C^{\fin}$.  As the next lemma shows, $\Mod_A^{\fin}$ and $\Mod_C^{\fin}$ are not equivalent. It follows from this that $\Rep(\bO)$ and $\Rep(\Sp)$ are not Koszul dual to each other or themselves. Note that $C$ and $D$ are transpose dual to one another, so one can setup a Fourier transform from $A$ to $D$ (and from $B$ to $C$) in a way similar to \pref{rmk:glkoszul}.

\begin{lemma}
The categories $\Mod_A^{\fin}$ and $\Mod_C^{\fin}$ are not equivalent.
\end{lemma}

\begin{proof}
Both of these categories are module categories for certain quivers with relations. In both cases, the quivers (ignoring the relations) are described as follows: the vertices are partitions, and there is an arrow $\lambda \to \mu$ if and only if $\lambda \subset \mu$ and $\mu / \lambda$ consists of two boxes in different columns. The reason is that the vertices index simple objects and the arrows index $\ext^1$-groups. So any hypothetical equivalence between the two categories must induce an automorphism of the underlying quiver. It is easy to see that the only possible automorphism is the identity, so any equivalence needs to preserve the indexing (by partitions) on the simple objects. But the higher extension groups do not match (for example, $\ext^2_{\Mod_A}(L_0, L_{3,1}) = \bC$ but $\ext^2_{\Mod_C}(L_0, L_{3,1}) = 0$), so no such equivalence exists.
\end{proof}

\subsection{Orthogonal and symplectic Schur functors, universal property and specialization}
\label{sec:orthosymschur}

\article[Orthogonal Schur functors]
\label{o:schurfun}
Let $\cA$ be a tensor category.  Define $T(\cA)$ to be the category whose objects are pairs $(A, \omega)$, where $A$ is an object of $\cA$ and $\omega$ is a symmetric pairing $A \otimes A \to \bC$, and whose morphisms are the obvious things.  We typically suppress $\omega$ from the notation.  Given $A \in T(\cA)$, define $\cK(A)$ to be the object of $\cA^{\db}$ given by $L \mapsto A^{\otimes L}$.  Functoriality with respect to morphisms in $\db$ makes use of the pairing $\omega$, and is defined as in \pref{o:ker}.  For an object $M$ of $\Mod_{\ub}^{\fin}$, define
\begin{displaymath}
S_M(A)=M \otimes^{\db} \cK(A).
\end{displaymath}
Then $M \mapsto S_M(A)$ defines a covariant functor $\Mod_{\ub}^{\fin} \to \cA$ which is left-exact (see \pref{diag:tens}) and a tensor functor (since $\cK(A)$ is obviously a tensor kernel, see \pref{ker-tens}).  We call $S_M$ the {\bf orthogonal Schur functor} associated to $M$.  

\begin{Theorem}
\label{c:univ}
To give a left-exact tensor functor from $\Rep(\bO)$ to an arbitrary tensor category $\cA$ is the same as to give an object of $T(\cA)$.  More precisely, letting $\bM$ be the object of $\Mod_{\ub}$ corresponding to $\bV$, the functors
\begin{displaymath}
\Phi \colon \LEx^{\otimes}(\Mod_{\ub}^{\fin}, \cA) \to T(\cA), \qquad F \mapsto F(\bM)
\end{displaymath}
and
\begin{displaymath}
\Psi \colon T(\cA) \to \LEx^{\otimes}(\Mod_{\ub}^{\fin}, \cA), \qquad A \mapsto (M \mapsto S_M(A))
\end{displaymath}
are mutually quasi-inverse equivalences.
\end{Theorem}

\begin{proof}
The proof is the same as that of Theorem~\pref{gl:univ}.
\end{proof}

\begin{remark}
This theorem can be rephrased as follows:  the functor $T \colon \TCat \to \Cat$ is corepresented by $\Rep(\bO)$, with the universal object in $T(\Rep(\bO))$ being $\bV$.  See \pref{art:schur-univ} for notation.
\end{remark}

\article[Symplectic Schur functors]
The above results carry over in an evident way to the symplectic case.  Precisely, given a tensor category $\cA$, let $T'(\cA)$ be the category of pairs $(A, \omega)$ where $A$ is an object of $\cA$ and $\omega$ is an alternating pairing on $A$.  Given $A \in T'(\cA)$ one can build a kernel $\cK(A) \in \cA^{\dsb, -}$.  Given $M \in \Mod_{\dsb}^{-,\fin}$, we define $S'_M(A)=M \otimes^{\dsb} \cK(A)$.  We call $S'_M$ the {\bf symplectic Schur functor} associated to $M$.  The functor
\begin{displaymath}
T'(A) \to \LEx^{\otimes}(\Mod_{\dsb}^{-,\fin}, \cA)
\end{displaymath}
is an equivalence of categories.

\article[The specialization functors]
\label{c:special}
The standard representations of $\bO(d)$ and (for $d$ even) $\Sp(d)$ define objects of $T(\Rep(\bO(d)))$ and $T'(\Rep(\Sp(d)))$, and so by Theorem~\pref{c:univ} and its symplectic variant we obtain left-exact tensor functors
\begin{align*}
\Gamma_d \colon \Rep(\bO) &\to \Rep(\bO(d)),\\
\Gamma_d \colon \Rep(\Sp) &\to \Rep(\Sp(d)),
\end{align*}
which we call the {\bf specialization functors}.  In the orthogonal case, the results of \pref{o:weylfin} show that $\Gamma_d(V_{\lambda})$ is the irreducible $V_{\lambda}^d$ if the first two columns of $\lambda$ have at most $d$ boxes, and 0 otherwise.  In the symplectic case, the corresponding result is that $\Gamma_d(V_{\lambda})$ is the irreducible $V_{\lambda}^d$ if $\ell(\lambda) \le d$ and 0 otherwise.

\article[Specialization via invariants]
\label{c:special2}
We now give a more direct description of specialization.  We will treat both $\bO(\infty)$ and $\Sp(\infty)$ at the same time.  Choose a decomposition $\bC^{\infty}=V \oplus V'$ where $\dim(V)=d$ and $\bomega(V, V')=0$ and the restriction of $\bomega$ to both $V$ and $V'$ is non-degenerate.  If $d$ is even, one can take $V$ to be the span of the $e_i$ with $i \le d$ and $V'$ to be the span of the $e_i$ with $i>d$.  If $d$ is odd, one can take $V$ to be the span of $e_1, \ldots, e_{d-1}, e_d+e_{d+1}$ and $V'$ to be the span of $e_d-e_{d+1}, e_{d+2}, \ldots$  Let $G_d$ (resp.\ $H_d$) be the subgroup of $\bO(\infty)$ or $\Sp(\infty)$ which preserves $V$ (resp.\ $V'$) and acts as the identity on $V'$ (resp.\ $V$).  Then $G_d$ is isomorphic to $\bO(d)$ or $\Sp(d)$ while $H_d$ is isomorphic to $\bO(\infty)$ or $\Sp(\infty)$.  The subgroups $G_d$ and $H_d$ commute, and so the $H_d$-invariants of any representation of $\bO(\infty)$ or $\Sp(\infty)$ form a representation of $G_d$.

\begin{proposition}
We have a natural identification $\Gamma_d(V)=V^{H_d}$.
\end{proposition}

\begin{proof}
The proof is similar to the proof of Proposition~\pref{prop:gl_invts}.
\end{proof}

\begin{remark}
Since $\Gamma_d$ is a tensor functor, the above proposition shows that if $V$ and $W$ are algebraic representations of $\bO(\infty)$ (or $\Sp(\infty)$) then $(V \otimes W)^{H_d}=V^{H_d} \otimes W^{H_d}$.
\end{remark}

\article[Derived specialization] \label{o:dersp}
As the categories $\Rep(\bO)$ and $\Rep(\Sp)$ have enough injectives, the derived functor $\rR\Gamma_d$ of $\Gamma_d$ exists.  The injective resolution of the simple object $V_{\lambda}$ is given by the Littlewood complex $L^{\lambda}_{\bullet}$ (see \pref{art:sp_littlewoodcx} for the dual picture).  Since specialization behaves in the obvious manner on polynomial Schur functors (see \pref{c:special}), $\rR\Gamma_d(V_{\lambda}) = \Gamma_d(L^{\lambda}_{\bullet})$ is just $L^{\lambda}_{\bullet}(\bC^d)$, which is by definition the result of evaluating the Schur functors in $L^{\lambda}_\bullet$ on $\bC^d$.  The cohomology of this complex is computed in \cite[\S\S 3.5, 4.5]{ssw}, the result being:

\begin{theorem}
Let $\lambda$ be a partition and let $d \ge 1$ be an integer.  Then $\rR^i\Gamma_d(V_{\lambda})$ either vanish identically or else there exists a unique $i$ for which it is non-zero, and it is then an irreducible representation of $\bO(d)$ or $\Sp(d)$.
\end{theorem}

Furthermore, there is a rule, the {\bf modification rule}, which calculates where the cohomology is non-zero, and what the resulting irreducible of $\bO(d)$ or $\Sp(d)$ is.  See \cite[\S\S 3.4, 4.4]{ssw} for details. The Euler characteristic of this complex was previously computed by \cite[\S 2.4]{koiketerada}, which suggested the results of \cite{ssw}.

\article \label{art:osp-T0T1}
Let $T_0=T(\Vec^{\fin})$ be the category of pairs $(V, \omega)$ where $V$ is a finite dimensional vector space and $\omega$ is a symmetric bilinear form on $V$, and let $T_1$ be the full subcategory where $\omega$ is perfect.  Let $\cA$ be the category of representations of $\bO(\infty)$ such that every element is stabilized by $H_d$ for some $d$.  Define a functor
\begin{displaymath}
\sF \colon \cA \to \Fun(T_1, \Vec)
\end{displaymath}
as follows.  For $U \in \cA$ and $V \in T_1$, pick an isomorphism $V \cong \bC^d$ respecting the form and put $\sF(U)(V)=U^{H_d}$; see \pref{gl:acat} for how to make this canonical.  Define a functor
\begin{displaymath}
\sG \colon \Fun(T_1, \Vec) \to \cA
\end{displaymath}
by the formula $\sG(F)=\varinjlim F(\bC^d)$, where the transition map $\bC^d \to \bC^{d+1}$ is the inclusion given by our choice of bases in \pref{c:special2}, and the form on $\bC^d$ is the one restricted from $\bC^\infty$.  Basic properties of the direct limit show that $\sG$ is exact and respects tensor products.  There are natural maps $\sG \sF \to \id$ and $\id \to \sG \sF$, the first of which is an isomorphism.

\begin{Theorem} \label{o:schur}
The functor $\sF$ induces a left-exact fully faithful tensor functor $\Rep(\bO) \to \Fun(T_1, \Vec^{\fin})$.
\end{Theorem}

\begin{proof}
The proof is the same as that of Theorem~\pref{gl:schur}.  
\end{proof}

\article
By Theorem~\pref{o:schur}, we can think of $\Rep(\bO)$ as a category of functors $T_1 \to \Vec$.  As in the GL-case, $\Rep(\bO)$ is not an abelian subcategory of the functor category.  We expect the abelian closure to have similar properties as in the GL-case, see \pref{gl:catc}.  Similar comments apply in the symplectic case, using $T_1'$ in place of $T_1$.

\section{The general affine group} \label{sec:alin}

\subsection{\texorpdfstring{Representations of $\GA(\infty)$}{Representations of GA(inf)}}

\article \label{art:gadefn}
Let $\bt \colon \bV \to \bC$ be a non-zero linear map which annihilates all but finitely many basis vectors.  For concreteness, we take $\bt=e_1^*$.  The {\bf general affine group}, denoted $\GA(\infty)$, is the subgroup of $\GL(\infty)$ stabilizing $\bt$.  We let $\bV_0$ be the kernel of $\bt$; this is the subspace spanned by $e_2, e_3, \ldots$.  There is a natural surjection $\GA(\infty) \to \GL(\bV_0)$, and we let $\bT=\bT(\infty)$ be the kernel (the group of translations).  We let $\GA(d)$ and $\bT(d)$ be the intersections of $\GA(\infty)$ and $\bT(\infty)$ with $\GL(d) \subset \GL(\infty)$.

We say that a representation of $\GA(\infty)$ is {\bf polynomial} if it appears as a subquotient of a finite direct sum of the spaces $T_n=\bV^{\otimes n}$.  We let $\Rep^{\pol}(\GA)$ denote the category of polynomial representations.  It is an abelian category and stable under tensor products.  Some remarks:
\begin{enumerate}
\item We regard $\bV_0$ as a representation of $\GA(\infty)$ with $\bT$ acting trivially.  It is a sub of $T_1$ and thus polynomial.
\item The center of $\GA(\infty)$ is trivial, even in the approximate sense discussed in \pref{gl:center}.  We will see that the category $\Rep^{\pol}(\GA)$ has only one block.
\item The group $\GA(\infty)$ also admits a faithful linear representation on $\bV_*$.  This leads to a category of {\bf anti-polynomial representations} which is in fact equivalent to $\Rep^{\pol}(\GA)^{\op}$.
\item One can consider pro-polynomial representations of $\GA(\infty)$, and this leads to a category $\wh{\Rep}{}^{\pol}(\GA)$.  ``Continuous dual'' provides a contravariant equivalence of the category of pro-polynomial representations with the category of anti-polynomial representations.  Thus the categories $\wh{\Rep}{}^{\pol}(\GA)$ and $\Rep^{\pol}(\GA)$ are equivalent.
\item Define a representation of $\GA(\infty)$ to be {\bf algebraic} if it appears as a subquotient of a direct sum of representations of the form $\bV^{\otimes n} \otimes \bV_*^{\otimes m}$.  The category of algebraic representations of $\GA(\infty)$ appears to be somewhat different from the other categories we have studied:  for instance, the trivial representation is neither injective nor projective.  It would be interesting to determine the structure of this category.
\end{enumerate}

\article[An analogue of Weyl's construction] \label{art:ga-weyl}
Let $t_i \colon T_n \to T_{n-1}$ be the map which applies $\bt$ to the $i$th tensor factor.  Let $T_{[n]}$ be the intersection of the kernels of the $t_i$.  This is a representation of $S_n \times \GA(\infty)$.  We define
\begin{displaymath}
V_{\lambda}=\Hom_{S_n}(\bM_{\lambda}, T_{[n]}).
\end{displaymath}
We have a decomposition
\begin{equation}
\label{ga:decomp}
T_{[n]}=\bigoplus_{\vert \lambda \vert=n} \bM_{\lambda} \boxtimes V_{\lambda}.
\end{equation}
Note that we also have an exact sequence
\begin{equation}
\label{ga:seq}
0 \to T_{[n]} \to T_n \to T_{n-1}^{\oplus n},
\end{equation}
where the right map is made up of the $n$ trace maps.

\begin{Proposition}
The $V_{\lambda}$ constitute a complete irredundant set of simple objects of $\Rep^{\pol}(\GA)$.
\end{Proposition}

\begin{proof}
We have $T_{[n]}=\bV_0^{\otimes n}$, and so $V_{\lambda}$ is just the representation $\bS_{\lambda}(\bV_0)$ of $\GL(\bV_0)$, with $\bT$ acting trivially.  This shows that the $V_{\lambda}$ are simple and distinct.  The sequence \eqref{ga:seq} shows that every simple of $\Rep^{\pol}(\GA)$ is a constituent of some $T_{[n]}$, while the decomposition \eqref{ga:decomp} shows that every simple constituent of $T_{[n]}$ is some $V_{\lambda}$.
\end{proof}

\begin{Proposition}
Every object of $\Rep^{\pol}(\GA)$ has finite length.
\end{Proposition}

\begin{proof}
This is just like the proof of Proposition~\pref{gl:finlen}.
\end{proof}

\begin{Proposition}
\label{ga:const}
The simple constituents of $T^n$ are those $V_{\lambda}$ with $\vert \lambda \vert \le n$.
\end{Proposition}

\begin{proof}
This is just like the proof of Proposition~\pref{gl:const}.
\end{proof}

\article
By a {\bf weight} of $\GA(\infty)$ we mean a weight of the diagonal torus in $\GL(\bV_0)$.  We identify weights with sequences of integers $(a_2, a_3, \ldots)$ which are eventually zero.  (We start from 2 since $\bV_0$ has for a basis the $e_i$ with $i \ge 2$.)  The weights appearing in polynomial representations of $\GA(\infty)$ satisfy $a_i \ge 0$.  The {\bf magnitude} of a weight is the sum of the (absolute values of) the $a_i$.  Every weight in $T_n$ clearly has magnitude at most $n$.

\begin{Proposition}
Every non-zero submodule of $T_n$ has a weight of magnitude $n$.
\end{Proposition}

\begin{proof}
The space $T_n$ has for a basis pure tensors of the $e_i$, with $i \ge 1$.  The magnitude of the weight of a basis vector is simply the number of $e_i$'s with $i>1$ it has.  Let $M$ be a non-zero submodule of $T_n$ and let $x$ be a non-zero element of $M$.  Choose an index $i>1$ such that $e_i$ does not occur in the expansion of $x$ in this basis.  Regard $e_i$ as an element of $\bT$.  Then $e_i x$ is computed by changing each $e_1$ in the occurrence of $x$ to $e_1+e_i$.  Let $x'$ be the element $x$ but with all $e_1$'s changed to $e_i$.  Then $x'$ is non-zero, since we have simply changed the basis vectors occurring in $x$ and not their coefficients.  Furthermore, it is clear that all weights appearing in $x'$ have magnitude $n$, since $x'$ has no $e_1$'s in it.  Finally, it is clear that $x'$ is exactly the magnitude $n$ piece of $e_i x$:  the only terms of $e_i x$ with magnitude $n$ are those in which only $e_i$'s are chosen in the expansion of the tensor product.  Thus $x'$ belongs to $M$ since $M$ is a weight module, which completes the proof.
\end{proof}

\begin{Proposition}
\label{ga:nohom}
Let $M$ be a submodule of $T_n$.  Then $\Hom_{\GA}(M, T_{n+r})=0$ for $r>0$.
\end{Proposition}

\begin{proof}
This is just like the proof of Proposition~\pref{gl:nohom}.
\end{proof}

\begin{Proposition}
We have
\begin{displaymath}
\Hom_{\GA}(V_{\lambda}, T_n)=\begin{cases}
\bM_{\lambda} & \textrm{if $n=\vert \lambda \vert$} \\
0 & \textrm{otherwise.}
\end{cases}
\end{displaymath}
\end{Proposition}

\begin{proof}
This follows immediately from Propositions~\pref{ga:nohom} and~\pref{ga:const}.
\end{proof}

\subsection{A diagram category}

\article \label{art:defnds}
We now give an analogue of the Brauer category for the general affine group.  The {\bf downwards subset category}, denoted $\ds$, is the category whose objects are finite sets and where a morphism $L \to L'$ is a pair $(U, f)$ where $U$ is a subset of $L$ and $f \colon L \setminus U \to L'$ is a bijection.  Composition is defined as usual.  The automorphism group of an object $L$ of $\ds$ is the symmetric group on $L$.  Disjoint union endows $\ds$ with a symmetric monoidal structure.   We let $\otimes=\otimes_{\#}$ be the resulting convolution tensor products, as in \pref{diag:conv}.  We also have the {\bf upwards subset category}, denoted $\us$, with everything reversed.  Note that $\us$ can be described as the category of finite sets with morphisms being injections.

\article
\label{ga:K}
Given an object $L$ of $\ds$, put $\cK_L=\bV^{\otimes L}$.  Given a morphism $L \to L'$ in $\ds$ represented by $(U, f)$ we obtain a morphism
\begin{displaymath}
\cK_L = \bV^{\otimes U} \otimes \cK_{L\setminus U} \to \cK_{L'}
\end{displaymath}
by applying the maps $\bt \colon \bV \to \bC$ on the first factor and using $f$ on the second factor.  Thus $\cK$ defines an object of $\Rep^{\pol}(\GA)^{\ds}$.

\begin{Theorem}
The functors of \pref{ker-func} associated to the kernel $\cK$ provide contravariant mutually quasi-inverse equivalences of tensor categories between $\Rep^{\pol}(\GA)$ and $\Mod_{\ds}^{\fin}$.
\end{Theorem}

\begin{proof}
The proof of the theorem is just like that of Theorem~\pref{gl:dwb}.  
\end{proof}

\begin{Corollary}
The tensor categories $\Rep^{\pol}(\GA)$ and $\Mod_{\us}^{\fin}$ are equivalent.
\end{Corollary}

This comes from the identification $(\Mod_{\ds}^{\fin})^{\op}=\Mod_{\us}^{\fin}$.  A direct equivalence $\Mod_{\us}^{\fin} \to \Rep^{\pol}(\GA)$ is given by $M \mapsto M \otimes^{\ds} \cK$ (see \pref{diag:tens} for notation).

\article[Classification of injectives]
\label{art:ga-injectives}
As before, the above descriptions of $\Rep^{\pol}(\GA)$ allow for an easy description of its injective objects.

\begin{proposition}
The Schur functor $\bS_{\lambda}(\bV)$ is the injective envelope of the simple module $V_{\lambda}$.  The representations $\bS_{\lambda}(\bV)$ constitute a complete irredundant set of indecomposable injectives.
\end{proposition}

\subsection{\texorpdfstring{Modules over $\Sym(\bC\langle 1 \rangle)$}{Modules over Sym(C<1>)}}
\label{ga:amod}

\article \label{thm:ga:tca}
For a set $L$, let $\cG_L$ be the one point set.  Then $\cG_L$ is a tc monoid.  The category $\Lambda$ associated to $\cG$ in \pref{sg:def} is exactly $\us$.  The tca $A$ associated to $\cG$ is $\Sym(\bC\langle 1 \rangle)$.  As usual, $\bE$ denotes a copy of $\bC^{\infty}$ and we identify $A$ with $\Sym(\bE)$.

\begin{theorem}
The tensor categories $\Rep^{\pol}(\GA)$ and $\Mod_A^{\fin}$ are equivalent, where the latter is given the tensor product $\ast^A$.
\end{theorem}

\begin{proof}
The proof is the same as the proof of Theorem~\pref{gl:amod}.
\end{proof}

\article \label{art:ga:modK}
As in previous settings, we can realize the equivalence between $\Rep^{\pol}(\GA)$ and $\Mod_A^{\fin}$ directly.  Put $B=\Sym(\bV \otimes \bE)$.  We regard $A$ and $B$ as both algebras and coalgebras.  We have a natural map $B \to \bE$ by first projecting onto $\bV \otimes \bE$ and then using the linear map $\bt \colon \bV \to \bC$.  This induces a coalgebra homomorphism $B \to A$.  We thus obtain functors
\begin{displaymath}
\Phi \colon \Mod_A^{\fin} \to \Rep^{\pol}(\GA), \qquad M \mapsto \Hom_A(M^{\vee}, B)
\end{displaymath}
and
\begin{displaymath}
\Psi \colon \Rep^{\pol}(\GA) \to \Mod_A^{\fin}, \qquad V \mapsto \Hom_{\GA}(V, B)^{\vee}.
\end{displaymath}
These functors are mutually quasi-inverse equivalences.

\begin{remark}
In fact, the equivalence can be seen even more directly.  The group $\GA(\infty)$ is isomorphic to the semi-direct product $\bV \rtimes \GL(\infty)$, so giving a representation of it is the same as giving a representation of $\bV$, which is the same as giving a $\Sym(\bV)$ module, that is equipped with a $\GL(\infty)$ equivariance.  One can easily verify directly that the equivariant modules one obtains are exactly the finite length polynomial ones.
\end{remark}

\article \label{ga:modK}
The category $\Mod_A$ is thoroughly studied in \cite{symc1}. There, the category $\Mod_A^\fin$ is denoted $\Mod_A^\tors$, and the Serre quotient is denoted $\Mod_K = \Mod_A / \Mod_A^\tors$. Among other things, it is shown that these categories are equivalent:
\[
\Mod_A^\tors \simeq \Mod_K,
\]
and many basic invariants were computed, such as minimal injective resolutions of the simple objects, and extension groups between simple objects \cite[\S 2.3]{symc1}. Furthermore, both of these categories were shown to be equivalent to the module category of the Pieri quiver ${\rm Part}_{\rm HS}$: this is the quiver whose vertices are all partitions, and there is exactly 1 arrow $\lambda \to \mu$ if $\lambda \subseteq \mu$ and the skew Young diagram $\mu / \lambda$ is a horizontal strip, i.e., no two boxes lie in the same column, and there are no arrows, otherwise. The relations are governed by the rule that the composition $\lambda \to \mu \to \nu$ is 0 if $\nu / \lambda$ is not a horizontal strip, and the rule that all such compositions are equal otherwise.

\article[Littlewood varieties]
Following \pref{art:gl-lw} and \pref{art:o-lw}, we can define the Littlewood variety in the present context.  Let $E$ be a finite dimensional vector space, let $A=\Sym(E)$, and let $B=\Sym(E \otimes \wh{\bV}_*)$.  Since $\wh{\bV}_*$ has an invariant, we have an inclusion $E \subset B$, and thus obtain an algebra homomorphism $A \to B$.  We let $C$ be the quotient of $B$ by the ideal generated by $U$.  We call $Y=\Spec(C)$ the {\bf Littlewood variety}.  Since $Y$ is defined by linear equations, the minimal free resolution of $C$ is given by the acyclic Koszul complex $K_\bullet$ where $K_i = B \otimes \bigwedge^i E$.

\begin{Proposition} \label{prop:ga:coord}
We have $C=\bigoplus_\lambda \bS_\lambda(E) \otimes V_{\lambda}^{\vee}$.
\end{Proposition}

\article[Littlewood complexes]
For a partition $\lambda$, define the {\bf Littlewood complex} $L^{\lambda}_{\bullet}$ by
\begin{displaymath}
L^{\lambda}_{\bullet}=\Hom_{\GL(E)}(\bS_{\lambda}(E), K_{\bullet}),
\end{displaymath}
where $E$ is of sufficient dimension.  Proposition~\pref{prop:ga:coord} and the acyclicity of the Koszul complex show that
\begin{displaymath}
\rH_i(L^{\lambda}_{\bullet})=\begin{cases}
V_{\lambda}^{\vee} & \textrm{if $i=0$} \\
0 & \textrm{otherwise,} \end{cases}
\end{displaymath}
and so $L^{\lambda}_{\bullet}$ is a resolution of the simple object $V_{\lambda}^{\vee}$.  Furthermore, it is clear that $K_{\bullet}$ is built from polynomial Schur functors applied to $\wh{\bV}_*$, and so each $K_i$ is projective in $\Rep^{\pol}(\GA)^{\op}$.  In fact, we have
\begin{displaymath}
L^{\lambda}_i=\bS_{\lambda/(1^i)}(\wh{\bV}_*).
\end{displaymath}
Thus $L^{\lambda}_{\bullet}$ is a projective resolution of $V_{\lambda}^{\vee}$; in fact, it is minimal.

\article
The projective resolution $L_{\lambda}^{\bullet}$ of $V_{\lambda}^{\vee}$ gives an injective resolution
\[
0 \to V_{\lambda} \to \bS_\lambda(\bV) \to \bS_{\lambda/(1)} (\bV) \to \cdots \to \bS_{\lambda / (1^i)}(\bV) \to \cdots \to \bS_{\lambda / (1^{\ell(\lambda)})} (\bV) \to 0
\]
in $\Rep^{\pol}(\GA)$. By Theorem~\pref{thm:ga:tca}, $\Rep^{\pol}(\GA)$ is equivalent to $\Mod_A^{\fin}$, which is in turn equivalent to the category $\Mod_K$ (see the discussion in \pref{art:ga:modK}). Under these equivalences, the above injective resolutions become the injective resolutions of simple objects of $\Mod_K$ described in \cite[Theorem 2.3.1]{symc1}.

\subsection{Schur functors for objects with trace, universal property and specialization}

\article
Let $\cA$ be a tensor category.  Define $T(\cA)$ to be the category whose objects are pairs $(A, t)$, where $A$ is an object of $\cA$ and $t$ is a map $A \to \bC$, and whose morphisms are the obvious things.  We will typically suppress the $t$ from notation.  Given $A \in T(\cA)$, define $\cK(A)$ to be the object of $\cA^{\ds}$ given by $L \mapsto A^{\otimes L}$.  Functoriality with respect to morphisms in $\ds$ makes use of $t$, and is as defined in \pref{ga:K}.  For an object $M$ of $\Mod^{\fin}_{\us}$, define
\begin{displaymath}
S_M(A)=M \otimes^{\ds} \cK(A).
\end{displaymath}
Then $M \mapsto S_M(A)$ defines a covariant functor $\Mod_{\us}^{\fin} \to \cA$ which is left-exact (see \pref{diag:tens}) and a tensor functor (since $\amalg^*\cK(A)$ is a tensor kernel, see \pref{ker-tens}).

\begin{Theorem}
\label{ga:univ}
To give a left-exact tensor functor from $\Rep^{\pol}(\GA)$ to a tensor category $\cA$ is the same as to give an object of $T(\cA)$.  More precisely, letting $\bM$ be the object of $\Mod_{\us}^{\fin}$ corresponding to $\bV$, the functors
\begin{displaymath}
\Phi_{\cA} \colon \LEx^{\otimes}(\Mod_{\us}^{\fin}, \cA) \to T(\cA), \qquad F \mapsto F(\bM)
\end{displaymath}
and
\begin{displaymath}
\Psi_{\cA} \colon T(\cA) \to \LEx^{\otimes}(\Mod_{\us}^{\fin}, \cA), \qquad A \mapsto (M \mapsto S_M(A))
\end{displaymath}
are mutually quasi-inverse equivalences of categories.
\end{Theorem}

\begin{proof}
The proof is the same as that of Theorem~\pref{gl:univ}.
\end{proof}

\begin{remark}
This theorem can be rephrased as follows: the functor $T \colon \TCat \to \Cat$ is corepresented by $\Rep^{\pol}(\GA)$, with the universal object in $T(\Rep^{\pol}(\GA))$ being $\bV$.  See \pref{art:schur-univ} for notation.
\end{remark}

\article[The specialization functor] \label{art:ga-special}
The object $\bC^d$ defines an object of $T(\Rep^{\pol}(\GA(d)))$, and so by Theorem~\pref{ga:univ} we obtain a left-exact tensor functor
\begin{displaymath}
\Gamma_d \colon \Rep^{\pol}(\GA) \to \Rep^{\pol}(\GA(d)),
\end{displaymath}
which we call the {\bf specialization functor}.

\article[Specialization via invariants] \label{art:GA-special}
The group $\GA(\infty)$ is the subgroup of $\GL(\infty)$ consisting of matrices where the top left entry is 1, and all other entries in the first row are 0.  Let $G_d$ be the subgroup of $\GA(\infty)$ which agrees with the identity matrix outside of the upper left $d \times d$ block, and let $H_d$ be the subgroup which agrees with the identity matrix outside the complementary block.  The group $G_d$ is isomorphic to $\GA(d)$, while $H_d$ is isomorphic to $\GL(\infty)$ --- the infinite general {\it linear} group.  One can show that $\Gamma_d(V)=V^{H_d}$; we leave this to the reader.  Since a polynomial representation of $\GA(\infty)$ restricts to a polynomial representation of $H_d$, and the category of such representations is semi-simple, we obtain:

\begin{proposition}
The specialization functor $\Gamma_d$ is exact.
\end{proposition}

\article
\label{ga:T0}
Let $T_0=T(\Vec^{\fin})$ be the category of pairs $(V, t)$ where $V$ is a finite dimensional vector space and $t \colon V \to \bC$ is a linear map.  Let $T_1$ be the subcategory on objects where $t$ is non-zero.  As in previous cases, one can build functors between $\Rep^{\pol}(\GA)$ and $\Fun(T_1, \Vec^{\fin})$. For each $d$, define $H_d \subset \GA(\infty)$ as in \pref{art:GA-special}. Let $\cA$ be the category of representations of $\GA(\infty)$ such that each element is stabilized by $H_d$ for some $d$. Define a functor
\[
\sF \colon \cA \to \Fun(T_1, \Vec),
\]
as follows. For $U \in \cA$ and $V \in T_1$, pick an isomorphism $V \cong \bC^d$ and put $\sF(U)(V) = U^{H_d}$ (see \pref{gl:acat} for how to make this canonical).

\begin{Theorem}
The functor $\sF$ induces an exact fully faithful tensor functor $\Rep^{\pol}(\GA) \to \Fun(T_1, \Vec^{\fin})$.
\end{Theorem}

\begin{proof}
The proof is similar to the proof of Theorem~\pref{gl:schur}.  We have exactness in this context due to the exactness of the specialization functor.
\end{proof}

\section{The symmetric group} \label{sec:symgroup}

\subsection{\texorpdfstring{Representations of $\fS(\infty)$}{Representations of S(inf)}}

\article[Algebraic representations] \label{art:symgpdefn}
Let $\cB=\{1,2,\ldots,\}$ and let $\fS=\fS(\infty)$ be the group of permutations of $\cB$ which fix all but finitely many elements.  The group $\fS$ acts on $\bV$ by permuting the basis vectors $\{e_i\}_{i \in \cB}$.  We say that a representation of $\fS$ is {\bf algebraic} if it appears as a subquotient of a direct sum of the representations $T_n=\bV^{\otimes n}$.  We let $\Rep(\fS)$ denote the category of algebraic representations.  It is an abelian tensor category.  Some remarks:
\begin{enumerate}
\item We use the notation $\fS$ when we think of the symmetric group on the ``representation theory side'' and $S$ when we think of it on the ``diagram category side.''
\item For an integer $d \ge 0$, we let $G_d$ (resp.\ $H_d$) be the subgroup of $\fS(\infty)$ which fixes all $i>d$ (resp.\ $i \le d$).  Then $G_d=\fS(d)$ while $H_d$ is isomorphic to $\fS(\infty)$.  The two subgroups commute, and so $G_d \times H_d$ is a subgroup of $\fS(\infty)$.
\item The representation $\bV_*$ of $\fS(\infty)$ is isomorphic to $\bV$, so we do not gain anything by considering tensor powers of $\bV_*$.
\item One can also define the category $\wh{\Rep}(\fS)$ of pro-algebraic representations.  ``Continuous dual'' provides an equivalence of $\Rep(\fS)^{\op}$ with $\wh{\Rep}(\fS)$.
\end{enumerate}

\article \label{art:symgpreps}
In preparation for some of the proofs in this section, we will need a few basic facts about the representation theory of the symmetric groups.

\begin{itemize}
\item If $\nu^1, \dots, \nu^r$ is the set of partitions obtainable from $\lambda$ by removing a single box, then $\dim \bM_\lambda = \dim \bM_{\nu^1} + \cdots + \dim \bM_{\nu^r}$ \cite[(2.8)]{expos}.

\item The multiplicity of $\bM_\mu$ in the tensor product $\bM_\lambda \otimes (\bM_{(n-1,1)} \oplus \bC)$ is the number of ways to remove a single box from $\lambda$ and add it back to get $\mu$ \cite[Exercise 7.81]{stanley}.
\end{itemize}

\article[An analogue of Weyl's construction (finite case)]
\label{sym:fin-weyl}
Before we study the structure of $\Rep(\fS)$, we give a construction of the irreducible representations of the finite symmetric groups $\fS(d)$ analogous to Weyl's construction for the classical groups.  Let $T_n^d=(\bC^d)^{\otimes n}$.  Let $t \colon T_1^d \to \bC$ be the augmentation map, sending each $e_i$ to 1, and let $t_i \colon T_n^d \to T_{n-1}^d$ be the map given by applying $t$ to the $i$th factor.  Let $s \colon T_2^d \to T_1^d$ be the map $e_i \otimes e_j \mapsto \delta_{i,j} e_i$, and let $s_{i,j} \colon T_n^d \to T_{n-1}^d$ be the map obtained by applying $s$ to the $i$th and $j$th factors and inserting the result in the final tensor factor.  We let $T_{[n]}^d$ be the intersection of the kernels of the $t_i$ and $s_{i,j}$.  It is clear that $T_{[n]}^d$ is stable under the action of $S_n \times \fS(d)$ on $T_n$.  For a partition $\lambda$ of $n$, we put
\begin{displaymath}
V_{\lambda}^d=\Hom_{S_n}(\bM_{\lambda}, T_{[n]}^d)
\end{displaymath}
This space carries an action of $\fS(d)$.  Before stating the main result, we introduce a piece of notation:  for a partition $\lambda=(\lambda_1, \ldots, \lambda_r)$ and an integer $k$, we let $\lambda[k]$ be the sequence $(k, \lambda_1, \ldots, \lambda_r)$.  This is a partition if $k \ge \lambda_1$.

\begin{proposition}
The representation $V^d_{\lambda}$ is isomorphic to $\bM_{\lambda[d-n]}$ if $\lambda[d-n]$ is a partition, and $0$ otherwise.
\end{proposition}

\begin{proof}
We will use the partition algebra $\cA_n(d)$ which will be defined in \S\ref{sec:partalg}. There is an action of $\cA_n(d)$ on $T_n^d$ so that the image of $\cA_n(d)$ in $\End(T_n^d)$ is the full centralizer of the action of $\fS_d$ (Theorem~\pref{thm:partalgcent}). The generators of this action are summarized in \pref{art:partalg-gens}. From this description, we see that if $I_n$ is the ideal generated by the generators $p_i$ and $p_{i+\frac{1}{2}}$, then $\cA_n(d) / I_n = \bC[S_n]$ and that $T_{[n]}^d$ is exactly the subspace of $T_n^d$ annihilated by $I_n$. By the double centralizer theorem, we have a decomposition 
\[
T_n^d = \bigoplus_\lambda X_\lambda \boxtimes Y_\lambda
\]
where $X_\lambda$ is an irreducible representation of $\cA_n(d)$ and $Y_\lambda$ is an irreducible representation of $\fS(d)$. From the previous discussion, $T_{[n]}^d$ is the direct sum of those $X_\lambda \boxtimes Y_\lambda$ for which the action of $I_n$ is identically 0 on $X_\lambda$, and in this case $Y_\lambda = V^d_\lambda$.

First suppose that $d-n \ge \lambda_1$. We show that $V^d_\lambda$ is a nonzero irreducible module by a character calculation. We use the following fact: the multiplicity of $\bM_\mu$ in $\bS_\lambda(\bC^d)$ is given by the coefficient of the Schur polynomial $s_\lambda(x_1, \dots, x_d)$ in the plethysm $s_\mu(1, x_1, \dots, x_d, x_1^2, x_1x_2, \dots)$ (i.e., plugging in all monomials in $x_1, \dots, x_d$ of all possible degrees) \cite[Ex.~7.74]{stanley}. Using the combinatorial definition of Schur polynomials as a generating function for semistandard Young tableaux \cite[\S 7.10]{stanley}, this immediately implies that $\bM_{\lambda[d-n]}$ appears with multiplicity 1 in $\bS_\lambda(\bC^d)$ and that if $\bM_\mu$ appears, then we must have $\sum_{i \ge 2} \mu_i \le |\lambda|$. In particular, $\bM_{\lambda[d-n]}$ is in the kernel of all maps of the form $T_n^d \to T_{n-1}^d$ discussed above, so we conclude that $V^d_\lambda \cong \bM_{\lambda[d-n]}$.

Now suppose that $\lambda_1 > d-n$. We claim that $V_\lambda^d = 0$. It is easy to see that the construction of $T_{[n]}^d$ is stable with respect to $d$ in the following sense: if we identify $\bC^d \subset \bC^{d+1}$ as the subspace spanned by $e_1, \dots, e_d$, then $T_{[n]}^d \subset T_{[n]}^{(d+1)}$, so in particular, $V_\lambda^d \subset V_{\lambda}^{(d+1)}$. Pick $d'$ so that $\lambda_1 = d'-n+1$. It is enough to show that $V^{(d')}_{\lambda} = 0$.

By what we have already shown, $V_\lambda^{(d'+1)} = \bM_{(d'+1-n, \lambda)}$, so $V_\lambda^{(d')}$, if nonzero, must be of the form $\bM_{(d'+1-n,\mu)}$ where $\mu$ is obtained from $\lambda$ by removing a single box. We claim that the multiplicity $m_\mu$ of $\bM_{(d'+1-n,\mu)}$ in $T_n^{(d')}$ is strictly bigger than the dimension of $\bM_\lambda$ (this implies that $V_\lambda^{(d')} = 0$ by the discussion above on the relationship between $T^n_d$ and $T^{[n]}_d$). Let $\nu^1, \dots, \nu^r$ be all partitions we can get from $\lambda$ by removing a single box. 
For each $i$, the multiplicity of $\bM_{(d'-n+1, \nu^i)}$ in $T_{n-1}^{(d')}$ is $\dim \bM_{\nu^i}$. So from the dimension equation and tensor product rule in \pref{art:symgpreps}, they each contribute $\dim \bM_{\nu^i}$ to $m_\mu$, and so far we see $m_\mu \ge \dim \bM_\lambda$. Finally, let $\eta$ be the result of removing a single box from $\mu$ (the only obstruction to the existence of $\eta$ is if $n=1$, but in this case, the proposition is trivial). Then the multiplicity of $\bM_{(d'-n+2, \eta)}$ in $T_{n-1}^{(d')}$ is at least 1, and it also contributes to $m_\mu$, so we conclude that $m_\mu > \dim \bM_\lambda$, which proves the claim.
\end{proof}

\begin{remark}
This traceless tensor construction for the symmetric group was also considered by Littlewood in \cite[\S 4]{littlewood2}. However, he phrases the construction in terms of the quadratic and cubic invariants of the symmetric group.
\end{remark}

\article[An analogue of Weyl's construction (infinite case)] \label{symgp:weyl:inf}
Much of the above discussion carries over to the infinite case.  We define the maps $t_i \colon T_n \to T_{n-1}$ and $s_{i,j} \colon T_n \to T_{n-1}$ as before, and let $T_{[n]}$ be the intersection of their kernels.  For a partition $\lambda$ of $n$, we put
\begin{displaymath}
V_{\lambda}=\Hom_{S_n}(\bM_{\lambda}, T_{[n]}).
\end{displaymath}
This space carries an action of $\fS(\infty)$ and obviously forms an algebraic representation.  We have a decomposition
\begin{equation}
\label{sym:decomp}
T_{[n]}=\bigoplus_{\vert \lambda \vert=n} \bM_{\lambda} \boxtimes V_{\lambda}
\end{equation}
of $S_n \times \fS(\infty)$ modules.  We have an exact sequence
\begin{equation}
\label{sym:seq}
0 \to T_{[n]} \to T_n \to (T_{n-1})^{\oplus n(n+1)/2},
\end{equation}
where the rightmost map is the $n$ maps $t_i$ and the $n(n-1)/2$ maps $s_{i,j}$.

\begin{Proposition}
The $V_{\lambda}$ constitute a complete irredundant set of simple objects of $\Rep(\fS)$.
\end{Proposition}

\begin{proof}
Since $V_{\lambda}=\bigcup_{d \gg 0} V_{\lambda}^d$ and each $V_{\lambda}^d$ is an irreducible representation of $\fS(d)$ (for $d \gg 0$), the representation $V_{\lambda}$ is irreducible.  An induction argument using \eqref{sym:decomp} and \eqref{sym:seq} shows that every simple object of $\Rep(\fS)$ is isomorphic to some $V_{\lambda}$.  Finally, we show that if $V_{\lambda}$ is isomorphic to $V_{\mu}$ then $\lambda=\mu$. It is possible to give an argument similar to the cases of the classical groups above. To do this, one replaces the maximal torus with the subalgebra of the group algebra of the symmetric group generated by the Jucys--Murphy elements \cite[\S 2]{ov} $X_i$, which is the sum of the transpositions $(1,i) + (2,i) + \cdots + (i-1,i)$, and is well-defined for the infinite symmetric group. The relevant character theory is explained in \cite[\S 5]{ov}. The punchline is that $V^d_\lambda$ has a basis $v_T$ indexed by standard Young tableaux $T$ of shape $(d-|\lambda|, \lambda)$ which is an eigenbasis for $X_2, X_3, \dots, X_d$. The eigenvalue of $X_i$ on $v_T$ is the content (row index minus column index) of the box of $T$ which contains the label $i$. Furthermore, this eigenbasis is compatible with the inclusions $V_\lambda^d \subset V_\lambda^{(d+1)}$, so is well-defined for $d \to \infty$, and we see that $V_\lambda \cong V_\mu$ if and only if $\lambda = \mu$.
\end{proof}

\begin{remark}
One can picture $V_{\lambda}$ as corresponding to the Young diagram $\lambda[\infty]$, i.e., $\lambda$ placed below an infinite first row.
\end{remark}

\begin{Proposition} \label{prop:fsfinitelength}
Every object of $\Rep(\fS)$ has finite length.
\end{Proposition}

\begin{proof}
This is just like the proof of Proposition~\pref{gl:finlen}.
\end{proof}

\begin{Proposition}
\label{sym:const}
The simple constituents of $T_n$ are those $V_{\lambda}$ with $\vert \lambda \vert \le n$.
\end{Proposition}

\begin{proof}
This is just like the proof of Proposition~\pref{gl:const}.
\end{proof}

\begin{Remark}
The analogues of Propositions~\pref{gl:nohom} and~\pref{gl:oneway} do not hold.  For example, the map $T_1 \to T_n$ taking $e_i$ to $e_i^{\otimes n}$ is $\fS(\infty)$-equivariant, and therefore the irreducible $V_{(1)}$ occurs as a submodule of $T_n$ for all $n$.
\end{Remark}

\subsection{\texorpdfstring{Modules over $\Sym(\bC\langle 1 \rangle)$}{Modules over Sym(C<1>)}}

\article
In this section we relate $\Rep(\fS)$ to the category of modules over the tca $A=\Sym(\bC\langle 1 \rangle)=\Sym(\bV)$.  We let $\Mod_A^{\fg}$ denote the category of finitely generated $A$-modules and we let $\Mod_K$ be the Serre quotient of $\Mod_A^{\fg}$ by the category $\Mod_A^{\fin}$ of torsion $A$-modules; see \cite[\S 2]{symc1} for background on this category.  Write 
\[
T \colon \Mod_A^{\fg} \to \Mod_K
\]
for the quotient functor and 
\[
S \colon \Mod_K \to \Mod_A^{\fg}
\]
for its right adjoint.  The category $\Mod_A$ is equivalent to the category $\Mod_{\us}$, as discussed in \S\ref{ga:amod}.

\article
\label{sym:Tfunc}
Let $M$ be an $A$-module, thought of as an object of $\Mod_{\us}$.  Thus $M$ assigns to each finite set $L$ a vector space $M_L$ and to each injection $L \to L'$ of finite sets, a map of vector spaces $M_L \to M_{L'}$.  Define
\begin{displaymath}
T'(M) = \varinjlim_{L \subset \cB} M_L,
\end{displaymath}
where the colimit is over the finite subsets $L$ of $\cB$.  It is clear that $T'(M)$ is an $\fS$-module.  A simple computation shows that $T'(A \otimes \bV^{\otimes n})$ is an algebraic representation of $\fS$.  As every finitely generated $A$-module $M$ is a quotient of a finite direct sum of modules of the form $A \otimes \bV^{\otimes n}$, and direct limits are exact, we see that $T'(M)$ is a quotient of an algebraic representation and therefore algebraic.  We have thus defined an exact functor
\begin{displaymath}
T' \colon \Mod_A^{\fg} \to \Rep(\fS).
\end{displaymath}
It is clear that $T'$ kills $\Mod_A^{\fin}$, and therefore induces an exact functor
\begin{displaymath}
\Phi \colon \Mod_K \to \Rep(\fS).
\end{displaymath}

\article
\label{sym:Sfunc}
Let $M$ be an $\fS$-module.  Given a finite set $L$ of cardinality $d$, choose a bijection $L \to \ul{d}$ and put $S'(M)_L=M^{H_d}$ ($H_d$ is defined in \pref{art:symgpdefn}); more canonically, define $S'(M)_L$ via a limit as in \pref{gl:acat}.  Then $S'(M)$ defines an object of $\Mod_{\us}$.  We claim that $S'$ takes algebraic representations to finitely generated $A$-modules.  Since $S'$ is left-exact, it suffices to verify the claim for a simple algebraic module $V_{\lambda}$.  In fact, since $V_{\lambda}$ injects into $T_n$, for some $n$, it is enough to show that $S'(T_n)$ is finitely generated.  This is a simple computation, which we leave to the reader. We have thus defined a left-exact functor
\begin{displaymath}
S' \colon \Rep(\fS) \to \Mod_A^{\fg}.
\end{displaymath}
There is an obvious isomorphism $T'S'=\id$ and natural transformation $\id \to S'T'$, and  these give $S'$ the structure of a right adjoint of $T'$.  We define
\begin{displaymath}
\Psi \colon \Rep(\fS) \to \Mod_K
\end{displaymath}
to be the composition $\Psi=TS'$.

\begin{Theorem}
\label{sym:tca}
The functors $\Phi$ and $\Psi$ are mutually quasi-inverse equivalences between $\Mod_K$ and $\Rep(\fS)$.
\end{Theorem}

\begin{proof}
It follows from what we have already done that $\Phi \Psi=\id$ and that there is a natural map $\eta \colon \id \to \Psi \Phi$.  It suffices to show that $\eta$ is an isomorphism.  Let $M$ be an object of $\Rep(\fS)$.  We have a natural morphism $f \colon M \to S'(T'(M))$.  As $T'(f)$ is an isomorphism and $T'$ is exact, we see that $T'(\ker{f})=T'(\coker{f})=0$.  If $N$ is a finitely generated $A$-module with $T'(N)=0$ then $N$ has finite length.  We thus see that $\ker(f)$ and $\coker(f)$ have finite length.  This shows that $T(\ker{f})=T(\coker{f})=0$, and so $T(f)$ is an isomorphism.   Thus for any $M \in \Mod_A$, the natural map $TM \to T(S'(T'(M)))=\Psi(\Phi(TM))$ is an isomorphism.  Since every object of $\Mod_K$ is of the form $TM$ for some $M \in \Mod_A$, this proves the result.
\end{proof}

\begin{Corollary} \label{cor:repfSmodA}
The categories $\Rep(\fS)$ and $\Mod_A^{\fin}$ are equivalent.
\end{Corollary}

\begin{proof}
This follows from Theorem~\pref{sym:tca} and \cite[Thm.~2.5.1]{symc1}, which states that $\Mod_K$ and $\Mod_A^{\fin}$ are equivalent.
\end{proof}

\begin{Corollary} \label{cor:sym-ga}
The categories $\Rep(\fS)$ and $\Rep^{\pol}(\GA)$ are equivalent.
\end{Corollary}

\begin{proof}
This follows from Corollary~\pref{cor:repfSmodA} and Theorem~\pref{thm:ga:tca}.
\end{proof}

\begin{remark}
This equivalence can be realized as an ``infinite Schur--Weyl duality,'' as follows. Let $\bt \in \bV_*$ be as in \pref{art:gadefn}. Define $\bV_*^{\otimes \infty}$ to be the direct limit of the $\bV_*^{\otimes n}$, where the transition maps are given by $v \mapsto v \otimes \bt$. Concretely, an element of $\bV_*^{\otimes \infty}$ is a finite sum of tensors of the form $v_1 \otimes v_2 \otimes v_3 \otimes \cdots$, where $v_i \in \bV_*$ and for all but finitely many $i$ we have $v_i=\bt$. The group $\fS(\infty)$ acts on $\bV_*^{\otimes \infty}$ by permuting tensor factors. We thus have a functor
\begin{displaymath}
\Rep(\fS) \to \Rep^{\rm apol}(\GA), \qquad M \mapsto \Hom_{\fS}(M, \bV_*^{\otimes \infty}).
\end{displaymath}
(Here $\rm apol$ means ``anti-polynomial,'' see \pref{art:gadefn}.) This functor is a contravariant equivalence. Since $\Rep^{\rm apol}(\GA)$ is the opposite category of $\Rep^{\pol}(\GA)$, the above functor induces a covariant equivalence $\Rep(\fS) \to \Rep^{\pol}(\GA)$.
\end{remark}

\article[Compatibility with tensor products] \label{art:symwrongtensor}
The equivalences of \pref{sym:tca}, \pref{cor:repfSmodA} and \pref{cor:sym-ga} are {\it not} equivalences of tensor categories using the usual tensor structures.  We now describe an alternative tensor structure on $\Mod_K$ which the equivalence \pref{sym:tca} respects.  Identifying $\Mod_A$ with $\Mod_{\us}$, define the {\bf pointwise tensor product} $\boxtimes$ as in \pref{art:ptwise}, i.e., by $(M \boxtimes N)_L=M_L \otimes N_L$.  It is not difficult to show that this preserves $\Mod_A^{\fg}$.  Furthermore, if $M$ has finite length and $N$ is finitely generated then $M \boxtimes N$ has finite length.  It follows that $\boxtimes$ induces a tensor structure on $\Mod_K$.  Since tensor products and direct limits commute, the functor $T'$ in \pref{sym:Tfunc} is a tensor functor if we use the $\boxtimes$ tensor structure on $\Mod_A$.  It follows that $\Phi$ and $\Psi$ are equivalences of tensor categories if we use $\boxtimes$ on $\Mod_K$.

\begin{Proposition} \label{prop:symc1tensorproj}
If $M$ and $N$ are finitely generated projective $A$-modules then so is $M \boxtimes N$.
\end{Proposition}

\begin{proof}
Let $A'$ be the $\GL(\infty)$-equivariant Weyl algebra on $\bC\langle 1 \rangle$, see \pref{art:bialgebras}. By Proposition~\pref{art:bialgebras}, a finitely generated $A$-module is projective if and only if the action of $A$ on extends to an action of $A'$.  The proposition follows from the fact that the pointwise tensor product of two $A'$-modules admits a natural $A'$-module structure.  (To see this, note that one can think of an $A'$-module as an object $M$ having compatible $A$-module and $A$-comodule structures.  The pointwise tensor product of $A$ (co)modules is again an $A$ (co)module, so all that remains to verify is that compatibility is preserved.  We leave this to the reader.) 
\end{proof}

\begin{Proposition} \label{prop:symc1tensorinj}
If $M$ and $N$ are injective objects of $\Mod_K$ then so is $M \boxtimes N$.
\end{Proposition}

\begin{proof}
This follows from Proposition~\pref{prop:symc1tensorproj} and the fact \cite[Remark~4.2.7]{symc1} that the functor $T$ is an equivalence between the categories of projectives in $\Mod_A^{\fg}$ and injectives in $\Mod_K$.
\end{proof}

\begin{Proposition} \label{prop:repfStensorinj}
The tensor product of two injectives in $\Rep(\fS)$ is again injective.
\end{Proposition}

\begin{proof} 
This follows from Proposition~\pref{prop:symc1tensorinj} and the fact that the equivalence of Theorem~\pref{sym:tca} is compatible with the $\boxtimes$ tensor product.
\end{proof}

\begin{Proposition}
\label{sym:Vninj}
The object $\bV^{\otimes n}$ of $\Rep(\fS)$ is injective.
\end{Proposition}

\begin{proof}
By Proposition~\pref{prop:repfStensorinj}, it suffices to treat the $n=1$ case.  An easy computation shows that $\bV=T'(A \otimes \bV)=\Phi(T(A \otimes \bV))$.  Since $A \otimes \bV$ is projective in $\Mod_A^{\fg}$, its image under $T$ is injective in $\Mod_K$.  The result follows, since $\Phi$ is an equivalence.
\end{proof}

\article[The objects $\bS_{\lambda}(\bV)$] \label{sym:inj}
By Proposition~\pref{sym:Vninj}, the Schur functors $\bS_{\lambda}(\bV)$ are injective in $\Rep(\fS)$.  In contrast to previous situations, i.e., \pref{art:gl-injectives}, \pref{art:o-injectives}, \pref{art:ga-injectives}, these objects are not indecomposable.  For example, we have $\Sym^2(\bV)=I_{(2)} \oplus I_{(1)}$, where we write $I_{\lambda}$ for the injective envelope of the simple $V_{\lambda}$.  Determining the general decomposition of $\bS_{\lambda}(\bV)$ into indecomposable injectives seems like a difficult problem. It is also an interesting problem to calculate the multiplicities of simple objects in $\bS_\lambda(\bV)$, and one can interpret the intermediate step of decomposing into indecomposable injectives as putting additional structure on these multiplicities. 

\subsection{The partition algebra and category} \label{sec:partalg}

\article[The monoid $\cP_n$]
Let $\cV_n$ be the set of vertices $\{x_1, x_1', \dots, x_n, x_n'\}$. We will think of $\{x_1, \dots, x_n\}$ as being the ``bottom'' vertices and $\{x_1', \dots, x_n'\}$ as being the ``top'' vertices, as in \pref{wb:ex}.  Let $\cP_n$ denote the set of partitions of $\cV_n$.  We give $\cP_n$ the structure of a monoid, as follows.  Let $\cU$ and $\cU'$ be two partitions of $\cV_n$.  Put $\cU$ above $\cU'$, i.e., identify the bottom row of vertices of $\cU$ with the top row of vertices of $\cU'$, and merge all parts of $\cU$ and $\cU'$ which overlap.  Let $n(\cU, \cU')$ be the number of connected components (including singletons) which are contained entirely in the middle row; discard these parts, and all vertices in the middle row.  The resulting partition is the composition $\cU \cU'$.  The symmetric group $S_n$ sits inside of $\cP_n$ as the set of partitions in which each part contains exactly one vertex from the top row and one from the bottom row.

\article[The algebra $\cA_n$] \label{art:partition-algebra}
Let $\cA_n$ be the free $\bC[t]$-module spanned by $\cP_n$.  We write $X_{\cU}$ for the element of $\cA_n$ corresponding to the partition $\cU \in \cP_n$.  We give $\cA_n$ the structure of an algebra by defining $X_{\cU} X_{\cU'}$ to be $t^{n(\cU, \cU')} X_{\cU \cU'}$.  For a number $\alpha \in \bC$, we let $\cA_n(\alpha)$ be defined similarly to $\cA_n$, but with $\alpha$ in place of $t$; of course, $\cA_n(\alpha)$ is just the quotient of $\cA_n$ by the two-sided ideal generated by $t-\alpha$.  The algebra $\cA_n$ is called the {\bf partition algebra}, and was introduced in \cite{martin} and \cite{jones} (see also \cite{halversonram}).

\article[The $\cA_n(d)$-module $T_n^d$]
We now give $T_n^d$ the structure of a $\cA_n(d)$-module.  We use notation similar to that in \pref{gl:wmodule}.  Let $\cU$ be an element of $\cP_n$ and let $v$ be an element of $T_n^d$.  We assume, without loss of generality, the $v$ is a pure tensor of basis vectors $e_{i_1} \otimes \cdots \otimes e_{i_n}$, where each $i_k$ is between 1 and $d$.  The element $w=X_{\Gamma} v$ is defined as a product over the parts of $\cU$, so we just have to describe the contribution of each part.
\begin{itemize}
\item A part $\{x_{k_1}, \ldots, x_{k_r} \}$ concentrated in the first row contributes 1 if $i_{k_1}=\cdots=i_{k_r}$ and 0 otherwise.
\item A part $\{x'_{\ell_1}, \ldots, x'_{\ell_s} \}$ concentrated in the second row contributes $\sum_{j=1}^d f_{\ell_1}(e_j) \cdots f_{\ell_s}(e_j)$.
\item A part $\{x_{k_1}, \ldots, x_{k_r}, x'_{\ell_1}, \ldots, x'_{\ell_s} \}$ which meets each row (so $r>0$ and $s>0$) contributes $f_{\ell_1}(e_i) \cdots f_{\ell_s}(e_i)$ if $i_{k_1}=\cdots=i_{k_r}=i$ and 0 otherwise.
\end{itemize}
We leave it to the reader to verify that this is a well-defined action.  As in previous cases, it is important that the parameter $t$ of $\cA_n$ has been specialized to $d$.

\article[An example]
We now give an example of the module structure introduced above.  Suppose $n=4$ and $\cU$ is the partition
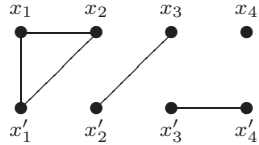
\begin{displaymath}
\begin{xy}
(-15, 5)*{}="A1"; (-5, 5)*{}="B1"; (5, 5)*{}="C1"; (15, 5)*{}="D1";
(-15, -5)*{}="A2"; (-5, -5)*{}="B2"; (5, -5)*{}="C2"; (15, -5)*{}="D2";
(-15, 8)*{\ss x_1}; (-5, 8)*{\ss x_2}; (5, 8)*{\ss x_3}; (15, 8)*{\ss x_4};
(-15, -8)*{\ss x_1'}; (-5, -8)*{\ss x_2'}; (5, -8)*{\ss x_3'}; (15, -8)*{\ss x_4'};
"A1"*{\bullet}; "B1"*{\bullet}; "C1"*{\bullet}; "D1"*{\bullet};
"A2"*{\bullet}; "B2"*{\bullet}; "C2"*{\bullet}; "D2"*{\bullet};
"A1"; "B1"; **\dir{-};
"A1"; "A2"; **\dir{-};
"B1"; "A2"; **\dir{-};
"C1"; "B2"; **\dir{-};
"C2"; "D2"; **\dir{-};
\end{xy}
\end{displaymath}
Thus $\{x_1, x_2, x_1'\}$ constitutes one part, as does $\{x_4\}$.  Then:
\begin{displaymath}
X_{\Gamma}(e_{i_1} \otimes e_{i_2} \otimes e_{i_3} \otimes e_{i_4})=\delta_{i_1, i_2} \cdot e_{i_1} \otimes e_{i_3} \otimes \left( \sum_{j=1}^d e_j \otimes e_j \right).
\end{displaymath}

\article[Generators of $\cA_n$] \label{art:partalg-gens}
Generators for the partition algebra are given in \cite[Theorem 1.11(d)]{halversonram}. We explain the action of these generators on $T^d_n$ (see \cite[\S 3]{halversonram}).
\begin{itemize}
\item For $1 \le i \le n-1$, the generator $s_i$ swaps the tensor factors in positions $i$ and $i+1$. 
\item For $1 \le i \le n$, the generator $p_i$ replaces the $i$th tensor factor with $e_1 + \cdots + e_d$.
\item For $1 \le i \le n-1$, the generator $p_{i+\frac{1}{2}}$ is defined by $e_{j_1} \otimes \cdots \otimes e_{j_n} \mapsto \delta_{j_i, j_{i+1}} e_{j_1} \otimes \cdots \otimes e_{j_n}$.
\end{itemize}
Let $I_n$ be the ideal generated by the $p_i$ and $p_{i+\frac{1}{2}}$. Then $\cA_n(d) / I_n \cong \bC[S_n]$ where the generators for $\bC[S_n]$ are the images of the $s_i$.

\article
The action of $\cA_n(d)$ on $T_n^d$ obviously commutes with that of $\fS(d)$.  The following is the main result on how these actions relate. See \cite[Theorem 3.6]{halversonram} for a proof.

\begin{theorem}[Martin] \label{thm:partalgcent}
The natural map $\cA_n(d) \to \End_{\fS(d)}(T_n^d)$ is always surjective, and is an isomorphism for $d>2n$.
\end{theorem}

\article
We now wish to apply the theory of the partition algebra in the infinite case, to obtain a diagrammatic model for $\Rep(\fS)$.  As in previous cases, there is a problem:  partitions concentrated in the bottom row would involve infinite sums.  As before, our solution is to simply disallow them. Again it is more convenient to work with a category than attempting to create a single algebra.

\article \label{art:defndp}
The {\bf downwards partition category}, denoted $\dpc$, is the following category.  Objects are finite sets.  A morphism $L \to L'$ is a partition of $L \amalg L'$ in which each part meets $L$.  Given a morphism $L \to L'$ represented by $\cU$ and a morphism $L' \to L''$ represented by $\cU''$, the composition $L \to L''$ is represented by the partition obtained by gluing $\cU$ and $\cU'$ along $L'$ and merging parts which meet.  The automorphism group of $L$ in $\dpc$ is the symmetric group on $L$, however, there are self-maps of $L$ which are not isomorphisms.  Disjoint union endows $\dpc$ with a monoidal functor $\amalg$.  We let $\otimes=\otimes_{\#}$ be the resulting convolution tensor product on $\Mod_{\dpc}$, as defined in \pref{diag:conv}.  We also have the {\bf upwards partition category}, denoted $\upc$, with everything reversed.

\begin{remark}
The category $\dpc$ is very different from all the previous combinatorial categories we have considered:  it is not weakly directed.  This causes the theory of modules over $\dpc$ to be significantly more complicated than the previous theories, and most of the results of \S \ref{ss:repsofcategory} break down.  For instance, we will see that any non-zero module over $\dpc$, even a simple module, is non-zero on all sufficiently large finite sets.  A few of the more formal results from \S \ref{ss:repsofcategory} (e.g., the existence of the convolution tensor product) do remain true, with the same proof, and we will take care when citing them.
\end{remark}

\article
Before getting to the main result of this section, Theorem~\pref{sym:dp}, we need to establish some intermediate results on the structure of $\Mod_{\dpc}$.  We summarize these results here:
\begin{itemize}
\item In \pref{sym:map}--\pref{sym:map2}, we define certain classes of maps in $\dpc$ and prove elementary results about them.
\item In \pref{sym:mindef}--\pref{sym:delta}, we define the notion of a (weakly) minimal element of a representation of $\dpc$, and prove a lifting result for these elements.
\item In \pref{sym:Kndef}--\pref{sym:Knproj}, we define objects $K^n$ of $\Mod_{\dpc}^{\gfin}$, and prove that they are projective.
\item In \pref{sym:mingen}--\pref{sym:finlen2}, we show that $K^n$ has finite length and that every finite length object of $\Mod_{\dpc}^{\gfin}$ is a quotient of a finite direct sum of $K^n$'s.
\end{itemize}

\article
\label{sym:map}
Let $f \colon L \to L'$ be a morphism in $\dpc$ defined by a partition $\cU=\{U_i\}_{i \in I}$.
\begin{itemize}
\item We say that $f$ is a {\bf monomorphism} if every part of $\cU$ meets $L'$ and contains a unique element of $L$.  This coincides with the categorical notion of monomorphism, i.e., if $g_1f = g_2f$, then $g_1=g_2$. The composition of two monomorphisms is again a monomorphism, and if $gf$ is a monomorphism then so is $f$. 
\item We say that $f$ is an {\bf epimorphism} if every part of $\cU$ contains at most one element of $L'$. This coincides with the categorical notion of epimorphism, i.e., if $fg_1 = fg_2$, then $g_1=g_2$. It is clear that the composition of two epimorphisms is again an epimorphism, and that if $fg$ is an epimorphism then so is $f$. 
\item We say that $f$ is {\bf proper} if each $U_i$ meets $L'$.  A composition $gf$ is proper if and only if $f$ is.  If $f$ is proper then the partition $\cU$ then defines a morphism $f^t \colon L' \to L$, which we call the {\bf transpose} of $f$; obviously $f^t$ is proper and $(f^t)^t=f$.  If $f$ is a monomorphism then it is proper, $f^t$ is an epimorphism and $f^tf=\id_L$.  If $f$ is a proper epimorphism then $f^t$ is a monomorphism and $ff^t=\id_{L'}$.
\item We say that $f$ is {\bf rectangular} if $U_i$ has the same number of elements of $L$ and $L'$, for each $i$.  Rectangular maps are proper.  If $f$ is any proper map then $ff^t$ and $f^tf$ are rectangular idempotents.
\end{itemize}
If $f$ is a monomorphism and $M \in \Mod_{\dpc}$ then $f$ induces an injection $M_L \to M_{L'}$, since $f^t$ provides a left-inverse.  In particular, if $M \ne 0$ then $M_L \ne 0$ for all sufficiently large $L$.

\begin{Proposition}
\label{sym:map1}
Any morphism $f$ in $\dpc$ can be factored as $f=hg$ where $g$ is an epimorphism and $h$ is a monomorphism.
\end{Proposition}

\begin{proof}
Let $\cU=\{U_i\}_{i \in I}$ be the partition representing $f$.  Let $J \subset I$ be the set of indices $i$ for which $U_i$ meets $L'$.  Define a partition $\cU'=\{U'_i\}_{i \in I}$ of $L \amalg J$ by letting $U'_i$ be $(U_i \cap L) \cup \{i\}$ or $(U_i \cap L)$ depending on if $i \in J$ or not.  Define a partition $\cU''=\{U_i''\}_{i \in J}$ of $J \amalg L'$ by $U''_i=\{i\} \cup (U_i \cap L')$.  Then $\cU'$ defines an epimorphism $g \colon L \to J$ and $\cU''$ defines a monomorphism $h \colon J \to L'$ and $f=h$.
\end{proof}

\begin{Proposition}
\label{sym:map3}
Let $f \colon L \to L'$ be a morphism in $\dpc$ represented by a partition $\cU$ such that some part of $\cU$ contains more than one element of $L$.  Then there is a factorization $f=hg$ where $g$ is a proper non-monomorphism.
\end{Proposition}

\begin{proof}
Let $\cU=\{U_i\}_{i \in I}$ be the partition representing $f$.  Define a partition $\cU' = \{U'_i\}_{i \in I}$ of $L \amalg I$ by setting $U'_i = (U_i \cap L) \cup \{i\}$. Define a partition $\cU'' = \{U''_i\}_{i \in I}$ of $I \amalg L'$ by $U''_i = \{i\} \cup (U_i \cap L')$. Then $\cU'$ defines a proper non-monomorphism $g \colon L \to I$ and $\cU''$ defines a morphism $h \colon I \to L'$ and $f = hg$.
\end{proof}

\begin{Proposition}
\label{sym:map4}
Let $f \colon L \to L''$ be a non-monomorphism in $\dpc$.  Then $f$ can be factored as $f=hg$, where $g \colon L \to L'$ with $\# L'<\# L$.
\end{Proposition}

\begin{proof}
Let $\cU = \{U_i\}_{i \in I}$ be the partition representing $f$. There are two ways that $f$ could fail to be a monomorphism. 

In the first case, every part of each $U_i$ contains a unique element of $L$, but some part does not meet $L''$. Let $J \subset I$ be the set of indices $i$ for which $U_i$ meets $L''$. Set $L' = J$. Define a partition $\cU' = \{U'_i\}_{i \in I}$ of $L \amalg L'$ by $U'_i = (U_i \cap L) \cup (\{i\} \cap J)$. Define a partition $\cU'' = \{U''_i\}_{i \in J}$ of $L' \amalg L''$ by $U''_i = \{i\} \cup (U_i \cap L'')$. Then $\cU'$ defines a morphism $g \colon L \to L'$ and $\cU''$ defines a morphism $h \colon L' \to L''$ and $f = hg$. Note that $\#J < \#I = \#L$, so $g$ has the required property.

For the second case, some part of $U_i$ contains at least two elements of $L$. Set $L' = I$. Define a partition $\cU' = \{U'_i\}_{i \in I}$ of $L \amalg L'$ by $U'_i = (U_i \cap L) \cup \{i\}$ and define a partition $\cU'' = \{U''_i\}_{i \in I}$ of $L' \amalg L''$ by $U''_i = \{i\} \cup (U_i \cap L'')$. Then $\cU'$ defines a morphism $g \colon L \to L'$ and $\cU''$ defines a morphism $h \colon L' \to L''$ and $f = hg$. Note that $\#I < \#L$, so $g$ has the required property.
\end{proof}

\begin{Proposition}
\label{sym:map2}
Let $f$ and $f'$ be monomorphisms $L \to L'$.  Then either 
\begin{enumerate}[\rm (a)]
\item there exists an automorphism $\sigma$ of $L$ such that $f'=f \sigma$, in which case $f^tf'=\sigma$, or 
\item $f^tf'$ is a proper non-monomorphism.
\end{enumerate}
\end{Proposition}

\begin{proof}
Since both $f$ and $f'$ are monomorphisms, $f^t f'$ is proper. If (b) fails, then $f^t f'$ is a monomorphism. In particular, $\sigma = f^t f'$ is an automorphism of $L$. Note that $ff^t \colon L' \to L'$ is a partition $\{U_i\}_{i \in I}$ with the property that for each $U_i$, the intersection of $U_i$ with both copies of $L'$ are the same subset. This implies that $ff^tf' =f'$, i.e., $f' = f\sigma$, so (a) holds.
\end{proof}

\article[Minimal elements]
\label{sym:mindef}
Let $M$ be an object of $\Mod_{\dpc}$.  We say that $x \in M_L$ is {\bf minimal} if $f(x)=0$ for all $f \colon L \to L'$ which are non-monomorphisms.  We say that $x \in M_L$ is {\bf weakly minimal} if $f(x)=0$ for all $f \colon L \to L'$ which are proper non-monomorphisms.  We let $\Delta_L(M)$ denote the set of weakly minimal elements of $M_L$.  Obviously, $\Delta_L$ defines a left-exact functor $\Mod_{\dpc} \to \Vec$.

\begin{Proposition}
\label{sym:weakmin}
An element $x \in M_L$ is weakly minimal if and only if $f(x)=0$ for all non-invertible rectangular idempotent elements $f$ of the monoid $\End(L)$.
\end{Proposition}

\begin{proof}
If $x$ is weakly minimal and $f$ is a non-invertible rectangular idempotent then $f(x)=0$ since $f$ is proper and not a monomorphism.  Conversely, suppose $f(x)=0$ for all non-invertible rectangular idempotents.  Let $f \colon L \to L'$ be a proper non-monomorphism.  Write $f=hg$ with $g$ an epimorphism and $h$ a monomorphism.  Then $g$ is necessarily proper and not a monomorphism.  We thus have $(g^tg)(x)=0$.  Since $g^t$ is a monomorphism, this implies $g(x)=0$, and so $f(x)=0$.
\end{proof}

\begin{Proposition}
\label{sym:idempot}
Let $f \colon V \to W$ be a surjection of finite dimensional vector spaces.  Let $\{X_i\}_{i \in I}$ be a family of idempotents operating on both $V$ and $W$ such that $f(X_iv)=X_if(v)$.  Let $V'$ (resp.\ $W'$) be the set of elements of $V$ (resp.\ $W$) annihilated by all the $X_i$.  Then $f$ induces a surjection $V' \to W'$.
\end{Proposition}

\begin{proof}
It is clear that $f$ maps $V'$ into $W'$.  We must prove that it does so surjectively.  Let $\wt{R}$ be the non-commutative polynomial ring in the $X_i$ and let $\epsilon \colon \wt{R} \to \bC$ be the algebra homomorphism sending each $X_i$ to 0.  Let $R$ be the image of $\wt{R}$ in $\End(V)$, a finite dimensional algebra.  Both $V$ and $W$ are $R$-modules.

If $W'=0$ then there is nothing to prove, so suppose $W' \ne 0$.  For any $x \in \wt{R}$ and $w \in W'$, we have $xw=\epsilon(x)w$.  Since the left side only depends on the image of $x$ in $R$, this shows that $\epsilon$ factors through $R$.  Write $R=\bigoplus_{i=1}^n R_i$ where each $R_i$ cannot be further decomposed into a direct sum, and let $p_i \colon R \to R_i$ be the projection map.  Then $\epsilon$ factors through some $p_i$, say $p_1$.  Let $I_1=\rad(R_1)$.  Then $R_1/I_1$ is a simple algebra that admits an algebra homomorphism to $\bC$, and is therefore isomorphic to $\bC$.  We thus see that $I_1=\ker(\epsilon)$.  Since $\epsilon(X_i)=0$, we find that $p_1(X_1) \in I_1$.  Thus $p_1(X_i)$ is both nilpotent and idempotent, and therefore vanishes.  Since $R_1$ is generated as an algebra by the $p_1(X_i)$, it follows that $R_1=\bC$.

Let $e \in R$ be the central idempotent corresponding to $R_1$.  Suppose that $w \in W'$.  Let $v \in V$ be such that $f(v)=w$.  Since $X_ie=0$ for all $i$, we see that $X_iev=0$ for all $i$, i.e., $ev \in V'$.  As $\epsilon(e)=1$, we see that $ew=w$.  Thus $f(ev)=ef(v)=ew=w$, which shows that $w \in f(V')$.  This completes the proof.
\end{proof}

\begin{Proposition}
\label{sym:delta}
The functor $\Delta_L \colon \Mod^{\gfin}_{\dpc} \to \Vec$ is exact.
\end{Proposition}

\begin{proof}
This follows immediately from Propositions~\pref{sym:weakmin} and~\pref{sym:idempot}.
\end{proof}

\article[The object $K^n$]
\label{sym:Kndef}
Let $n$ be an integer and let $\bC^n$ be the vector space with basis $e_1, \ldots, e_n$.  For a finite set $L$, put $K^n_L=(\bC^n)^{\otimes L}$.  Given a morphism $L \to L'$ in $\dpc$ represented by a partition $\cU$, we define a morphism $f \colon K^n_L \to K^n_{L'}$ as follows.  Given a map $\alpha \colon L \to \un$, let $e_{\alpha}$ be the basis vector $\bigotimes_{i \in L} e_{\alpha(i)}$ of $K^n_L$.  Then $f(e_{\alpha})=0$ if $\alpha$ is non-constant on the partition $\cU \cap L$ of $L$.  If $\alpha$ is constant on this partition, let $\beta \colon L' \to \un$ be the unique function such that $\beta(x)=\alpha(y)$ if $y \in L$ is in the same part as $x \in L'$.  Then $f(e_{\alpha})=e_{\beta}$.  Thus $K^n$ defines an object of $\Mod^{\gfin}_{\dpc}$.  We note that $K^0_L$ is 0 if $L$ is not empty and $\bC$ if $L$ is empty.

\article[The relationship between $K^n$ and $\Delta_L$]
\label{sym:Kn}
Let $v$ be the element $e_1 \otimes \cdots \otimes e_n$ of $K^n_{\un}$.  Then $v$ is obviously weakly minimal, and thus belongs to $\Delta_{\un}(K^n)$.  Given any morphism $K^n \to M$, we therefore get an element of $\Delta_{\un}(M)$ by taking the image of $v$.  This defines a map of functors 
\[
\eps_n \colon \Hom_{\dpc}(K^n, -) \to \Delta_{\un}.
\]

\begin{proposition}
The map $\eps_n$ is an isomorphism, i.e., $(K^n, v)$ corepresents $\Delta_{\un}$.
\end{proposition}

\begin{proof}
Let $\alpha \colon L \to \un$ be a map of sets.  For $i \in \un$ let $U_i$ be the subset $\{i\} \cup \alpha^{-1}(i)$ of $\un \amalg L$.  Then $\{U_i\}_{i \in \un}$ defines a partition of $\un \amalg L$, and a map $f_{\alpha} \colon \un \to L$ in $\dpc$.  Let $e_{\alpha}$ be as above.  One easily verifies that $e_{\alpha}=f_{\alpha}(v)$, which shows that $v$ generates $K^n$.  It follows that $\eps_n$ is injective, for if $f \colon K^n \to M$ is a map with $f(v)=0$, then $f=0$.

We now show that $K^n$ is surjective.  Thus let $w \in M_{\un}$ be weakly minimal.  Define a map $\eta_L \colon K^n_L \to M_L$ by $\eta_L(e_{\alpha})=f_{\alpha}(w)$.  Clearly $\eta(v)=w$, so it suffices to show that $\eta$ defines a map $K^n \to M$ in $\Mod_{\dpc}$.  Let $g \colon L \to L'$ be a map in $\dpc$.  We must show that $g(\eta_L(e_{\alpha}))=\eta_{L'}(g(e_{\alpha}))$ for all $\alpha \colon L \to \un$.  We consider two cases:
\begin{itemize} 
\item \ul{Case 1}: $\alpha$ is not constant on the pieces of $g$.  Let $x$ and $y$ be in the same part of $g$ with $\alpha(x) \ne \alpha(y)$.  Then $g(e_{\alpha})=0$ by definition.  On the other hand, $x$ and $y$ belong to the same part of $gf_{\alpha}$, and so $g(f_{\alpha}(w))=0$ by \pref{sym:map3}.  Thus $g(\eta_L(e_{\alpha}))=\eta_{L'}(g(e_{\alpha}))=0$.
\item \ul{Case 2}: $\alpha$ is constant on the pieces of $g$.  Define $\beta \colon L' \to \un$ by $\beta(x)=\alpha(y)$ if $x$ and $y$ belong to the same piece of $g$.  Then $g(e_{\alpha})=e_{\beta}$, by definition.  A short computation shows that $gf_{\alpha}=f_{\beta}$, as morphisms of $\dpc$, and so $g(f_{\alpha}(w))=f_{\beta}(w)$.  Thus 
\begin{displaymath}
g(\eta_L(e_{\alpha}))=g(f_{\alpha}(w))=f_{\beta}(w)=\eta_{L'}(e_{\beta})=\eta_{L'}(g(e_{\alpha})).
\qedhere
\end{displaymath}
\end{itemize}
\end{proof}

\article \label{art:symKmn}
The above proof can be modified to get an analogous result for $\dpc \times \dpc$ which we will need later, so we record it here. Pick integers $m,n \ge 0$ and for $M \in \Mod_{\dpc \times \dpc}$, let $\Delta_{\ul{m}, \ul{n}}(M)$ be the subspace of elements $x \in M_{(\ul{m}, \ul{n})}$ which are weakly minimal with respect to both factors of $\dpc$, i.e., $(f,g)(x) = 0$ whenever $f$ or $g$ is a proper non-monomorphism. Let $p_i \colon \dpc \times \dpc \to \dpc$ be the projections for $i=1,2$ and set $K^{m,n} = p_1^* K^m \boxtimes p_2^* K^n$. Let $v \in K^{m,n}_{\ul{m},\ul{n}}$ be the element $(e_1 \otimes \cdots \otimes e_m) \otimes (e_1 \otimes \cdots \otimes e_n)$. Given any morphism $K^{m,n} \to M$, we get an element of $\Delta_{\ul{m}, \ul{n}}(M)$ by considering the image of $v$, and hence we get a map of functors
\[
\eps_{m,n} \colon \hom_{\dpc \times \dpc}(K^{m,n}, -) \to \Delta_{\ul{m},\ul{n}}.
\]

\begin{proposition}
The map $\eps_{m,n}$ is an isomorphism, i.e., $(p_1^* K^m \boxtimes p_2^* K^n, v)$ corepresents $\Delta_{\ul{m}, \ul{n}}$.
\end{proposition}

\begin{Proposition}
\label{sym:Knproj}
We have that $K^n$ is a projective object of $\Mod^{\gfin}_{\dpc}$.
\end{Proposition}

\begin{proof}
This follows from the above Propositions~\pref{sym:delta} and~\pref{sym:Kn}.
\end{proof}

\begin{Proposition}
\label{sym:mingen}
Suppose that $M \in \Mod_{\dpc}$ is generated by the minimal elements of $M_L$.  Then any subobject $N$ of $M_L$ is generated by the minimal elements of $N_L$.
\end{Proposition}

\begin{proof}
Let $x$ be an element of $N_{L'}$.  Then we can write $x=\sum_{i=1}^n f_i(y_i)$, where $y_i$ is a minimal element of $M_L$ and $f_i \colon L \to L'$ is a monomorphism.  Choose such an expression with $n$ minimal.  Let $i \ne j$.  Then $f_i$ and $f_j$ do not differ on the right by an automorphism of $L$, as otherwise we could find an expression with $n$ smaller.  It follows from Proposition~\pref{sym:map2} that $f_i^t f_j$ is not a monomorphism, and so $f_i^t(f_j(y_j))=0$ by minimality of $y_j$.  As $f_i^t f_i=\id_L$, we see that $y_i=f_i^t(x)$.  This shows that each $y_i$ belongs to $N_L$, which completes the proof.
\end{proof}

\begin{Corollary}
Let $M$ be as in Proposition~\pref{sym:mingen}.  Then the map
\begin{displaymath}
\{ \textrm{subobjects of $M$} \} \to \{ \textrm{subspaces of $M_L$} \}, \qquad N \mapsto N_L
\end{displaymath}
is injective.
\end{Corollary}

\begin{proof}
It suffices to show that $N \subset N'$ if and only if $N_L \subset N'_L$.  The ``only if'' direction is obvious.  Thus suppose that $N_L \subset N'_L$.  If $x \in N_{L'}$ then, by Proposition~\pref{sym:mingen}, we can write $x=\sum_{i=1}^n f_i(y_i)$ with $y_i \in N_L$ and $f_i \colon L \to L'$.  By assumption, each $y_i$ lies in $N'_L$, and thus $f_i(y_i)$ lies in $N'_{L'}$.  It follows that $x$ belongs to $N'_{L'}$, and so $N \subset N'$.
\end{proof}

\begin{Corollary}
\label{sym:finlen}
Let $M$ be as in Proposition~\pref{sym:mingen} and suppose that $M_L$ is finite dimensional.  Then $M$ has finite length.
\end{Corollary}

\begin{Proposition}
\label{sym:Knlen}
The object $K^n$ has finite length.
\end{Proposition}

\begin{proof}
The statement is obvious for $n=0$.  Assume now that $K^{n-1}$ is finite length.  Let $M$ be the subobject of $K^n$ generated by those $K^n_L$ with $\#L<n$.  Choose linear maps $f_i \colon \bC^{n-1} \to \bC^n$, for $1 \le i \le r$, such that for any $L$ with $\#L<n$, the map
\begin{displaymath}
\bigoplus_{i=1}^r (\bC^{n-1})^{\otimes L} \to (\bC^n)^{\otimes L}
\end{displaymath}
induced by the $f_i$ is surjective.  It follows that the map $f \colon (K^{n-1})^{\oplus r} \to K^n$ induced by the $f_i$ is surjective when evaluated on any set $L$ with $\#L <n$.  Thus the image of $f$ contains $M$.  Since $K^{n-1}$ is generated in degree $n-1$ (as shown in \pref{sym:Kn}), the image of $f$ is contained in $M$.  Thus $\im(f)=M$, and so $M$ is finite length by the inductive hypothesis.

Let $N=K^n/M$.  Then $N$ is generated by the element $v=e_1 \otimes \cdots \otimes e_n$, since $K^n$ is.  If $f \colon L \to L''$ is a non-monomorphism then we have a factorization $f=hg$, where $g \colon L \to L'$ with $\#L'<\#L$.  Since $N_{L'}=0$, we have $g(v)=0$ and so $f(v)=0$.  It follows that $v$ is minimal.  Thus the minimal elements of $N_L$ generate $N$.  Since $N_L$ is finite dimensional, $N$ is finite length by \pref{sym:finlen}.  Thus $K^n$ is an extension of two finite length objects, and therefore finite length.
\end{proof}

\begin{Lemma}
\label{sym:fgf}
A finite length object of $\Mod_{\dpc}$ is graded-finite.
\end{Lemma}

\begin{proof}
A finite length object is finitely generated, and every finitely generated object is graded finite since $\dpc$ is Hom-finite.
\end{proof}

\begin{Proposition}
\label{sym:finlen2}
An object of $\Mod_{\dpc}$ is of finite length if and only if it is a quotient of a finite direct sum of $K^n$'s.
\end{Proposition}

\begin{proof}
Since $K^n$ is finite length, so is any quotient of a finite direct sum of $K^n$'s.  Now suppose $M$ is a finite length object of $\Mod_{\dpc}$.  Then $M$ belongs to $\Mod_{\dpc}^{\gfin}$ by Lemma~\pref{sym:fgf}.  Let $n \ge 0$ be minimal with $M_{\un}$ non-zero.  Then any element of $M_{\un}$ is minimal by Proposition~\pref{sym:map4}, and thus in the image of a map from $K^n$ by Proposition~\pref{sym:Kn}.  It follows that there is a non-zero map $K^n \to M$.  The cokernel of this map is of smaller length than $M$, and thus, by induction, is a quotient of a sum of $K^n$'s.  Since $M$ belongs to $\Mod_{\dpc}^{\gfin}$ and the $K^n$'s are projective in this category, it follows that $M$ is a quotient of a finite direct sum of $K^n$'s.
\end{proof}

\article
\label{sym:ker}
For an object $L$ of $\dpc$, put $\cK_L=\bV^{\otimes L}$.  Extend $\cK$ to a functor on $\dpc$ as in \pref{sym:Kndef}.  In fact, $\cK$ is the direct limit of the $K^n$.  Clearly, $\cK$ defines an object of $\Rep(\fS)^{\dpc}$.  We thus have functors
\begin{displaymath}
\Phi \colon \Mod_{\dpc}^{\fin} \to \Mod_{\fS}, \qquad \Psi \colon \Mod_{\fS} \to \Mod_{\dpc}
\end{displaymath}
defined by the same formulas as in \pref{ker-func}.  Here we write $\Mod_{\fS}$ for the category of all $\fS(\infty)$-modules.  The following is the main theorem of this section:

\begin{Theorem}
\label{sym:dp}
The functors $\Phi$ and $\Psi$ induce mutually quasi-inverse contravariant equivalences of tensor categories between $\Rep(\fS)$ and $\Mod_{\dpc}^{\fin}$.
\end{Theorem}

\begin{proof}
Let $V_n=\bC[\fS/H_n]$. If $V$ is any representation of $\fS$ then $\Hom_{\fS}(V_n, V)=V^{H_n}$.  One easily verifies that the map $f \colon V_n \to T_n$ provided by $e_1 \otimes \cdots \otimes e_n \in T_n^{H_n}$ is injective, which shows that $V_n$ belongs to $\Rep(\fS)$.  A simple computation shows that $(\bV^{\otimes L})^{H_n}=(\bC^n)^{\otimes L}$, and so $\Psi(V_n)=K^n$.  Another simple computation shows that $\Phi(K^n)=\Delta_{\un}(\cK)$ is the subspace of $\bV^{\otimes n}$ spanned by tensors of the form $e_{\alpha}$, where $\alpha \colon \un \to \cB$ is an injection.  This space is the image of $f$, and so $\Phi(K^n)=V_n$.  Furthermore, one can verify that, with these identifications, the natural maps $V_n \to \Phi(\Psi(V_n))$ and $K^n \to \Psi(\Phi(K^n))$ are the identity maps.

If $V$ is an algebraic representation of $\fS$ then it is a quotient of a finite direct sum of $V_n$'s (pick a finite set of generators for $V$; they are all $H_n$-invariant for $n$ large), and so $\Psi(V)$ is a subobject of a finite direct sum of $K^n$'s, and therefore of finite length by Proposition~\pref{sym:Knlen}.  Similarly, if $M$ is a finite length object of $\Mod_{\dpc}$ then it is a quotient of a finite direct sum of $K^n$'s by Proposition~\pref{sym:finlen2}, and so $\Phi(M)$ is a subobject of a finite direct sum of $V_n$'s, and therefore algebraic.  Thus $\Phi$ and $\Psi$ induce functors between $\Rep(\fS)$ and $\Mod_{\dpc}^{\fin}$.

By Proposition~\pref{sym:Vninj}, $\bV^{\otimes L}$ is an injective object of $\Rep(\fS)$, and so $\Psi$ is exact on algebraic representations of $\fS$.  It follows that $\Phi \Psi$ is a left-exact functor from $\Rep(\fS)$ to itself and $\Psi \Phi$ is a right-exact functor from $\Mod_{\dpc}^{\fin}$ to itself.  Lemma~\pref{sym:dp-1} below, combined with what we have already shown, establishes that the natural maps $\id \to \Phi \Psi$ and $\id \to \Psi \Phi$ are isomorphisms (on $\Rep(\fS)$ and $\Mod_{\dpc}^{\fin}$). 

We now show that $\Phi$ is a tensor functor.  Proposition~\pref{ker-tens} does not literally apply, since $\dpc$ is not weakly directed, but we can follow the same plan.  In fact, looking at its proof, it is enough to show that if $M$ and $N$ are finite length objects of $\Mod_{\dpc}$ and $M'$ and $N'$ are arbitrary objects, then the natural map
\begin{displaymath}
\Hom_{\dpc}(M,M') \otimes \Hom_{\dpc}(N,N') \to \Hom_{\dpc \times \dpc}(p_1^*M \boxtimes p_2^*N, p_1^*M' \boxtimes p_2^*N')
\end{displaymath}
is an isomorphism.  Following the proof of Lemma~\pref{ker-tens-1}, and appealing to Proposition~\pref{sym:finlen2}, it suffices to treat the case where $M=K^m$ and $N=K^n$.  We then have $p_1^*M \boxtimes p_2^*N=K^{m,n}$ in the notation of \pref{art:symKmn}.  By Propositions~\pref{sym:Kn} and~\pref{art:symKmn}, it thus suffices to show that the natural map
\begin{displaymath}
\Delta_{\ul{m}}(M') \otimes \Delta_{\ul{n}}(N') \to \Delta_{\ul{m},\ul{n}}(M' \boxtimes N')
\end{displaymath}
is an isomorphism.  This, however, is straightforward and left to the reader.  Since $\Psi$ is the quasi-inverse of $\Phi$, it too is a tensor functor.
\end{proof}

\begin{Lemma}
\label{sym:dp-1}
Let $\cA$ be an abelian category, let $F \colon \cA \to \cA$ be a left (right) exact functor and let $\eta \colon \id \to F$ be a natural transformation.  Suppose that there is a class $\cQ$ of finite length objects of $\cA$ such that 
\begin{enumerate}[\rm (a)] 
\item $\eta(Q)$ is an isomorphism for $Q \in \cQ$; and 
\item every finite length object of $\cA$ is a sub (quotient) of an object in $\cQ$.  
\end{enumerate}
Then $\eta(M)$ is an isomorphism for all $M$ of finite length.
\end{Lemma}

\begin{proof}
It suffices to prove the ``left'' and ``sub'' form of the lemma.  Let $M$ be a finite length object of $\cA$.  We can find an exact sequence $0 \to M \to Q \to Q'$ with $Q$ and $Q'$ in $\cQ$.  We obtain a commutative diagram
\begin{displaymath}
\xymatrix{
0 \ar[r] & F(M) \ar[r] & F(Q) \ar[r] & F(Q') \\
0 \ar[r] & M \ar[u] \ar[r] & Q \ar[u] \ar[r] & Q' \ar[u] }
\end{displaymath}
where the vertical maps are $\eta$'s.  The result now follows from a simple diagram chase.
\end{proof}

\begin{Corollary}
The tensor categories $\Rep(\fS)$ and $\Mod_{\upc}^{\fin}$ are equivalent.
\end{Corollary}

\subsection{Symmetric Schur functors, universal property and specialization}
\label{sec:symgp-univ}

\article[Symmetric Schur functors]
\label{sym:univ}
For a tensor category $\cA$, let $T(\cA)$ be the category of tuples $(A, m, \Delta, \eta)$ where $A$ is an object of $\cA$, and $m \colon A \otimes A \to A$ and $\Delta \colon A \to A \otimes A$ and $\eta \colon A \to \bC$ are morphisms in $\cA$, such that:
\begin{itemize}
\item $m$ defines an associative commutative algebra structure on $A$.
\item $(\Delta, \eta)$ defines a counital coassociative cocommutative coalgebra structure on $A$.
\item $m\Delta=\id$ and $\Delta m=(m \otimes 1)(1 \otimes \Delta)$.
\end{itemize}
Let $A \in T(\cA)$.  Given a map $f \colon L \to L'$ in $\dpc$, we define a map $\phi \colon A^{\otimes L} \to A^{\otimes L'}$ as follows.  Suppose $f$ is represented by the partition $\cU=\{U_i \amalg U'_i\}_{i \in I}$, where $U_i \subset L$ and $U_i' \subset L'$.  We first define a map $\phi_i \colon A^{\otimes U_i} \to A^{\otimes U_i'}$ to be the composition of the multiplication map $A^{\otimes U_i} \to A$ with the map $A \to A^{\otimes U_i'}$, which is either the counit (if $U_i'$ is empty) or comultiplication (if not).  The map $\phi$ is then the tensor product of the $\phi_i$.  We define $\cK(A)$ to be the object of $\cA^{\dpc}$ given by $L \mapsto A^{\otimes L}$.  For an object $M$ of $\Mod_{\upc}^{\fin}$ and $A \in T(\cA)$, define
\begin{displaymath}
S_M(A)=M \otimes^{\dpc} \cK(A).
\end{displaymath}
Then $M \mapsto S_M(A)$ defines a covariant functor $\Mod_{\upc}^{\fin} \to \cA$ which is left-exact and a tensor functor (we use a modified version of Proposition~\pref{prop:tensorkernel}, see the proof of Theorem~\pref{sym:dp}). We call $S_M$ the {\bf symmetric Schur functor} associated to $M$.

\begin{remark}
Let $A$ be an object of $T(\cA)$ and let $A^{\otimes L} \to A^{\otimes L'}$ be a map built out of $m$, $\Delta$ and $\eta$.  One can represent such a composition graphically: multiplications are represented by convergences and comultiplications by divergences, and counits are terminal points.  For example, the map $A^{\otimes 4} \to A^{\otimes 2}$ given by $\Delta m \otimes \eta m$ is represented by the diagram
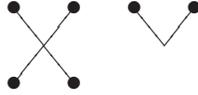
\begin{displaymath}
\begin{xy}
(0, 5)*{}="A"; (8, 5)*{}="B"; (16, 5)*{}="C"; (24, 5)*{}="D";
(4, 0)*{}="E"; (20, 0)*{}="H"; (0, -5)*{}="F"; (8, -5)*{}="G";
"A"*{\bullet}; "B"*{\bullet}; "C"*{\bullet}; "D"*{\bullet};
"F"*{\bullet}; "G"*{\bullet};
"A"; "E"; **\dir{-};
"B"; "E"; **\dir{-};
"E"; "F"; **\dir{-};
"E"; "G"; **\dir{-};
"C"; "H"; **\dir{-};
"D"; "H"; **\dir{-};
\end{xy}
\end{displaymath}
The connected components of this picture define a partition of $L \amalg L'$, and one can show that two compositions are equal if they define the same partition.  This is the reason $\cK(A)$ is well-defined.
\end{remark}

\begin{Example}
\label{sym:Tex}
Take $\cA=\Rep(\fS)$.  Then $\bV$ naturally has the structure of an object of $T(\cA)$.  The counit $\bV \to \bC$ is the augmentation map.  The comultiplication $\bV \to \bV \otimes \bV$ sends $e_i$ to $e_i \otimes e_i$, while the multiplication $\bV \otimes \bV \to \bV$ sends $e_i \otimes e_j$ to $\delta_{i,j} e_i$.  The object $\cK(\bV)$ of $\cA^{\dpc}$ is the object $\cK$ of \pref{sym:ker} and $M \mapsto S_M(\bV)$ is the equivalence $\Mod_{\upc}^{\fin} \to \Rep(\fS)$.
\end{Example}

\begin{Theorem}
To give a left-exact tensor functor from $\Rep(\fS)$ to an arbitrary tensor category $\cA$ is the same as to give an object of $T(\cA)$.  More precisely, letting $\bM$ be the object of $\Mod_{\upc}^{\fin}$ corresponding to $\bV$, the functors
\begin{displaymath}
\Phi_{\cA} \colon \LEx^{\otimes}(\Mod_{\upc}^{\fin}, \cA) \to T(\cA), \qquad F \mapsto F(\bM)
\end{displaymath}
and
\begin{displaymath}
\Psi_{\cA} \colon T(\cA) \to \LEx^{\otimes}(\Mod_{\upc}^{\fin}, \cA), \qquad A \mapsto (M \mapsto S_M(A))
\end{displaymath}
are mutually quasi-inverse equivalences.
\end{Theorem}

\begin{proof}
The proof is the same as that of Theorem~\pref{gl:univ}
\end{proof}

\begin{remark}
This theorem can be rephrased as follows: the functor $T \colon \TCat \to \Cat$ is corepresented by $\Rep(\fS)$, with the universal object in $T(\Rep(\fS))$ being $\bV$.  See \pref{art:schur-univ} for notation.
\end{remark}

\article[The specialization functor]
\label{sym:special}
The permutation representation $\bC^d$ of $\fS(d)$ defines an object of $T(\Rep(\fS(d)))$, just as in \pref{sym:Tex}.  We therefore obtain a left-exact tensor functor
\begin{displaymath}
\Gamma_d \colon \Rep(\fS) \to \Rep(\fS(d))
\end{displaymath}
from the universal property of $\Rep(\fS)$.  We call this functor the {\bf specialization functor}.  The kernel $\cK(\bC^d)$ associated to $\bC^d$ in \pref{sym:univ} is just the object $K^d$ (equipped with its $\fS(d)$ action) defined in \pref{sym:Kndef}.  If we identify $\Rep(\fS)^{\op}$ with $\Mod_{\dpc}^{\fin}$, then $\Gamma_d$ is just the functor $\Hom_{\dpc}(-, K^d)$.

\article[Specialization via invariants]
We now give a more direct description of specialization from the point of view of representation theory.

\begin{proposition}
We have a natural identification $\Gamma_d(V)=V^{H_d}$.
\end{proposition}

\begin{proof}
Under the equivalence $\Mod_{\dpc}^{\fin} \simeq \Rep(\fS)$, the module $K^d$ goes to the representation $V_d=\bC[\fS/H_d]$ (see the proof of Theorem~\pref{sym:dp}). So $\Gamma_d(V)=\Hom_{\fS}(V_d, V)=V^{H_d}$.
\end{proof}

\begin{remark}
Since $\Gamma_d$ is a tensor functor, the above result shows that if $V$ and $W$ are algebraic representations of $\fS$ then $(V \otimes W)^{H_d}=V^{H_d} \otimes W^{H_d}$.
\end{remark}

\article[Derived specialization]
As $\Rep(\fS)$ has enough injectives, the derived functor $\rR\Gamma_d$ of $\Gamma_d$ exists. To compute this on simple objects $V_\lambda$, we can use the injective resolutions $I_\lambda^\bullet$ that one gets from using \cite[Theorem 2.3.1]{symc1} and Theorem~\pref{sym:tca}. The calculation is carried out in \cite[\S 7.4]{symc1}, the result being:

\begin{theorem}
Let $\lambda$ be a partition and $d \ge 1$ an integer.  Either $\rR^i\Gamma_d(V_{\lambda})$  vanishes identically, or it is non-zero for a unique $i$, and then it is an irreducible representation of $\fS(d)$.
\end{theorem}

Precisely, $\rR^i\Gamma_d(V_{\lambda})$ is non-zero if and only if there exists a border strip $R \subset \lambda$ with the following properties: (i) $R$ is connected; (ii) $R$ contains the last box in the first row of $\lambda$; (iii) $R$ has height $i$; (iv) $\mu=\lambda\setminus R$ has $d$ boxes.  If such a border strip exists, then $\rR^i\Gamma_d(V_{\lambda})$ is the irreducible $V_{\mu}$ of $\fS(d)$.

\article
\label{sym:T0}
Let $T_0=T(\Vec^{\fin})$.  Given a finite set $L$, the vector space $A_L$ with basis $\{e_i\}_{i \in L}$ naturally has the structure of an object of $T_0$: multiplication is given by $m(e_i \otimes e_j)=\delta_{ij} e_i$, comultiplication by $\Delta(e_i)=e_i \otimes e_i$, and the counit is the augmentation map.  If $L \to L'$ is an injection of sets then the induced linear map $A_L \to A_{L'}$ is a morphism in $T_0$.  We thus obtain a functor $\us \to T_0$, where $\us$ is the upwards subset category of \pref{art:defnds}.  In fact:

\begin{proposition}
The functor $\us \to T_0$ is an equivalence of categories.
\end{proposition}

\begin{proof}
We first prove essential surjectivity.  Let $A$ be an object of $T_0$. For psychological reasons, it will be easier to work with $B=A^*$, the linear dual of $A$.  The space $B$ is a commutative associative unital ring, equipped with a map $\Delta \colon B \to B \otimes B$ which is cocommutative and coassociative and satisfies $m\Delta=\id$ and $\Delta m=(m \otimes 1)(1 \otimes \Delta)$.  From these identities, one can deduce that $\Delta$ respects multiplication, i.e., $\Delta(xy)=\Delta(x) \Delta(y)$; this is most easily seen by appealing to the remark in \pref{sym:univ}.  Note that $\Delta$ need not preserve the identity of $B$, however.

Being finite dimensional, the algebra $B$ is artinian.  We thus have a decomposition $B=\prod_{i \in L} B_i$ of $B$ into local rings, where $L$ is some finite index set.  Each factor corresponds to some minimal idempotent $e_i$.  The minimal idempotents of $B \otimes B$ are the elements $e_i \otimes e_j$, and a general idempotent of $B \otimes B$ can be written as a sum of these elements with coefficients 0 or 1.

Since $e_i$ is an idempotent of $B$ and $\Delta$ respects multiplication, $\Delta(e_i)$ is an idempotent of $B \otimes B$.  We can therefore write $\Delta(e_i)=\sum_{n,m} A^i_{n,m} e_n \otimes e_m$, where each $A^i_{n,m}$ is 0 or 1.  Applying the identity $\Delta m=(m \otimes 1)(1 \otimes \Delta)$ to $e_i \otimes e_i$, we find $A^i_{n,m}=\delta_{n,i} A^i_{n,m}$, which shows that $A^i_{n,m}=0$ if $n \ne i$.  By symmetry, we find $A^i_{n,m}=0$ if $m \ne i$, and so only $A^i_{i,i}$ can be non-zero.  The identity $m\Delta=\id$ shows that $A^i_{i,i}$ is indeed non-zero.  We have thus shown $\Delta(e_i)=e_i \otimes e_i$ for all $i$.

This identity implies that $\Delta$ maps $B_i$ into $B_i \otimes B_i$.  The multiplication map $m \colon B_i \otimes B_i \to B_i$ gives $B_i$ the structure of a $B_i \otimes B_i$ module.  Since $\Delta$ is a section of $m$, it follows that $B_i$ is a summand of $B_i \otimes B_i$, and thus projective as a $B_i \otimes B_i$ module.  Let $d=\dim(B_i)$.  Since $B_i \otimes B_i$ is local, any finite dimensional projective module over it is free, and thus has dimension divisible by $d^2=\dim(B_i \otimes B_i)$.  Thus $d^2 \mid d$, and so $d=1$.  This shows that $B_i=\bC e_i$ for each $i$, and so $B$ is spanned by the $e_i$.  Hence $A$ is isomorphic to $\bC[L]$, which completes the proof of essential surjectivity.

To complete the proof of the proposition, we must show that any map $f \colon A_L \to A_{L'}$ in $T_0$ comes from an injection $L \to L'$.  Note that the elements $e_i$ of $A_L$ are precisely the idempotents of $A_L$ which map to 1 under the counit.  It follows that $f$ has to map the set $\{e_i\}_{i \in L}$ to the set $\{e_i\}_{i \in L'}$.  Furthermore, this induced map must be injective, since the idempotents are orthogonal.  This completes the proof.
\end{proof}

\article \label{art:repfs-symc1}
The universal property of $\Rep(\fS)$ yields a left-exact tensor functor $\Rep(\fS) \to \Fun(T_0, \Vec^{\fin})$.  Identifying $T_0$ with $\us$ by Proposition~\pref{sym:T0}, and $\Fun(\us, \Vec)$ with $\Mod_A$, where $A = \Sym(\bC\langle 1 \rangle)$, this functor coincides with $S'$ from \pref{sym:Sfunc}.  It follows from what we have shown, and facts about $\Mod_A$ from \cite{symc1}, that this functor is fully faithful and its essential image consists of finitely generated saturated $A$-modules.  The analogues of the unproven statements in \pref{gl:catc} can be deduced in this case from the results in \cite[\S 4.2]{symc1}.

\section{Branching rules}
\label{sec:relations}
\setcounter{article}{0}
\renewcommand{\thearticle}{\thesection.\arabic{article}}%

\article[Classical groups]
\label{br:classical}
We now discuss some canonical functors between the classical representation categories that have been considered in this paper. They come in four different flavors, and calculating their effect on simple objects can be interpreted as calculating classical branching rules.  All of these functors are exact and respect tensor products.

\begin{itemize}
\item Diagonal embeddings $G \subset G \times G$ lead to tensor product functors 
\[
\otimes_G \colon \Rep(G) \times \Rep(G) \to \Rep(G)
\]
for $G \in \{\GL, \bO, \Sp\}$. We discuss this for $G=\bO$ (which is the same as $G=\Sp$ by Theorem~\pref{thm:ospequiv}) in \pref{o:tensor-prod}, and for $G=\GL$ in \pref{gl:tensor-prod}.  See \pref{br:example} for a specific example.

\item Dual to the tensor product functors, we have comultiplication functors 
\[
\Delta_G \colon \Rep(G) \to \Rep(G \times G)
\]
for $G \in \{\GL, \bO, \Sp\}$.  These can be defined by applying the appropriate universal property of $\Rep(G)$ to the object $\bV \boxtimes \bV$, where $\bV \in \Rep(G)$ is the vector representation, or by restricting along an appropriate embedding $G \times G \to G$ (which exists because $G$ is infinite).  These are discussed in \pref{br:comult}.

\item We have inclusions $G \subset \GL$ for $G \in \{\bO, \Sp\}$ which lead to restriction functors 
\[
\res_G \colon \Rep(\GL) \to \Rep(G).
\]
These are studied in \pref{br:restrict}.

\item Dual to the restriction functors, we have polarization functors
\[
{\rm pol}_G \colon \Rep(G) \to \Rep(\GL)
\]
for $G \in \{\bO, \Sp\}$.  These are defined by putting either a symmetric or skew-symmetric nondegenerate pairing on $\bV \oplus \bV_\ast$ for $\bV, \bV_\ast \in \Rep(\GL)$ and applying the relevant universal property, or by restricting along an appropriate embedding $\GL \to G$.  They are discussed in \pref{br:pol}.
\end{itemize}
The corresponding branching rules for all of these functors have been classically studied, and we refer to \cite{htw} for these rules. All of the rules follow the same pattern: they are sums over products of Littlewood--Richardson coefficients. We will deduce the rules for tensor product and restriction using some general formalism for diagram categories and some calculations with symmetric groups.

\article[Other groups]
The inclusion $\GA \subset \GL$ yields a restriction functor $\Rep(\GL) \to \Rep(\GA)$, where $\Rep(\GA)$ is the category of algebraic representations of $\GA(\infty)$.  This is likely an interesting construction to study, but we do not consider it here.  The inclusion also yields a restriction functor $\Rep^{\pol}(\GL) \to \Rep^{\pol}(\GA)$, which we have already studied:  the image consists of the injective objects of the target category (Proposition~\pref{art:ga-injectives}).  There are also maps from $\Rep^{\pol}(\GA)$ to the categories $\Rep(\GL)$, $\Rep(\bO)$, $\Rep(\Sp)$ and $\Rep(\fS)$ given by the natural objects admitting a trace map, namely, $\bV \otimes \bV_*$, $\Sym^2(\bV)$, $\lw^2(\bV)$ and $\bV$. 

The inclusion $\fS \subset \GL$ yields a restriction functor $\Rep(\GL) \to \Rep(\fS)$, which, as far as we know, has not been well-studied, likely because it is difficult to understand (see \pref{sym:inj}). The quadratic invariant on $\bV \in \Rep(\fS)$ gives us a left-exact tensor functor
\[
\res_{\fS} \colon \Rep(\bO) \to \Rep(\fS)
\]
using Theorem~\pref{c:univ}.  We are also unaware of any results on this functor.

\article \label{tens:decomp}
We begin with a general formula that makes the pushforward in \pref{diag:op} more explicit.  Let $\Lambda$ and $\Lambda'$ be categories as in \pref{cat-cond}, and let $F \colon \Lambda \to \Lambda'$ be a functor. Let $\Pi$ be the category of tuples $(x, y, f)$ where $x \in \Lambda$ and $y \in \Lambda'$, and $f$ is a morphism $f \colon F(x) \to y$.  A morphism $(x, y, f) \to (x', y', f')$ consists of a morphism $x \to x'$ in $\Lambda$ and a morphism $y \to y'$ in $\Lambda'$ such that the obvious diagram commutes.  We typically abbreviate the object $(x, y, f)$ by $f$ and write $\Pi_{x,y}$ for the objects of the form $(x,y,f)$.  We say that $f=(x,y,f)$ is {\bf irreducible} if any morphism $(x,y,f) \to (x',y,g)$ is an isomorphism.  We write $\Pi^{\irr}$ for the full subcategory of irreducible objects.  Note that for any $f \in \Pi_{x,y}$ we have maps from $\Aut(f)$ to $\Aut(x)$ and $\Aut(y)$.

\begin{Proposition}  \label{prop:tens:decomp}
Choose $x \in \Lambda$ and $y \in \Lambda'$, and let $V$ be a representation of $\Aut(x)$.  Let $\{f_i\}_{i \in I}$ be a complete irredundant set of irreducible objects of $\Pi_{x,y}$.  Then we have a natural isomorphism
\begin{displaymath}
(F_\#(S_x(V)))_y = \bigoplus_{i \in I} \Ind_{\Aut(f_i)}^{\Aut(y)} (V)
\end{displaymath}
as representations of $\Aut(y)$.  (Here $\Ind$ denotes the adjoint to restriction along the natural homomorphism $\Aut(f_i) \to \Aut(y)$.)
\end{Proposition}

\begin{proof}
Put $M=S_x(V)$.  Essentially by definition, we have
\begin{displaymath}
(F_\#(M))_y=\varinjlim_{(x',y,f) \in \Pi} F(M_{x'}).
\end{displaymath}
There is a natural map
\begin{displaymath}
\varinjlim_{f \in \Pi^{\irr}_{x,y}} F(V) \to \varinjlim_{(x',y,f) \in \Pi} F(M_{x'}).
\end{displaymath}
Since any map $f \to f'$ with $f \in \Pi^{\irr}_{x,y}$ and $f' \in \Pi_{x',y}$ is an isomorphism, the above map is injective.  We claim that it is surjective.  Suppose $f \in \Pi_{x',y}$.  We must show that the image of $F(M_{x'})$ corresponding to $f$ in the right direct limit comes from the left direct limit.  If $x'$ is not isomorphic to $x$, then $M_{x'}=0$ and there is nothing to show. Also, if $f$ is irreducible, there is nothing to show.  Thus suppose $x=x'$ and that $f$ is not irreducible, and choose a map $f \to f'$ with $f' \in \Pi_{x'',y}$ with $x$ not isomorphic to $x''$.  Then the image of $F(M_x)$ corresponding to $f$ in the right direct limit factors through the image of $F(M_{x''})$ corresponding to $f'$, and thus vanishes.  This establishes the claim.

We have thus shown that the natural map
\begin{displaymath}
\varinjlim_{f \in \Pi^{\irr}_{x,y}} F(V) \to (F_\#(M))_y
\end{displaymath}
is an isomorphism.  The result is now a simple calculation; note that $\Pi^{\irr}_{x,y}$ is a groupoid, so the direct limit is easy to calculate.
\end{proof}

\article[Tensor products (orthogonal group)]
\label{o:tensor-prod}
We now determine tensor product multiplicities in $\Rep(\bO)$, recovering \cite[Thm.~3.1]{koike}, by applying Proposition~\pref{prop:tens:decomp} on the functor $\amalg \colon \db \times \db \to \db$. Recall that we use $c$ to denote Littlewood--Richardson coefficients \pref{art:LR}.

\begin{proposition}
The multiplicity of $V_{\nu}$ in $V_{\lambda} \otimes V_{\mu}$ is
\begin{displaymath}
\sum_{\alpha, \beta, \gamma} c^{\lambda}_{\alpha, \beta} c^{\mu}_{\beta, \gamma} c^{\nu}_{\alpha, \gamma},
\end{displaymath}
where the sum is over all partitions $\alpha$, $\beta$ and $\gamma$.
\end{proposition}

\begin{proof}
The multiplicity in question is the same as the multiplicity of $M_{\nu}$ in $M_{\lambda} \otimes M_{\mu}$ taken in $\Mod_{\db}$, and this is what we compute.  Let $\ell=\vert \lambda \vert$, $m=\vert \mu \vert$ and $n=\vert \nu \vert$.  Let $L$ be a set of cardinality $\ell$ and choose $M$ and $N$ similarly but with $m$ and $n$.  Suppose that $\ell+m-n$ is a non-negative even number; otherwise there is no map $L \amalg M \to N$ in $\db$ and the multiplicity is 0.  Choose a set $E$ of cardinality $(\ell+m-n)/2$ and choose injections $E \to L$ and $E \to M$ and a bijection
\begin{displaymath}
(L \setminus E) \amalg (M \setminus E) \to N.
\end{displaymath}
The elements of $E$ can be regarded as the edges of a matching on $L \amalg M$, and so we have a map $f \colon L \amalg M \to N$ in $\db$.  This map is clearly irreducible (see \pref{tens:decomp} for the definition) and any irreducible map is isomorphic to $f$.  By Proposition~\pref{prop:tens:decomp}, we thus find that the evaluation of $M_{\lambda} \otimes M_{\mu}$ on $N$ is given by
\begin{equation}
\label{o:tens1}
\Ind_{\Aut(f)}^{\Aut(N)}(\bM_{\lambda} \boxtimes \bM_{\mu}).
\end{equation}

We have
\begin{displaymath}
\Aut(f)=\fS_{L \setminus E} \times \fS_{M \setminus E} \times \fS_E.
\end{displaymath}
The maps from $\Aut(f)$ to $\fS_L$, $\fS_M$ and $\fS_N$ are the obvious ones.  To compute the induction in \eqref{o:tens1}, we take invariants under the kernel of the map $\Aut(f) \to \Aut(N)$ and then form the more usual induction from the image to $\Aut(N)$.  The map $\Aut(f) \to \Aut(L) \times \Aut(M)$ factors as
\begin{equation}
\label{o:tens2}
\begin{split}
& \fS_{L \setminus E} \times \fS_{M \setminus E} \times \fS_E \\
\to & \fS_{L \setminus E} \times \fS_{M \setminus E} \times \fS_E \times \fS_E \\
\to & \fS_L \times \fS_M
\end{split}
\end{equation}
where the first map uses the diagonal map on $\fS_E$ and the second is a product of inclusions of Young subgroups.  When we restrict the representation $\bM_{\lambda} \boxtimes \bM_{\mu}$ of the group on the final line to the second line, we use the Littlewood--Richardson rule; when we further restrict to the first line, we take the tensor product of the two representations of $\fS_E$ that show up.  The final result is
\begin{displaymath}
\bigoplus c^{\lambda}_{\alpha, \beta} c^{\mu}_{\gamma, \delta} \bM_{\alpha} \boxtimes \bM_{\gamma} \boxtimes (\bM_{\beta} \otimes \bM_{\delta}).
\end{displaymath}
The sum is taken over all partitions $\alpha$, $\beta$, $\gamma$, $\delta$ with $\vert \beta \vert=\vert \delta \vert=\# E$.  The representations appear in the same order as the groups in the first line of \eqref{o:tens2}.  Now, the kernel of $\Aut(f) \to \Aut(N)$ is $\fS_E$.  We thus need to form the invariants of the above representation under this group.  This is easy:  it amounts to enforcing $\beta=\delta$ and discarding the final factor in the above equation.  We thus have
\begin{displaymath}
\bigoplus c^{\lambda}_{\alpha, \beta} c^{\mu}_{\gamma, \beta} \bM_{\alpha} \boxtimes \bM_{\gamma}.
\end{displaymath}
This representation is one of the group $\fS_{L \setminus E} \times \fS_{M \setminus E}$, which is identified with the image of $\Aut(f) \to \Aut(N)$.  We thus need to induce this representation to $\Aut(N)$.  This is accomplished using the Littlewood--Richardson rule.  The result is
\begin{displaymath}
\bigoplus c^{\lambda}_{\alpha, \beta} c^{\mu}_{\gamma, \beta} c^{\eta}_{\alpha, \gamma} \bM_{\eta}.
\end{displaymath}
This completes the computation of the induction \eqref{o:tens1}.  Taking the $\bM_{\nu}$ component gives the final answer.  Note that it is unnecessary to impose the condition $\vert \beta \vert=\# E$ in the final answer, as this is required for the product of Littlewood--Richardson coefficients to be non-zero.
\end{proof}

\article[Tensor products (general linear group)]
\label{gl:tensor-prod}
The tensor product multiplicities in $\Rep(\GL)$ are given in \cite[Thm.~2.4]{koike}. The proof is similar to the calculation for $\Rep(\bO)$, so we omit it, but we give the statement.

\begin{proposition}
The multiplicity of $V_{\nu,\nu'}$ in $V_{\lambda,\lambda'} \otimes V_{\mu,\mu'}$ is given by
\begin{displaymath}
\sum c^{\lambda}_{\alpha, \beta} c^{\lambda'}_{\alpha', \beta'} c^{\mu}_{\gamma, \beta'} c^{\mu'}_{\gamma', \beta} c^{\nu}_{\alpha, \gamma} c^{\nu'}_{\alpha', \gamma'}
\end{displaymath}
where the sum is over all partitions $\alpha$, $\alpha'$, $\beta$, $\beta'$, $\gamma$, $\gamma'$.
\end{proposition}

\article[Graphical representation of tensor product multiplicities]
\label{gl:tens-gr}
Suppose $\Gamma$ is an undirected graph whose edges are labeled by partitions; we write $\lambda(e)$ for the partition on edge $e$.  Given a function $\mu$ from the vertices of $\Gamma$ to partitions, we define
\begin{displaymath}
c_{\Gamma}(\mu)=\prod_{e=(i,j)} c^{\lambda(e)}_{\mu(i),\mu(j)},
\end{displaymath}
where the product is taken over the edges of $\Gamma$.  We define
\begin{displaymath}
c_{\Gamma}=\sum_{\mu} c_{\Gamma}(\mu)
\end{displaymath}
where the sum is over all functions $\mu$.
Let $\Gamma$ be the following graph:
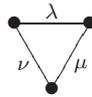
\begin{displaymath}
\begin{xy}
(5,2.88)*{}="A"; (0,-5.76)*{}="B"; (-5,2.88)*{}="C";
"A"*{\bullet}; "B"*{\bullet}; "C"*{\bullet};
"A"; "B"; **\dir{-}; "B"; "C"; **\dir{-}; "C"; "A"; **\dir{-};
(0, 4.58)*{\ss \lambda}; (3.8,-2.74)*{\ss \mu}; (-3.8,-2.74)*{\ss \nu};
\end{xy}
\end{displaymath}
Then Proposition~\pref{o:tensor-prod} can be rephrased as $[V_{\nu} : V_{\lambda} \otimes V_{\mu}]=c_{\Gamma}$.

Now let $\Gamma$ be the graph
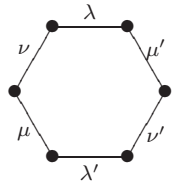
\begin{displaymath}
\begin{xy}
(5,8.66)*{}="A"; (10, 0)*{}="B"; (5,-8.66)*{}="C"; (-5,-8.66)*{}="D"; (-10,0)*{}="E"; (-5,8.66)*{}="F";
"A"*{\bullet}; "B"*{\bullet}; "C"*{\bullet}; "D"*{\bullet}; "E"*{\bullet}; "F"*{\bullet};
"A"; "B"; **\dir{-}; "B"; "C"; **\dir{-}; "C"; "D"; **\dir{-};
"D"; "E"; **\dir{-}; "E"; "F"; **\dir{-}; "F"; "A"; **\dir{-};
(0,10.66)*{\ss \lambda}; (8.8,5.63)*{\ss \mu'}; (8.8,-5.63)*{\ss \nu'};
(0,-10.66)*{\ss \lambda'}; (-8.8,-5.63)*{\ss \mu}; (-8.8,5.63)*{\ss \nu};
\end{xy}
\end{displaymath}
Then Proposition~\pref{gl:tensor-prod} can be rephrased as: $[V_{\nu,\nu'} : V_{\lambda,\lambda'} \otimes V_{\mu,\mu'}]=c_{\Gamma}$.  Note that the roles of $\lambda$, $\mu$ and $\nu$ are not the same: $\lambda$ and $\mu$ border edges labelled by primed variables, while $\nu$ does not.  On the other hand, the relationships of the edges are preserved by the rotation
\begin{displaymath}
(\mu, \nu, \lambda, \mu', \nu', \lambda') \to (\lambda', \mu, \nu, \lambda, \mu', \nu').
\end{displaymath}
That is, we have the following equality of multiplicities:
\begin{displaymath}
[V_{\nu,\nu'} : V_{\lambda,\lambda'} \otimes V_{\mu,\mu'}] = [V_{\mu,\mu'} : V_{\lambda', \lambda} \otimes V_{\nu, \nu'}].
\end{displaymath}
This is not unexpected, as it holds in the finite dimensional case.

\article[Comultiplication]
\label{br:comult}
We state the formulas for the functors $\Delta_G$, without proof. 

\begin{proposition}
We have the following:
  \begin{itemize}
  \item The multiplicity of $V_{\mu, \mu'} \boxtimes V_{\nu, \nu'}$ in $\Delta_\GL(V_{\lambda, \lambda'})$ is
\[
\sum c_{\mu, \nu}^\alpha c_{\mu', \nu'}^{\beta} c_{\alpha, \gamma}^\lambda c_{\beta, \gamma}^{\lambda'}
\]
where the sum is over all partitions $\alpha, \beta, \gamma$.
\item The multiplicity of $V_\mu \boxtimes V_\nu$ in $\Delta_\bO(V_\lambda)$ is
\[
\sum c^\alpha_{\mu, \nu} c^\lambda_{\alpha, 2 \beta}
\]
where the sum is over all partitions $\alpha ,\beta$.
\item The multiplicity of $V_\mu \boxtimes V_\nu$ in $\Delta_\Sp(V_\lambda)$ is
\[
\sum c^\alpha_{\mu, \nu} c^\lambda_{\alpha, (2 \beta)^\dagger}
\]
where the sum is over all partitions $\alpha, \beta$.
  \end{itemize}
\end{proposition}

\article[Restriction maps]
\label{br:restrict}
We now study the restriction functors
\begin{align*}
\res_{\bO} \colon \Rep(\GL) \to \Rep(\bO), \qquad \res_{\Sp} \colon \Rep(\GL) \to \Rep(\Sp).
\end{align*}
Consider the functor $F \colon \dwb \to \db$ which sends a biset $(L_+, L_-)$ to the set $L_+ \amalg L_-$ and which does the obvious thing to morphisms. Then $F_\# \colon \Mod_{\dwb}^{\rf} \to \Mod_{\db}^{\rf}$ becomes the restriction functor $\Rep(\GL) \to \Rep(\bO)$ via the contravariant equivalences of Theorem~\pref{gl:dwb} and Theorem~\pref{o:db}. There is a similar functor $\dwb \to \dsb$ and similar comments apply to its relation with $\Rep(\GL) \to \Rep(\Sp)$.

To get the branching rule from $\GL$ to $\bO$, we apply Proposition~\pref{prop:tens:decomp} to $F \colon \dwb \to \db$. Let $V_{\lambda, \lambda'}$ be simple object in $\Rep(\GL)$, and use $W_\mu$ to denote the simple objects in $\Rep(\bO)$.

\begin{proposition} \label{prop:classicalbranching}
The multiplicity of $W_\mu$ in $\res(V_{\lambda, \lambda'})$ is
\[
\sum c_{\alpha, \beta}^\mu c_{\alpha, 2\gamma}^\lambda c_{\beta, 2\delta}^{\lambda'}
\]
where the sum is over all partitions $\alpha, \beta, \gamma, \delta$. A similar formula holds for the restriction to $\Rep(\Sp)$ if we replace $2\gamma$ and $2\delta$ in the sum with $(2\gamma)^\dagger$ and $(2\delta)^\dagger$.
\end{proposition}

\begin{proof}
Let $\ell=\vert \lambda \vert$, $\ell'=\vert \lambda' \vert$, and $m=\vert \mu \vert$. Let $L$ be a biset with $\# L_+=\ell$ and $\# L_-=\ell'$ and let $M$ be a set of size $m$. We work directly in the language of $\Mod_{\dwb}^{\rf}$ and $\Mod_{\db}^{\rf}$. Then $V_{\lambda, \lambda'}$ becomes $S_L(\bM_\lambda \boxtimes \bM_{\lambda'})$. Suppose that
\begin{displaymath}
r=\ell+\ell'-m \equiv 0 \pmod 2.
\end{displaymath}
(If not then there is no map $F(L) \to M$ in $\db$ and the multiplicity is zero.)  Choose non-negative integers $i$ and $j$ such that $i+j=r$ and $i \le \ell$ and $j \le \ell'$.  Let $E_+$ and $E_-$ be sets of cardinality $i$ and $j$ and choose injections $E_+ \to L_+$ and $E_- \to L_-$, a bijection $(L_+ \setminus E_+) \amalg (L_- \setminus E_-) \to M$ and a perfect matching on $E = E_+ \amalg E_-$.

This data gives us a map $f = f_{i,j} \colon F(L) \to M$ in $\db$, and hence an element of $\Pi_{L,M}$ (see \pref{tens:decomp} for the definition). This map is irreducible if and only if there are no edges in the perfect matching of $E$ that go from $E_+$ to $E_-$ (otherwise, we could delete these edges to get a smaller biset $L'$ and there would be a morphism $(L,M,f) \to (L',M,f')$). In particular, both $i$ and $j$ are even. The isomorphism class of this map depends only on $i$ and $j$, and every irreducible map $F(L) \to M$ is isomorphic to one of this form.  By Proposition~\pref{prop:tens:decomp}, we thus find that the evaluation of $F_\#(S_L(\bM_{\lambda} \boxtimes \bM_{\lambda'}))$ on $M$ is given by summing
\begin{equation}
\label{eqn:globranch}
\Ind_{\Aut(f_{i,j})}^{\Aut(M)} ( \bM_{\lambda} \boxtimes \bM_{\lambda'})
\end{equation}
over all possible values of $i$ and $j$. We have
\begin{displaymath}
\Aut(f)=\fS_{L_+ \setminus E_+} \times (\fS_{i/2} \ltimes (\bZ/2)^{i/2}) \times \fS_{L_- \setminus E_-} \times (\fS_{j/2} \ltimes (\bZ/2)^{j/2})
\end{displaymath}
where the semidirect product groups are the automorphism groups of the perfect matching on $E_+$ and $E_-$. The maps from $\Aut(f)$ to $\Aut(L_+)$, $\Aut(L_-)$, $\Aut(E_+)$, and $\Aut(E_-)$ are the obvious things.  To compute the induction in \eqref{eqn:globranch}, we take invariants under the kernel of the map $\Aut(f) \to \Aut(M)$ and then form the usual induction from the image to $\Aut(M)$.  

To get the action of $\Aut(f)$ on $\bM_\lambda \boxtimes \bM_{\lambda'}$, we restrict along the obvious inclusion $\Aut(f) \to \Aut(L_+) \times \Aut(L_-)$. If we do this and take invariants under 
\[
\ker(\Aut(f) \to \Aut(M)) = (\fS_{i/2} \ltimes (\bZ/2)^{i/2}) \times (\fS_{j/2} \ltimes (\bZ/2)^{j/2}),
\]
then, noting that the restriction of $\bM_\eta$ from $\fS_{2n}$ to $\fS_n \ltimes (\bZ/2)^n$ contains an invariant if and only if all row lengths of $\eta$ are even \cite[Example 7.A2.9]{stanley}, we get
\[
\bigoplus c^{\lambda}_{\alpha, 2 \gamma} c^{\lambda'}_{\beta, 2 \delta} \bM_{\alpha} \boxtimes \bM_{\beta}.
\]
Here the sum is over all partitions subject to the condition $|\gamma| = i$ and $|\delta| = j$, and $\bM_\alpha \boxtimes \bM_\beta$ is a representation of $\fS_{L_+ \setminus E_+} \times \fS_{L_- \setminus E_-}$, which is the image of $\Aut(f) \to \Aut(M)$. Calculating the induction replaces $\bM_\alpha \boxtimes \bM_\beta$ with $\bigoplus c^\mu_{\alpha, \beta} \bM_\mu$. To obtain the final result, we sum over all possible values of $i$ and $j$; this simply amounts to dropping the conditions $\vert \beta \vert=i$ and $\vert \beta' \vert=j$.  Taking the $\bM_{\mu}$ component of the above gives the stated result.

The formula for $\res_{\Sp}$ can be deduced in a similar manner, or by applying orthogonal-symplectic duality (Theorem~\pref{thm:ospequiv}) to what we have already established.
\end{proof}

\article[Polarization] \label{br:pol}
Finally, we state the branching rules for the polarization functors 
\[
{\rm pol}_\bO \colon \Rep(\bO) \to \Rep(\GL), \qquad {\rm pol}_\Sp \colon \Rep(\Sp) \to \Rep(\GL).
\]
Let $W_\lambda$ denote either a simple object of $\Rep(\bO)$ or $\Rep(\Sp)$, and let $V_{\mu, \mu'}$ denote a simple object of $\Rep(\GL)$.

\begin{proposition}
The multiplicity of $V_{\mu, \mu'}$ in ${\rm pol}_{\bO}(W_\lambda)$ is
\[
\sum c^\lambda_{\alpha, (2 \beta)^\dagger} c^\alpha_{\mu, \mu'} 
\]
where the sum is over all partitions $\alpha, \beta$. Similarly, the multiplicity of $V_{\mu, \mu'}$ in ${\rm pol}_{\Sp}(W_\lambda)$ is
\[
\sum c^\lambda_{\alpha, 2 \beta} c^\alpha_{\mu, \mu'} .
\]
\end{proposition}

\article[Example]
\label{br:example}
We now give an example demonstrating how the stable theory can be applied to problems at finite level.  For notational ease, we let $U_{\lambda}$ be the irreducible $V^4_{\lambda}$ of $\Sp(4)$, when $\lambda$ has at most two parts (see \pref{o:weylfin}).  Suppose we want to compute the tensor product of $U_{(2,1)}$ and $U_{(1,1)}$.  We first compute the stable tensor product using Proposition~\pref{o:tensor-prod}:
\begin{displaymath}
V_{(2,1)} \cdot V_{(1,1)} = V_{(3,2)} + V_{(2,2,1)} + V_{(3,1,1)} + V_{(2,1,1,1)} + V_{(3)} + V_{(1,1,1)} + 2 V_{(2,1)} + V_{(1)}.
\end{displaymath}
This equality holds in $\rK(\Rep(\Sp))$.  We now apply $\rR \Gamma_4$, which is a ring homomorphism from $\rK(\Rep(\Sp))$ to $\rK(\Rep(\Sp(4)))$.  By Theorem~\pref{o:dersp}, the complex $\rR \Gamma_4(V_{\lambda})$ has at most one non-zero cohomology group, and it can be computed using a certain combinatorial rule.  If $\lambda$ has at most two parts then $\Gamma_4(V_{\lambda})=U_{\lambda}$, and the higher derived functors vanish.  When $\lambda$ has exactly three parts, the complex $\rR\Gamma_4(V_{\lambda})$ is acyclic.  Finally, we have $\rR^1\Gamma_4(V_{(2,1,1,1)})=U_{(2,1)}$.  So
\begin{displaymath}
U_{(2,1)} \cdot U_{(1,1)} = U_{(3,2)} - U_{(2,1)} + U_{(3)} + 2 U_{(2,1)} + U_{(1)}.
\end{displaymath}
Since $\Rep(\Sp(4))$ is semi-simple, the above equality in $\rK$-theory gives an actual decomposition of representations
\begin{displaymath}
U_{(2,1)} \otimes U_{(1,1)} = U_{(3,2)} \oplus U_{(3)} \oplus U_{(2,1)} \oplus U_{(1)}.
\end{displaymath}
The same principle can be used to compute any of the other branching rules at finite level.

\section{Questions and problems}
\label{sec:ques}
\setcounter{article}{0}
\renewcommand{\thearticle}{\thesection.\arabic{article}}%

\article[General theory]
The biggest question raised by this paper is whether the material can be treated in a uniform manner.  We do not know the answer, but offer one observation.  Let $E_1, \ldots, E_n$ be infinite dimensional vector spaces, let $G_i=\GL(E_i)$, let $G=G_1 \times \cdots \times G_n$ and let $V$ be a polynomial representation of $G$ such that $V^*$ has an ``approximate'' open dense orbit: we mean that this representation and group is a union of finite-dimensional spaces and groups, respectively, and we ask that each finite-dimensional space has an open dense orbit. 

Let $H$ be the generic stabilizer of $V^*$.  By definition, $H$ is a subgroup of $G$, and so each $E_i$ is a representation of $H$.  Call a representation of $H$ {\bf algebraic} if it is a subquotient of a finite direct sum of tensor products of $E_i$'s, and let $\Rep(H)$ be the category of algebraic representations.  Let $A=\Sym(V)$, so that $\Spec(A)=V^*$, and let $\Mod_K$ be the Serre quotient of $\Mod_A^{\fg}$ by the subcategory of modules with proper support.  There is a functor $\Mod_K \to \Rep(H)$ given by taking the fiber at a generic point.  In certain cases, this is an equivalence, and $\Mod_K$ is equivalent to $\Mod_A^{\fin}$, and we recover some of the results of this paper.  These cases are:
\begin{itemize}
\item $n=1$, $V=E$, $H=\GA(\infty)$, $\Rep(H)=\Rep^{\pol}(\GA)$.
\item $n=1$, $V=\Sym^2(E)$, $H=\bO(\infty)$, $\Rep(H)=\Rep(\bO)$.
\item $n=1$, $V=\lw^2(E)$, $H=\Sp(\infty)$, $\Rep(H)=\Rep(\Sp)$.
\item $n=2$, $V=E_1 \otimes E_2$, $H=\GL(\infty)$, $\Rep(H)=\Rep(\GL)$.
\end{itemize}
Unfortunately, we do not know how to prove the claims we have just made except by invoking the results of this paper, and so we are unable to use this set-up to develop the theory.  Also, we are unaware of how the symmetric group fits in to this picture.

\article[Plethysms]
Let $T$ be as in \pref{o:schurfun} and let $\cA$ be a tensor category.  If $(A, \omega) \in T(A)$ and $M$ is an object of $\Rep(\bO)$, then we obtain an object $S_M(A)$ of $\cA$ from the universal property of $\Rep(\bO)$.  In fact, $\omega$ induces a symmetric form on $S_M(A)$, and so if $N$ is a second object of $\Rep(\bO)$ then one can make sense of $S_N(S_M(A))$.  Can one functorially define an object $N \circ M$ of $\Rep(\bO)$ such that $S_{N \circ M}(A)$ is naturally isomorphic to $S_N(S_M(A))$? The natural guess is to put a symmetric form on $M$ and define $N \circ M$ to be $S_N(M)$.  However, there is not a canonical choice of form, so this construction is likely not functorial.  

\article[Positive characteristic]
An interesting (but difficult) problem is to extend the results of this paper to positive characteristic.  One encounters difficulties from the very beginning:  the polynomial theory of $\GL(\infty)$ is not semi-simple in positive characteristic, and the equivalence between it and representations of symmetric groups breaks down.

\article[Algebraic representations of $\GA$]
What is the structure of $\Rep(\GA)$, the category of all algebraic representations of the infinite general affine group?  (We only determined the structure of the category of polynomial representations.)

\article[Tannakian duality]
If $G$ is an algebraic group then $\Rep(G)$ is a rigid tensor category admitting a fiber functor (an exact faithful functor to $\Vec^{\fin}$).  Tannakian duality provides a converse to this statement, when certain mild hypotheses are satisfied.  The categories we have considered, such as $\Rep(\GL)$, are not rigid tensor categories, since they lack a good notion of duality. However, the forgetful functor to $\Vec$ could reasonably be called a fiber functor.  Is there any sort of Tannakian formalism for this class of categories?

\article[Maps between derived categories]
It is desirable to have a better understanding of the functors (both tensor and not) between the derived categories of the representation categories we have been considering.  Some specific questions:
\begin{enumerate}
\item What are the universal properties of $\rD^b(\Rep(\GL))$, $\rD^b(\Rep(\bO))$, etc.?
\item What are the auto-equivalences groups of $\rD^b(\Rep(\GL))$, etc.?
\item The categories $\Rep^{\pol}(\GA)$ and $\Rep(\fS)$ are equivalent as abelian categories, not equivalent as tensor categories, but have isomorphic Grothendieck rings.  Are their derived categories equivalent, as triangulated tensor categories?  We suspect not, but do not have a proof.
\end{enumerate}

\article[A degeneration]
\label{ques:degen}
As we have mentioned in \pref{art:intro-degeneration}, the structure constants for multiplication in $\rK(\Rep^{\pol}(\GA))$ are the Littlewood--Richardson coefficients $c^{\lambda}_{\mu,\nu}$, while those for $\rK(\Rep(\fS))$ are the stable Kronecker coefficients $\ol{g}^{\lambda}_{\mu,\nu}$.  The stable Kronecker coefficient $\ol{g}^{\lambda}_{\mu, \nu}$ is nonzero only if $|\lambda| \le |\mu| + |\nu|$, and, in case of equality, it is the Littlewood--Richardson coefficient $c^{\lambda}_{\mu, \nu}$.  It follows that the ring $\rK(\Rep(\fS))$ can be naturally filtered in such a way that the associated graded is $\rK(\Rep^{\pol}(\GA))$.  The Rees algebra construction provides a flat $\bC[t]$-algebra with generic fiber $\rK(\Rep(\fS))$ and special fiber $\rK(\Rep^{\pol}(\GA))$.  Is there a corresponding categorical construction?  That is, does there exist a reasonable family of categories over the affine line $\bA^1$ with generic fiber $\Rep(\fS)$ and special fiber $\Rep^{\pol}(\GA)$?

\article[Structure of functor categories]
The category $\cC$ defined in \pref{gl:catc} and its analogues are quite interesting and deserve study.  We plan to return to this topic in a future paper.


\begin{thebibliography}{HTW}

\bibitem[BC$^+$]{walledbrauer} 
Georgia Benkart, Manish Chakrabarti, Thomas Halverson, Robert Leduc, Chanyoung Lee, Jeffrey Stroomer, 
Tensor product representations of general linear groups and their connections with Brauer algebras,
{\it J. Algebra} {\bf 166} (1994), no.~3, 529--567. 

\bibitem[BR]{bereleregev}
A.~Berele, A.~Regev, 
Hook Young diagrams with applications to combinatorics and to representations of Lie superalgebras,
{\it Adv. Math.} {\bf 64} (1987), no.~2, 118--175. 

\bibitem[BS]{boijsoderberg} Mats Boij, Jonas S\"oderberg, Graded Betti numbers of Cohen-Macaulay modules and the multiplicity conjecture, {\it J. Lond. Math. Soc. (2)} {\bf 78} (2008), no.~1, 85--106, \arxiv{math/0611081v2}.

\bibitem[Bra]{brauer}
Richard Brauer, 
On algebras which are connected with the semisimple continuous groups, 
{\it Ann. of Math. (2)} {\bf 38} (1937), no.~4, 857--872. 

\bibitem[Bru]{brundan}  
Jonathan Brundan, 
Kazhdan-Lusztig polynomials and character formulae for the Lie superalgebra $\fgl(m|n)$,
{\it J. Amer. Math. Soc.} {\bf 16} (2003), no.~1, 185--231,
\arxiv{math/0203011v3}.

\bibitem[BS1]{BS:III} Jonathan Brundan, Catharina Stroppel, Highest weight categories arising from Khovanov's diagram algebra III: category $\cO$, {\it Represent. Theory} {\bf 15} (2011), 170--243, \arxiv{0812.1090v3}.

\bibitem[BS2]{BS:IV} Jonathan Brundan, Catharina Stroppel, Highest weight categories arising from Khovanov's diagram algebra IV: the general linear supergroup, {\it J. Eur. Math. Soc.} {\bf 14} (2012), 373--419, \arxiv{0907.2543v2}.

\bibitem[BS3]{BS:grading} Jonathan Brundan, Catharina Stroppel, Gradings on walled Brauer algebras and Khovanov's arc algebra, {\it Adv. Math.} {\bf 231} (2012), 709--773, \arxiv{1107.0999v1}.

\bibitem[CL]{chenglam} Shun-Jen Cheng, Ngau Lam,
Irreducible characters of general linear superalgebra and super duality,
{\it Comm. Math. Phys.} {\bf 298} (2010), no.~3, 645--672,  
\arxiv{0905.0332v2}.

\bibitem[CLW]{chenglamwang} Shun-Jen Cheng, Ngau Lam, Weiqiang Wang,
Super duality and irreducible characters of ortho-symplectic Lie superalgebras,
{\it Invent. Math.} {\bf 183} (2011), no.~1, 189--224,
\arxiv{0911.0129v2}.

\bibitem[CEF]{fimodules} Thomas Church, Jordan Ellenberg, Benson Farb, FI-modules and stability for representations of symmetric groups, {\it Duke Math. J.}, to appear, \arxiv{1204.4533v3}.

\bibitem[CF]{churchfarb}
Thomas Church, Benson Farb, 
Representation theory and homological stability, {\it Adv. Math.} {\bf 245} (2013), 250--314, \arxiv{1008.1368v3}.

\bibitem[CW]{comeswilson}
Jonathan Comes, Benjamin Wilson,
Deligne's category ${\rm \ul{Re}p}(GL_\delta)$ and representations of general linear supergroups, 
{\it Represent. Theory} {\bf 16} (2012), 568--609, \arxiv{1108.0652v1}.

\bibitem[DPS]{koszulcategory}
Elizabeth Dan-Cohen, Ivan Penkov, Vera Serganova, 
A Koszul category of representations of finitary Lie algebras,
\arxiv{1105.3407v2}.

\bibitem[De1]{DeligneTannak}
P.~Deligne, Cat\'egories tannakiennes, {\it The Grothendieck Festschrift}, Vol.~II, 111--195, Progr.\ Math.\ {\bf 87}, Birkh\"auser Boston, Boston, MA, 1990. 

\bibitem[De2]{deligne} 
P.~Deligne, La cat\'egorie des repr\'esentations du groupe sym\'etrique $S_t$, lorsque $t$ n'est pas un entier naturel, 
{\it Algebraic groups and homogeneous spaces}, 209--273, Tata Inst. Fund. Res. Stud. Math., Tata Inst. Fund. Res., Mumbai, 2007. 

\bibitem[EFW]{efw} David Eisenbud, Gunnar Fl\o ystad, Jerzy Weyman, The existence of equivariant pure free resolutions, {\it Ann. Inst. Fourier (Grenoble)} {\bf 61} (2011), no.~3, 905--926, \arxiv{0709.1529v5}.

\bibitem[ES]{es:bs} David Eisenbud, Frank-Olaf Schreyer, Betti numbers of graded modules and cohomology of vector bundles, {\it J. Amer. Math. Soc.} {\bf 22} (2009), no.~3, 859--888, \arxiv{0712.1843v3}.

\bibitem[FH]{fultonharris} 
William Fulton, Joe Harris, 
{\it Representation Theory: A First Course}, 
Graduate Texts in Mathematics {\bf 129}, Springer-Verlag, New York, 1991.

\bibitem[GW]{goodmanwallach} 
Roe Goodman, Nolan R. Wallach, 
{\it Symmetry, Representations, and Invariants}, 
Graduate Texts in Mathematics {\bf 255}, Springer, 2009.

\bibitem[HR]{halversonram} 
Tom Halverson, Arun Ram, 
Partition algebras, 
{\it European J.\ Combin.} {\bf 26} (2005), no.~6, 869--921, 
\arxiv{math/0401314v2}.

\bibitem[HH]{hh} 
Mitsuyasu Hashimoto, Takahiro Hayashi, 
Quantum multilinear algebra, 
{\it Tohoku Math.\ J.\ (2)} {\bf 44} (1992), no.~4, 471--521. 

\bibitem[HTW]{htw} Roger Howe, Eng-Chye Tan, Jeb F. Willenbring, Stable branching rules for classical symmetric pairs, {\it Trans. Amer. Math. Soc.} {\bf 357} (2005), no.~4, 1601--1626, \arxiv{math/0311159v2}.

\bibitem[Jon]{jones} 
V.~F.~R. Jones, 
The Potts model and the symmetric group, 
{\it Subfactors (Kyuzeso, 1993)}, 259--267, World Sci. Publ., River Edge, NJ, 1994. 

\bibitem[Kin]{king} R.~C.~King, Modification rules and products of irreducible representations of the unitary, orthogonal, and symplectic groups, {\it J. Mathematical Phys.} {\bf 12} (1971), 1588--1598. 

\bibitem[Koi]{koike} 
Kazuhiko Koike, 
On the decomposition of tensor products of the representations of the classical groups: by means of the universal characters, 
{\it Adv. Math.} {\bf 74} (1989), no.~1, 57--86.

\bibitem[KT]{koiketerada} 
Kazuhiko Koike, Itaru Terada, Young-diagrammatic methods for the representation theory of the classical groups of type $B_n$, $C_n$, $D_n$, 
{\it J. Algebra} {\bf 107} (1987), no.~2, 466--511.

\bibitem[Lit1]{littlewood} 
Dudley~E. Littlewood, 
{\it The Theory of Group Characters and Matrix Representations of Groups}, 
reprint of the second (1950) edition, AMS Chelsea Publishing, Providence, RI, 2006.

\bibitem[Lit2]{littlewood2} D.~E. Littlewood, Products and plethysms of characters with orthogonal, symplectic and symmetric groups, {\it Canad. J. Math.} {\bf 10} (1958), 17--32.

\bibitem[Mac]{macdonald} I.~G.~Macdonald, {\it Symmetric Functions and Hall Polynomials}, second edition, Oxford Mathematical Monographs, Oxford, 1995.

\bibitem[Mar]{martin} Paul Martin, Temperley-Lieb algebras for nonplanar statistical mechanics--the partition algebra construction, {\it J. Knot Theory Ramifications} {\bf 3} (1994), no.~1, 51--82. 

\bibitem[NSS]{deg2tca} Rohit Nagpal, Steven V Sam, Andrew Snowden, Noetherianity of some degree two twisted commutative algebras, \arxiv{1501.06925v1}.

\bibitem[OV]{ov} Andrei Okounkov, Anatoly Vershik, A new approach to representation theory of symmetric groups, {\it Selecta Math. (N.S.)} {\bf 2} (1996), no.~4, 581--605, \arxiv{math/0503040v3}.

\bibitem[Ols]{olshanskii} G.~I. Ol$'$shanski{\u \i}, Representations of infinite-dimensional classical groups, limits of enveloping algebras, and Yangians, {\it Topics in representation theory}, 1--66, Adv. Soviet Math., 2, Amer. Math. Soc., Providence, RI, 1991.

\bibitem[OR]{ottavianirubei} Giorgio Ottaviani, Elena Rubei, Quivers and the cohomology of homogeneous vector bundles, {\it Duke Math. J.} {\bf 132} (2006), no.~3, 459--508, \arxiv{math/0403307v2}.

\bibitem[PSe]{penkovserganova} Ivan Penkov, Vera Serganova, Categories of integrable $sl(\infty)$-, $o(\infty)$-, $sp(\infty)$-modules, {\it Representation Theory and Mathematical Physics}, Contemp. Math. {\bf 557}, AMS 2011, pp. 335--357, \arxiv{1006.2749v1}.

\bibitem[PSt]{penkovstyrkas} Ivan Penkov, Konstantin Styrkas, Tensor representations of classical locally finite Lie algebras, {\it Developments and trends in infinite-dimensional Lie theory}, Progr. Math. {\bf 288}, Birkh\"auser Boston, Inc., Boston, MA, 2011, pp. 127--150, \arxiv{0709.1525v1}.

\bibitem[Pro]{proctor} Robert~A. Proctor, 
Odd symplectic groups,
{\it Invent. Math.} {\bf 92} (1988), no.~2, 307--332.

\bibitem[SS1]{symc1} Steven~V Sam, Andrew Snowden, GL-equivariant modules over polynomial rings in infinitely many variables, {\it Trans. Amer. Math. Soc.}, to appear, \arxiv{1206.2233v2}.

\bibitem[SS2]{expos} Steven~V Sam, Andrew Snowden, Introduction to twisted commutative algebras, \arxiv{1209.5122v1}.

\bibitem[SS3]{sskoszul} Steven~V Sam, Andrew Snowden, GL-equivariant modules over polynomial rings in infinitely many variables II; in preparation.

\bibitem[SS4]{spincat} Steven~V Sam, Andrew Snowden, Infinite rank spinor and oscillator representations; in preparation.

\bibitem[SSW]{ssw} Steven~V Sam, Andrew Snowden, Jerzy Weyman, Homology of Littlewood complexes, {\it Selecta Math. (N.S.)} {\bf 19} (2013), no.~3, 655--698, \arxiv{1209.3509v2}.

\bibitem[Se1]{serganova} Vera Serganova, Kazhdan-Lusztig polynomials and character formula for the Lie superalgebra $\fgl(m|n)$,
{\it Selecta Math. (N.S.)} {\bf 2} (1996), no.~4, 607--651.

\bibitem[Se2]{serganova-inf} Vera Serganova, Classical Lie superalgebras at infinity, {\it Advances in Lie superalgebras}, 181--201, Springer INdAM Ser. {\bf 7}, Springer, Cham, 2014.

\bibitem[Sv]{sergeev} A.~N. Sergeev, The tensor algebra of the identity representation as a module over the Lie superalgebras $\mathfrak{Gl}(n,m)$ and $Q(n)$,  
{\it Math. USSR Sbornik} {\bf 51} (1985), 419--427.

\bibitem[Sno]{snowden} Andrew Snowden, Syzygies of Segre embeddings and $\Delta$-modules, {\it Duke Math. J.} {\bf 162}, no.~2 (2013), 225--277, \arxiv{1006.5248v4}.

\bibitem[Sta]{stanley} Richard~P. Stanley, {\it Enumerative Combinatorics}, Vol. 2, Cambridge Studies in Advanced Mathematics {\bf 62}, Cambridge University Press, Cambridge, 1999.

\bibitem[Wen]{wenzl} Hans Wenzl, On the structure of Brauer's centralizer algebras, {\it Ann. of Math. (2)} {\bf 128} (1988), no.~1, 173--193.

\bibitem[Wyl]{weyl} Hermann Weyl, {\it The Classical Groups: their invariants and representations}, Fifteenth printing, Princeton Landmarks in Mathematics, Princeton University Press, Princeton, NJ, 1997.

\end{thebibliography}
\end{document}